\documentclass[11pt]{amsart}

\usepackage[a4paper,  left=25mm, right=25mm, top=25mm, bottom=24mm, marginparwidth= 20mm]{geometry}

\usepackage{amscd}
\usepackage{amsfonts}
\usepackage{amsmath}
\usepackage{amssymb}
\usepackage{amsthm}
\usepackage{blkarray}
\usepackage{bm}
\usepackage{bbm}
\usepackage{comment}
\usepackage{epsfig}
\usepackage{enumerate}
\usepackage{graphicx}
\usepackage{mathrsfs}
\usepackage{multicol}
\usepackage{multirow}
\usepackage{nccmath}
\usepackage{rotating}
\usepackage{smfthm}
\usepackage{stmaryrd}
\usepackage[normalem]{ulem}
\usepackage[dvipsnames]{xcolor}
\usepackage[all]{xy}
\usepackage[colorlinks=true,linkcolor=NavyBlue,urlcolor=RoyalBlue,citecolor=PineGreen,bookmarks=true,bookmarksopen=true,bookmarksopenlevel=2,unicode=true,linktocpage]{hyperref}

\newtheorem{theorem}{Theorem}[section]
\newtheorem{definition}[theorem]{Definition}
\newtheorem{remark}[theorem]{Remark}
\newtheorem{example}[theorem]{Example}
\newtheorem{proposition}[theorem]{Proposition}
\newtheorem{lemma}[theorem]{Lemma}
\newtheorem{corollary}[theorem]{Corollary}

\usepackage[]{todonotes}

\newcommand{\C}{\mathbb{C}}

\newcommand{\U}{\mathrm{U}}
\renewcommand{\1}{\mathbbm{1}}
\newcommand{\A}{\mathcal A}

\newcommand{\cA}{{\mathscr A}}

\newcommand{\B}{\mathcal{B}}

\newcommand{\bH}{\mathbb{H}}
\newcommand{\build}[3]{\mathrel{\mathop{\kern 0pt#1}\limits_{#2}^{#3}}}
\newcommand{\cC}{\mathcal{C}}

\renewcommand{\d}{{\; {\rm d}}}
\newcommand{\nsd}{{\rm d}}

\newcommand{\D}{\mathscr{D}}

\newcommand{\DLP}{{\sf DLP}}
\newcommand{\DLSF}{{\sf DLSF}}
\newcommand{\WDLSF}{{\sf WDLSF}}
\newcommand{\FDLSF}{{\sf FDLSF}}
\newcommand{\E}{\mathbb{E}}

\newcommand{\rE}{{\sf E}}

\newcommand{\gE}{[\rE]}

\newcommand{\End}{\mathrm{End}}
\renewcommand{\epsilon}{\varepsilon}
\newcommand{\ev}{{\rm ev}}
\newcommand{\ext}{\mathchoice
	    {\raisebox{1pt}{$\bigwedge$}}
             {\raisebox{1pt}{$\bigwedge$}}
             {\raisebox{0.5pt}{$\scriptstyle\bigwedge$}}
             {\raisebox{0.2pt}{$\scriptscriptstyle\bigwedge$}}}
\newcommand{\F}{{\sf F}}

\newcommand{\geom}[1]{[#1]}

\newcommand{\rG}{\mathsf{G}}
\newcommand{\Gr}{\mathsf{Gr}}
\renewcommand{\H}{\mathscr{H}}
\newcommand{\Hol}{\mathrm{H}}

\newcommand{\hol}{\mathrm{hol}}

\newcommand{\cI}{\mathcal{I}}

\newcommand{\id}{\mathrm{id}}

\newcommand{\im}{\mathop{\mathrm{im}}}

\newcommand{\cJ}{{\mathcal J}}

\newcommand{\K}{\mathbb{K}}

\newcommand{\rL}{{\sf L}}

\newcommand{\lapp}{\Delta} 		
	
\newcommand{\Loop}{\mathsf{L}}

\newcommand{\mM}{M}
\newcommand{\N}{\mathbb{N}}

\newcommand{\new}[1]{{\textcolor{black}{\em #1}}}
\renewcommand{\o}{\circ}
\newcommand{\oo}{\circ\hspace{-0.45mm}\circ}
\newcommand{\ol}{\overline}

\newcommand{\mP}{P}
\newcommand{\pP}{\mathsf{P}}
\newcommand{\pPind}{\pP_{\text{\sf i}}}
\newcommand{\pPbas}{\pP_{\text{\sf b}}}
\newcommand{\nPind}{\pP_{\text{\sf i}}^{\N}}
\newcommand{\nPbas}{\pP_{\text{\sf b}}^{\N}}

\newcommand{\proj}[1]{{\sf\Pi}^{#1}}

\newcommand{\Q}{{\mathsf Q}}

\newcommand{\R}{\mathbb{R}}
\newcommand{\sR}{\mathsf{R}}

\newcommand{\rank}{\mathop{\mathrm{rank}}}
\newcommand\res[2]{#1\raisebox{-.25ex}{$|$}_{#2}}
\renewcommand{\S}{{\mathsf S}}

\newcommand{\SU}{\mathrm{SU}}
\newcommand{\SO}{\mathrm{SO}}
\newcommand{\submod}{\varrho}
\newcommand{\Or}{\mathrm{O}}
\newcommand{\rB}{\mathsf{B}}

\newcommand{\rH}{\mathsf{H}}

\newcommand{\rk}{{\mathrm{rk}}}

\newcommand{\Tr}{\mathrm{Tr}}
\newcommand{\tr}{\mathop{\rm tr}}
\newcommand{\cT}{\mathcal{T}}

\def\set#1{{\llbracket #1 \rrbracket}} 

\newcommand{\Sp}{\mathrm{Sp}}

\newcommand{\T}{\mathsf{T}}
\renewcommand{\u}{\mathfrak{u}}

\newcommand{\ul}{\underline}
\newcommand{\Un}{{\mathrm{U}}}
\newcommand{\rV}{{\sf V}}

\newcommand{\Vect}{{\mathrm{Vect}}}

\newcommand{\X}{{\mathsf X}}
\newcommand{\x}{x}
\renewcommand{\P}{\mathbb{P}}

\newcommand{\Z}{\mathbb{Z}}

\newcommand{\sJ}{{\sf J}}
\newcommand{\sC}{{\sf C}}
\newcommand{\sP}{{\sf P}}
\newcommand{\sT}{{\sf T}}
\renewcommand{\le}{\leqslant}
\renewcommand{\ge}{\geqslant}
\renewcommand{\leq}{\leqslant}
\renewcommand{\geq}{\geqslant}
\newcommand{\sbullet}{\raisebox{0.5pt}{$\scriptstyle\bullet$}}
\newcommand{\vsbullet}{\raisebox{0.5pt}{$\scriptscriptstyle\bullet$}}

\newcommand{\lanx}{\langle} 
\newcommand{\ranx}{\rangle}

\newcommand\sPhi{{\mathsf \Phi}}

\title{Determinantal linear spanning forests}

\author{Adrien Kassel}
\address{Adrien Kassel -- CNRS \& UMPA, ENS de Lyon}
\email{adrien.kassel@ens-lyon.fr}

\author{Thierry L\'evy}
\address{Thierry L\'evy -- LPSM, Sorbonne Universit\'e, Paris}
\email{thierry.levy@sorbonne-universite.fr}

\date{\today}



\begin{document}

\begin{abstract}
We introduce a determinantal linear process on the space of $1$-forms over a weighted graph endowed with a unitary local system, which we call the determinantal linear spanning forest (DLSF). This process is the analog of the uniform spanning tree in the presence of an arbitrary rank $n$ quenched gauge field in $\U(n)$ (or $\Or(n)$, $\SO(n)$, $\SU(n)$, $\Sp(n)$), and generalizes the determinantal random cycle-rooted spanning forest which corresponds to the case of $\U(1)$ and $\SU(2)$.

We describe the support of this process by means of a polymatroid which is described in terms of the circular and bicircular matroids. We also explain how the DLSF relates to the computation of currents in an electrical network in the presence of a gauge field.

The partition function of the DLSF is shown to be the determinant of the twisted Laplacian, for which we also give a matrix-tree type theorem that computes its characteristic polynomial. We initiate the study of the twisted Laplacian determinant as a functional on the space of connections, and give a first order expansion of it in a neighborhood of the trivial connection, for which we provide a probabilistic interpretation. 

We extend the construction to infinite graphs, exhibiting the analogs of the classical wired and free distributions on spanning forests. We give a sufficient geometric condition on the local spectral gap for the existence of a Poincaré inequality as well as the uniqueness of the infinite volume measure, and the exponential decay of correlations. We show that on any recurrent graph, the wired and free DLSF coincide. We construct examples where they do not, for example, unlike in the classical case, on $\Z^d$ for $d\geq 3$.

\end{abstract}

\maketitle

\setcounter{tocdepth}{1}
{\small 
\tableofcontents
}

%%%%%%%%%%%%%%%%%%%%%%%%%%%%%%%%%%%%%%%%

\section{Introduction}\label{sec:intro}

In this paper, we propose an interpretation of the determinant of the covariant Laplacian on the space of $0$-forms with values in a Hermitian vector bundle of arbitrary rank with unitary connection over a graph, as the partition function of a geometrically natural probabilistic model. By geometrically natural, we mean that the definition of the probabilistic model does not require more structure than the definition of the Laplacian itself. 

The probabilistic model that we propose is continuous in nature, and takes the form of a random linear subspace $\Q$ of the space of $1$-forms $\Omega^{1}(\F)$ over the graph, that is the direct sum over the set $\gE$ of unoriented edges of the graph of a random linear subspace $\Q_{e}$ of the fibre $\F_{e}$ over each edge $e$. Thus, this random linear subspace $\Q$, that we call the determinantal linear spanning forest (DLSF), takes the form 
\[\Q=\bigoplus_{e\in [\rE]} \Q_e \subseteq \Omega^1(\F).\]
By remembering only the dimension of the linear subspace $\Q_{e}$ of each fibre, we obtain a slightly less detailed, but still very interesting probabilistic model, that is now a random vector of integers indexed by the edges of the graph, and that we call the trace of the DLSF:
\[\tr \Q = (\dim \Q_{e})_{e\in \gE} \in \N^{\gE}.\]

For the DLSF as well as for its trace, each possible configuration is affected with a probability that depends  in a simple way, at least in principle, on the connection $h$. In the case of the DLSF itself, the weight of a linear subspace $Q$ is the determinant
\[w_h(Q)=\det\big(d_{h}^{*} \proj{Q}d_{h}\big),\]
where $d_{h}$ is the discrete covariant differential, $d_{h}^{*}$ its adjoint, and $\proj{Q}$ the orthogonal projection on $Q$. This weight has to be understood with reference to an appropriately normalized rotation invariant measure $\nu$ on the Grassmannian of $\Omega^{1}(\F)$, and the corresponding probability is given by the square of the cosine of the angle between $Q$ and the space of exact $1$-forms: 
\[\d\P(\Q=Q) = \cos^{2}(Q,\im d_{h}) \, \d\nu (Q).\]

For the trace of the DLSF, the weight of a vector of integers $\ul{m}=(m_{e})_{e\in \gE}$ is the square norm of the $\ul{m}$-th component of the image by $d_{h}$ of the exterior product $\iota_{\Omega^{0}(\F)}$ of the elements of an orthonormal basis of $\Omega^{0}(\F)$
in the exterior algebra of $\Omega^{1}(\F)$, relatively to the splitting
\[\ext^{\sbullet} \Omega^{1}(\F)=\bigoplus_{\ul{m}\in \gE^{\N}}\bigg( \bigotimes_{e\in \gE} \ext^{m_{e}} \F_{e}\bigg).\]
The corresponding probability is the norm of the $\ul{m}$-th component, relatively to the same splitting, of a unit Pl\"ucker embedding of $\im d_{h}$:
\[\P(\tr \Q=\ul{m})=\big\| (\iota_{\im d_{h}})_{\ul{m}} \big\|^{2}.\]

We propose to think of the support of the distribution of the trace of the DLSF as a generalization of the sets of spanning trees or cycle-rooted spanning forests which are known to appear in this situation when the vector bundle has rank $1$ \cite{Burton-Pemantle, Forman}, and for a special unitary connection when it has rank $2$ \cite{Kenyon}. This support, as we shall explain, is the set of integer points of a certain convex polytope, the basis polytope of a polymatroid over the set $\gE$, and one of our main results states that the determinant of the covariant Laplacian is the partition function of the probabilistic model on this finite set that we just briefly described.
\medskip

The construction of DLSF is an instance of the more general theory of determinantal linear processes (DLP), that we developed with this purpose in mind in \cite{KL3}. The DLSF itself is one of the probabilistic manifestations of a connection, or gauge field, on a vector bundle over a graph. The adjective {\em deteminantal}, and the presence of an exterior algebra a few lines above, indicate that it would qualify as a ferminonic object. In \cite{KL1}, we studied a bosonic counterpart of the DLSF, the Gaussian free vector field (GFVF). These references to our work are here to emphasize the fact that the present paper is a place where several threads of our reflection come together. However, we consider that there is still much to be understood about the DLSF, and plan to pursue this line of investigation. 

To put the present work in a broader context, we believe that it fits into the program of constructing meaningful probabilistic couplings of matter fields and gauge fields over various kinds of space or space-time. Of course, the gauge field is deterministic, or quenched, in the present paper. However, it is one of our goals in the long run to study DLSF with respect to random, possibly Yang--Mills connections.
\smallskip

This paper was announced several times, in talks and in papers, over several years. For instance, some of the properties of the DLSF were announced in~\cite[Chapter 6, Section~5]{Kassel}. Over time, we changed our mind about the name to be given to the main object that we introduce, the DLSF. It was initially briefly introduced in~\cite[Section 1.5]{KL3} under the name {\em quantum spanning forest} (QSF), which we chose to reflect two of its main properties. Firstly, it is a determinantal linear processes (DLP), with an interpretation as a projector operator-valued measure (POVM, see \cite[Section 5.8]{KL3}), and it has a fermionic interpretation as all determinantal processes. Secondly, it can be thought of as an intricated superposition of spanning trees or cycle-rooted spanning forests (see Theorem \ref{thm:superposition-intro}). Nevertheless, we came up more recently with the more down-to-earth, and more informative, name of determinantal linear spanning forest, which we now prefer. 
\smallskip 

In the rest of this introduction, we describe DLSF and discuss the main mathematical results of this paper.

\subsection{Local systems on weighted graphs}

The geometric setting that we consider is that of a finite\footnote{We discuss infinite graphs below.} connected graph $\rG=(\rV,\rE)$, endowed with positive edge weights $\ul{x}=(x_e)$, and a \emph{local system}~$(\F,h)$ consisting in a \emph{vector bundle} $\F$ over $\rG$ and a \emph{connection} $h$, connecting the fibres of $\F$ over adjacent vertices and edges. We think of the graph as being bidirected, of $\rE$ as being the set of directed edges, and we denote by $\gE$ the set of undirected edges, which are pairs of inverse directed edges.

The direct sum of the fibres of $\F$ over the vertices, resp. over the edges, of $\rG$ is called the space of $0$-forms, denoted by~$\Omega^{0}(\F)$, resp. the space of $1$-forms, denoted by $\Omega^{1}(\F)$. 

The topology of the graph and the connection $h$ together determine a linear operator of discrete differentiation $d_{h}:\Omega^{0}(\F) \to \Omega^{1}(\F)$, see \eqref{eq:defdh}. The spaces of forms are endowed with inner products, in the case of $\Omega^{1}(\F)$ depending on the edge weights $\ul{x}$, see \eqref{eq:SPOmega0} and \eqref{eq:SPOmega1}, and this allows one to define the adjoint operator $d_{h}^{*}:\Omega^{1}(\F)\to \Omega^{0}(\F)$. The action of the operator~$d_{h}$ and its adjoint~$d_{h}^{*}$ on the spaces of $0$- and $1$-forms split by the fibred structure of $\F$ (see \eqref{eq:graphicalsplitting}) is one of our fundamental objects of study.

A local system is said to be {\em flat} if it is isomorphic, in a natural sense, to a trivial vector bundle, with constant fibres and connection identically equal to the identity. This is equivalent to the space of {\em horizontal sections}, that is, the kernel of $d_{h}$, having maximal dimension, equal to the rank of the bundle.  The case of flat local systems is, so to say, the classical case, where the connection plays no role whatsoever.

Any local system splits canonically into the direct sum of a flat local system and a local system on which $d_{h}$ is injective, that we call {\em fully twisted} (see Proposition~\ref{prop:canosplit}). We use this \emph{canonical splitting} throughout the paper to simplify our analysis and help us focus on the more interesting case of fully twisted local systems.

\subsection{Determinantal linear spanning forests}

On the weighted graph $(\rG,\ul{x})$ endowed with a local system $(\F,h)$, the determinantal linear spanning forest (DLSF) is an instance of the {\em determinantal linear processes} (DLP) introduced in \cite{KL3}. The DLSF is built to capture two aspects of the geometry: the geometry of the weighted graph, and the twist in the discrete differential introduced by the connection. 

The three ingredients needed to define a DLP are an inner product space, a splitting of this space, and a linear subspace of this space. For the DLSF, these are
\begin{itemize}
\item the space $\Omega^{1}(\F)$ of $1$-forms with its inner product structure \eqref{eq:SPOmega1},
\item  the splitting \eqref{eq:graphicalsplitting} of $\Omega^{1}(\F)$ into fibres, 
\item the subspace $\im d_{h}\subseteq \Omega^{1}(\F)$ of exact $1$-forms.
\end{itemize}

The DLSF is a random linear subspace $\Q$ of $\Omega^{1}(\F)$ such that almost surely the equality
\begin{equation}\label{eq:maxacyc}
\Omega^1(\F)=\Q\oplus\ker d_{h}^{*}
\end{equation}
holds. This equality is also equivalent to the fact that~$\proj{\Q}d_h:\Omega^{0}(\F)\to \Q$ is injective, where $\proj{\Q}$ denotes the orthogonal projection on $\Q$.

Here is a concrete procedure for sampling the DLSF. Pick an orientation $\rE^+\subset \rE$ of the set of edges. For each oriented edge $e$, pick uniformy at random an orthonormal basis $(\omega^{1}_{e},\ldots,\omega^{n}_{e})$ of the fibre $\F_e$ over $e$, where $n$ denotes the rank of the local system. Merge these bases to form an orthonormal basis $(\omega_e^i: e\in \rE^+, i\in \set{n})$ of $\Omega^1(\F)$ indexed by $\rE^{+}\times \set{n}$. Let $K$ be the matrix in this basis of the orthogonal projection on $\im d_{h}$. Sample a determinantal point process with kernel $K$ on the finite set $\rE^{+}\times \set{n}$, and denote it by $\X=\bigsqcup_{e\in \rE^{+}} \X_e$, where for each $e\in \rE^{+}$, we have set $\X_{e}=\X \cap (\{e\}\times \set{n})$. Now, for each oriented edge $e$, set $\Q_e=\bigoplus_{i\in \X_e} \C \omega_e^i\subseteq \F_e$, and finally
\[\Q=\bigoplus_{e\in \rE^+} \Q_e \subseteq \Omega^1(\F).\]
This is the determinantal linear spanning forest.

The equality \eqref{eq:maxacyc} can be regarded as a condition of maximal acyclicity for $\Q$, correctly suggesting that the DLSF is a generalization of a random spanning tree. As another manifestation of this analogy, the DLSF yields a probabilistic solution to the computation of currents in an electrical network in the presence of holonomy, see \eqref{eq:kir-solution-proba}. The common mechanism to the case of spanning trees and DLSF is the mean projection theorem proved for determinantal processes in~\cite{KL4}.

\subsection{Polymatroids and the trace of the DLSF} The {\em trace} of the DLSF is the random vector of integers
\[\tr \Q = (\dim \Q_e : e \in \gE),\]
which contains a lot of information about the DLSF. For instance, the support of the distribution of the DLSF, which is a closed subset of the Grassmannian of $\Omega^{1}(\F)$, is completely determined by the support of the distribution of its trace, which is a subset of $\N^{\gE}$ that we denote by $\Sigma_{h}$.

The subset $\Sigma_{h}$ of $\N^{\gE}$ turns out to be the set of integer points of a polytope in $\R_+^{\gE}$, which is the basis polytope of a polymatroid associated to the local system $(\F,h)$. This polymatroid is defined by its rank function $\submod_{h}$, which has the following expression (see Theorem~\ref{theo:computerank}).

\begin{theorem}\label{thm:computerank-intro}
For every subset $A$ of $\gE$, we have
\begin{equation*}
\submod_{h}(A)=\rk (d_{h}^{A})=N |\rV(A)|- \dim \ker (d_{h}^{A}).
\end{equation*}
where $d_h^A$ is the covariant derivative on the restriction of $(\F,h)$ to the subgraph $(\rV(A),A)$ induced by $A$.
\end{theorem}

The support of the distribution of $\tr \Q$ is then
 \[\Sigma_{h}=\bigg\{(m_{e}) \in \N^{\gE}: \sum_{e\in \gE} m_{e}=\submod_{h}(\gE) \text{ and } \forall A\subseteq \gE,  \sum_{e\in A} m_{e} \le \submod_{h}(A)\bigg\}.\]

For a generic connection $h$, the set $\Sigma_{h}$ admits an alternative description in terms of two subsets of $\{0,1\}^{\gE}$, namely the set of indicator functions of spanning trees of $\rG$, which are the integer points of the basis polytope $\pPbas(\submod_{\o})$ of the circular matroid, and the set of indicator functions of cycle rooted spanning forests (CRSF)  of $\rG$, which are the integer points of the basis polytope~$\pPbas(\submod_{\oo})$ of the bicircular matroid.

%{\color{red}Moreover, we show that $\Sigma_h$ is always included in this polytope, as summarized in the next restul.}

\begin{theorem}\label{thm:superposition-intro}
Set $k=\dim \ker d_{h}$. The support of the distribution of the trace of the DLSF is included in the Minkowski sum of $k$ times the basis polytope of the circular matroid, and $(n-k)$ times the basis polytope of the bicircular matroid:
\[\Sigma_{h}\subseteq \underbrace{\pPbas(\submod_{\o})+\ldots+\pPbas(\submod_{\o})}_{k}+\underbrace{\pPbas(\submod_{\oo})+\ldots+\pPbas(\submod_{\oo})}_{n-k}.\]
Moreover, for all $h$ in an open subset with negligible complement in the set of connections on $\F$ such that $\dim \ker d_{h}=k$, this inclusion is an equality.
\end{theorem}

In concrete terms, the first assertion says that almost surely, $\tr \Q$ is the sum of the indicator functions of $k$ spanning trees and $(n-k)$ CRSF of $\rG$.

The notation $\Sigma_{h}$ is not used in the course of the paper, where this set is denoted by $\nPbas(\submod_{h})$.

\subsection{The spectrum of the twisted Laplacian}

Pursuing the analogy between spanning trees and DLSF, we give an expression for the characteristic polynomial of the twisted Laplacian $\Delta_{h}=d_{h}^{*}d_{h}$, in which the combinatorial and holonomic contributions are well separated. In the following result, $\T(\ul{x})$ denotes the generating polynomial of spanning trees of $\rG$.

\begin{theorem}\label{thm:mtt-intro}
Set $k=\dim \ker d_{h}$. The product of non-zero eigenvalues of $\Delta_h$ is 
\[{\det}^{*}(\Delta_h) = \vert \rV \vert^k \T(\ul{x})^k \sum_{\ul{m}\in (n-k) \pPbas(\submod_{\oo})} \ul{x}^{\ul{m}} w_h(\ul{m})\,,\]
where $w_h(\ul{m})$ is defined in \eqref{eq:wh}.
\end{theorem}

With the notation $w_h(Q)$ introduced at the beginning of this introduction, the number $w_h(\ul{m})$ can be written as
\[w_h(\ul{m})=\int_{\prod_{e\in \gE} \Gr_{m_e}(\F_e)} w_h\Big(\bigoplus_{e\in \gE} Q_e \Big) \; \bigotimes_{e\in \gE} {\rm d}\nu_{m_e}(Q_e)\,.\] 
where the measure $\nu_{m_e}$ on $\Gr_{m_e}(\F_e)$ is defined in Section \ref{sec:dls}.

The result above is stated as Theorem \ref{thm:mtt}. It is extended by Theorem \ref{thm:carlap} into an expression for the full characteristic polynomial of $\Delta_{h}$. This is a generalization of the classical matrix-tree theorems (see \cite{KL2,KL2b} for an overview and unified proof of classical MTTs, and \cite{KL6} for some new expressions).

In the case of a fully twisted local system, that is, in the case where $k=0$, the Laplacian is invertible, and this theorem expresses its determinant as a sum over the integer points of the polymatroid with rank function $n\submod_{\oo}$, the $n$-fold multiple of the bicircular matroid. The corresponding subset of the Grassmannian of $\Omega^{1}(\F)$ consists in subspaces which we may call \emph{linear spanning forests}. 

As explained at the beginning of the introduction, Theorem \ref{thm:mtt-intro} provides a possible answer to \cite[Section 14, Open question 2]{Kenyon}.

\subsection{A functional on the moduli space of connections}

The function $\ul{m}\mapsto w_h(\ul{m})$ defined on the set of integer points of the convex polytope $(n-k)\pPbas(\submod_{\oo})$ involves integrals over the Grassmannian and is somewhat complicated. One of the motivations of the formulas obtained in~\cite{KL6} is to explain this quantity. 

Viewing $\det(\Delta_h)$ as a function of $h$, we observe that things simplify nicely when $h$ tends to a trivial connection $h_0$. First of all when $h=h_0$, $\det \Delta_h=0$, and the expression from Theorem~\ref{thm:mtt-intro} above simplifies to ${\det}^{*} (\Delta_h)=\vert\rV\vert^n \T(\ul{x})^n$. 

Moreover, we can  differentiate with respect to $h$ at $h_0$ and obtain a simpler result. This is the content of Theorem~\ref{thm:partition-limit} which we informally reproduce here. We denote by $\u(n)$ the space of skew-Hermitian matrices of size $n$.

\begin{theorem}\label{thm:infinitesimal-intro}
Let $A\in \Omega^1(\u(n))$ be a $\u(n)$-valued $1$-form and define $h_t=e^{tA}$ to be the associated unitary connection.
As $t$ tends to $0$, we have
\begin{equation*}
\det \Delta_{h_t}=t^{2n} \; \sum_{\underline{K}\in \cC(\rG,n)} |\det M_{A}(\ul{K})|^{2} \; \ul{x}^{\underline K}+O(t^{2n+1})
\,,
\end{equation*}
where $\cC(\rG,n)$ is the set of $n$-tuples of connected spanning subgraphs of $\rG$ with total first Betti number equal to $n$, and $M_A(\ul{K})$ is an $n\times n$ matrix built from the integration of $A$ along an integer basis of cycles of $\ul{K}$, see \eqref{eq:matriceA}.
\end{theorem}

As we explain in Section \ref{sec:infinitesimal-connections}, the above polynomial may be interpreted as the partition function of a random process which is the limit of the trace of the DLSF when $t$ tends to zero.

\subsection{Infinite graphs}

The definition of DLSF can be extended to infinite graphs following the general theory of DLP, as explained in  \cite[Section 4.9]{KL3}. As in the classical untwisted case (uniform spanning forests), there is a choice as to which measure in infinite dimension is the analog of the finite case, and this depends on how we want to approximate the infinite graph by finite graphs. Choosing wired or free boundary conditions for exhausting sequences of finite graphs yields two natural measures, which we call the wired and free DLSF, respectively denoted WDLSF and FDLSF.

A notable difference with the classical case is the following. It is well known (see~\cite{BLPS}) that on~$\Z^d$ with unit weights, we have $\text{WSF}=\text{FSF}$. When $d\ge 3$, we construct examples of local systems on~$\Z^d$ for which $\text{WDLSF}\neq \text{FDLSF}$.

We also give a condition for these measures to coincide, and provide examples, also on $\Z^d$, where this is the case. In that case, the correlations of the process can be evaluated, and are shown to decay at least exponentially with the graph distance. 

Studying the fine properties of these distributions leads to studying local holonomic properties of the local system, which seems to be a worthwhile path to explore. In particular, we exhibit a local-to-global principle for the existence of a spectral gap from holonomies of loops to the twisted Laplacian.

\begin{theorem}\label{thm:infinite-intro}
If there is a uniform holonomic spectral gap, then the Laplacian has a spectral gap. In particular, the Green function decays exponentially with the graph-distance. This implies that  $\text{WDLSF}=\text{FDLSF}$ and their correlations decay exponentially with the graph-distance.
\end{theorem}

The exponential decay of the correlation also holds for the Gaussian free vector field (GFVF), as explained, along with this theorem, in Section \ref{sec:wired-free}.

\subsection{Organization of the paper}

In Section~\ref{sec:local-systems} we introduce local systems on graphs and explain their basic properties.
In Section~\ref{sec:holonomic-constructions} we provide some useful properties related to holonomy, in particular a twisted Poincaré formula which plays a role in the study of infinite graphs. 

Determinantal linear spanning forests are introduced in Section~\ref{sec:DLSF}, where we also state their basic properties. Section~\ref{sec:DLSF-Polymatroids} is devoted to the characterization of the support of the distribution of DLSF and the link to polymatroids, for which we provide a self-contained introduction. 

In Section~\ref{sec:characteristic-polynomial}, we explain why the determinant of the twisted Laplacian is the partition function for DLSF, and also provide a matrix-tree type formula for the characteristic polynomial of the twisted Laplacian. In Section~\ref{sec:infinitesimal-connections} we study this determinant as a function of the connection. We compute the leading order of the partition function for a connection that is a small deformation of a trivial connection, and we provide a probabilistic interpretation for this partition function.

Section~\ref{sec:electrical-networks} is devoted to the electrical networks with gauge field interpretation of DLSF. It is in many ways similar to the classical electrical networks interpretation of spanning trees, but the presence of holonomy lets new phenomena emerge.%show up.

In Section~\ref{sec:infinite-graphs} we extend the theory of local systems from Section \ref{sec:local-systems} to infinite graphs and discuss the specific analytic aspects that appear. This is used in Section~\ref{sec:wired-free} to define DLSF on infinite graphs and exhibit two particular measures, the wired and free DLSF, obtained as limits on exhaustions by finite graphs with free or wired boundary conditions. Geometric conditions on the local system are given to ensure that these two infinite distributions coincide, and that their correlations decay exponentially.

\subsection{Acknowledgments}
We thank Richard Kenyon for inspiring discussions, almost 15 years ago, during the first-named author's PhD thesis where the question of higher rank bundles was raised, and which were somewhat at the root of this project. H\'{e}lo\"{i}se Constantin wrote part of her PhD thesis on the study of DLSF (then called QSF) on vector bundles of arbitrary rank on the hypercubic lattice $\Z^d$ and we thank her for stimulating discussions and for her understanding and patience in working on a model, the bases of which were only written down in the first draft version of this paper.
We also thank Omid Amini for his interest in our work and for inspiring discussions on related topics. 

%%%%%%%%%%%%%%%%%%%%%%%%%%%%%%%%%%%%%%%%

\section{Local systems over graphs}\label{sec:local-systems} 

\subsection{Graphs}\label{sec:def-graphs} 

In this paper, we work with locally finite graphs, that are bidirected, in the sense that each edge can be considered with both orientations.

\begin{definition}\label{defgraph}
A \new{graph} is a quintuple $\rG=(\rV,\rE,s,t,i)$ consisting of 
\begin{itemize}
\item two non-empty sets $\rV$ and $\rE$,
\item two maps $s,t:\rE\to \rV$ such that for all $v\in \rV$, the sets $s^{-1}(v)$ and $t^{-1}(v)$ are finite,
\item a fixed-point free involution $i$ of $\rE$, such that the equality $t\circ i=s$ holds.
\end{itemize}
If $\rV$ is finite, the graph is said to be \new{finite}.
\end{definition}

The elements of $\rV$ are the \new{vertices} of the graph, the elements of $\rE$ its \new{edges}. Each edge~$e$ is oriented and joins its source $s(e)$ to its target $t(e)$. We will usually write $\ul{e}=s(e)$, $\ol{e}=t(e)$ and $e^{-1}=i(e)$ for the source, target, and inverse of an edge. We will also write a graph simply as a pair $(\rV,\rE)$ and not mention the maps $s,t,i$.

The involution $i$ induces an equivalence relation on $\rE$ and we define the set of \new{geometric edges} of the graph as the quotient set
\[\gE=\rE/i\,.\]
The distinction between an edge and the corresponding geometric edge is often immaterial, but when it is needed, we will denote by $\geom{e}$ the geometric edge corresponding to an edge $e$, that is,  $\geom{e}=\{e,e^{-1}\}=\geom{e^{-1}}$. See Figure~\ref{fig:graph} for an example. 

\begin{figure}[!ht]
    \centering
    \includegraphics[scale=.7]{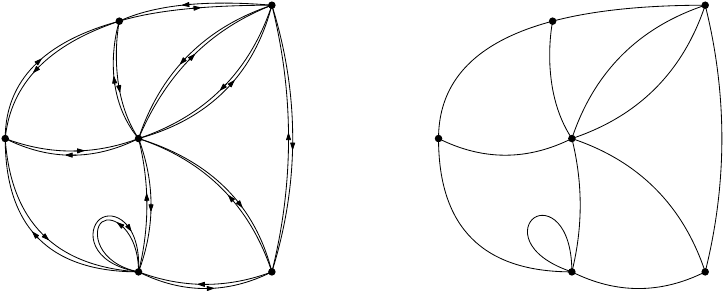}
    \caption{\small Left: a graph with all directed edges drawn. Right: the same graph with only geometric edges drawn.}
    \label{fig:graph}
\end{figure}

For any subset $A$ of $\rE$ stable by the involution $i$, we denote by $\rV(A)$ the set of vertices adjacent to at least one edge of $A$, and we call $(\rV(A),A)$ the subgraph of $\rG$ induced by $A$. We will also use this notion in the case where $A$ is a subset of $\gE$, with the obvious meaning.

\subsection{Vector bundles}

Let us now define a vector bundle over a graph $\rG$. For this, we choose a base field $\K$ that is either $\R$, $\C$, or $\bH$ (the skew-field of real quaternions), and an integer $N\geq 1$. 

\begin{definition}
A \new{$\K$-vector bundle of rank $N$ over $\rG$} is a collection $\F=((\F_{v})_{v\in \rV}, (\F_{e})_{e\in \rE})$ of vector spaces over $\K$ which all have the same dimension $N$, and such that for all edge $e\in \rE$, $\F_{e}=\F_{e^{-1}}$.
\end{definition}

In this paper, we will always consider vector bundles $\F$ which are \new{Euclidean} (if $\K=\R$), \new{Hermitian} (if $\K=\C$), or \new{quaternionic Hermitian} (if $\K=\bH$), in the sense that each of the vector spaces of which they consist is endowed with a Euclidean, Hermitian, or quaternionic Hermitian inner product. The bilinear or sesquilinear form on the vector space $\F_{v}$ for $v\in \rV$ (resp. $\F_{e}$ for $e\in \rE$) is denoted by $\langle \cdot,\cdot \rangle_{v}$ (resp. $\langle \cdot,\cdot \rangle_{e}$). When $\K=\C$ or $\bH$, we take this form to be antilinear in the first variable and linear in the second, according to the physicists' convention. We also denote by $\|\cdot\|_v$ and $\|\cdot\|_e$ the associated norms.

We will sometimes consider vector bundles of which all the fibres are the same inner product space~$F$, for instance $\C^{N}$ with its usual Hermitian structure. We will denote this vector bundle by $\rG\times F$ or, if no confusion is possible, simply by $F$. 

Let $\F$ be a vector bundle over a graph $\rG=(\rV,\rE)$. The space of \new{sections} of $\F$, or \new{$0$-forms} on~$\rG$ with values in $\F$, is the vector space
\begin{equation}\label{eq:splitOmega0}
\Omega^{0}(\F)=\prod_{v\in \rV}\F_{v}\,.
\end{equation}

The space of \new{$1$-forms} on $\rG$ with values in $\F$ is the space
\begin{equation}
\Omega^{1}(\F)=\bigg\{\omega \in \prod_{e\in \rE} \F_{e} : \forall e \in \rE, \; \omega(e^{-1})=-\omega(e) \bigg\}\,.
\end{equation}
It is convenient to attach a space of $1$-forms to each geometric edge $\geom{e}$ by setting
\begin{equation}\label{eq:Fe}
    \F_{\geom{e}}=\{\omega\in\Omega^{1}(\F) : \forall e'\notin \geom{e}, \, \omega(e')=0\}\,,
\end{equation}    
so that the space of $1$-forms can be written
\begin{equation}\label{eq:splitOmega1}
\Omega^{1}(\F)=\prod_{\geom{e}\in \gE} \F_{\geom{e}}\,.
\end{equation}
Note that for all $e\in \rE$, the linear maps $\ev_e:\F_{\geom{e}}\to \F_e, \omega\mapsto \omega(e)$ are isomorphisms (satisfying the relation $\ev_{e^{-1}}=-\ev_e$). Despite the subtle distinction that we just made, we will identify $\F_{e}$ and $\F_{[e]}$ in all situations where this does not create confusion.

The {\em total space} of $\F$ is the vector space of forms of inhomogeneous degree over $\rG$, that is,
\begin{equation}\label{eq:totalspace}
\Omega^{\vsbullet}(\F)=\Omega^{0}(\F)\oplus \Omega^{1}(\F).
\end{equation}

\subsection{Local systems}  \label{sec:locsys}

Our main object of study in this paper will be {\em local systems} over $\rG$, defined as vector bundles  endowed with a connection. 

In order to define connections, we define the sets of \new{initial} and \new{final half-edges}
\begin{equation}\label{fibprods}
\rH_{s}=\{(u,e) \in \rV\times \rE: u=s(e)\}  \mbox{ and } \rH_{t}=\{(e,v)\in \rE\times \rV : t(e)=v\}\,.
\end{equation}

\begin{definition}\label{def:connection}
A  \new{connection} $h$ on the vector bundle $\F$ is the data, for each $(u,e)\in \rH_{s}$, of a linear isomorphism $h_{e,u}:\F_{u}\to \F_{e}$.  
\end{definition}

According to this minimalistic definition, for each edge $e$ with endpoints $u=\ul{e}$ and $v=\ol{e}$, a connection~$h$ specifies two isomorphisms 
$h_{e,u}:\F_{u}\to \F_{e}$ and $h_{e^{-1},v}:F_{v}\to \F_{e^{-1}}=\F_{e}$. We complete this collection by setting
\[h_{e^{-1},u}=h_{e,u}\ , \ \ h_{u,e}=h_{u,e^{-1}}=h_{e,u}^{-1}\ , \ \ h_{e,v}=h_{e^{-1},v}\ , \ \ h_{v,e}=h_{v,e^{-1}}=h_{e,v}^{-1},\]
the logic being best explained by a glance at Figure \ref{fig:connection}. In all cases where $(a,b)$ is a half-edge of $\rG$, the connection $h$ thus specifies an isomorphism $h_{b,a}:\F_{a}\to \F_{b}$. This isomorphism is associated with the crossing of the half-edge joining $a$ to $b$, and reversing the direction turns the isomorphism into its inverse.

\begin{figure}[!ht]
    \centering
    \includegraphics[scale=1]{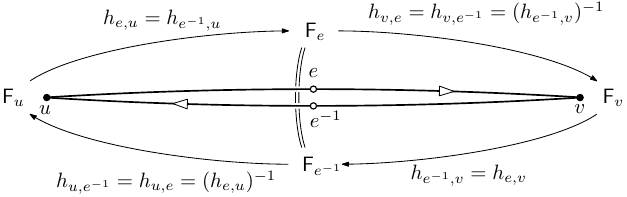}
    \caption{\small A connection on a graph, representing parallel transport along half-edges. }
    \label{fig:connection}
\end{figure}

In this paper, in addition to assuming that each fibre of a vector bundle is an inner product space, we will always assume that the isomorphisms defining the connection are isometries. For the sake of brevity, we will always use the words {\em Hermitian} and {\em unitary}, whether $\K$ be $\R$, $\C$, or~$\bH$.

We denote the set of all unitary connections on the Hermitian vector bundle $\F$ by $\A(\F)$.

\begin{definition}
A {\em local system} over $\rG$ is a pair $(\F,h)$, where~$\F$ is a vector bundle over $\rG$ and $h$ is a unitary connection on $\F$. 
\end{definition}

Two local systems $(\F,h)$ and $(\F',h')$ over $\rG$ are said to be {\em isomorphic} if there exists, for all vertex~$v$ and all edge $e$, isometries $f_{v}:\F_{v}\to \F'_{v}$ and $f_{e}:\F_{e}\to \F'_{e}$ such that for all half-edge $(v,e)$, one has 
\begin{equation} \label{eq:isomlocalsys}
f_{e}\circ h_{e,v}=h'_{e,v} \circ f_{v}.
\end{equation}

The group of {\em gauge transformations} 
\[\cJ(\F)=\prod_{v\in \rV} \U(\F_{v}) \times \prod_{e\in \gE} \U(\F_{e})\]
of the vector bundle $\F$, defined as the Cartesian product of the groups of unitary transformations of the fibres of $\F$, acts on the space $\A(\F)$ of connections on $\F$ according to the formulas
\[(j\cdot h)_{e,u}=j_{e}^{-1}\circ h_{e,u}\circ h_{u} \ \text{ and } \ (j\cdot h)_{v,e}=j_{v}^{-1}\circ h_{v,e}\circ j_{e}.\]
For any two connections $h,h'\in \A(\F)$, the local systems $(\F,h)$ and $(\F,h')$ are isomorphic if and only if $h$ and $h'$ differ by a gauge transformation. 

A {\em local subsystem} of $(\F,h)$ is a local system $(\F',h')$ such that 
\begin{itemize}
\item for all vertex $v$ and all edge~$e$, the inclusions $\F'_{v}\subseteq \F_{v}$ and $\F'_{e}\subseteq \F_{e}$ hold, 
\item  for all half-edge $(v,e)$, one has 
\begin{equation}\label{eq:defsubsystem}
h_{e,v}(\F'_{v})=\F'_{e},
\end{equation}
\item  $h'_{e,v}$ coincides with $h_{e,v}$ on $\F'_{v}$.
\end{itemize}

\subsection{Holonomy} \label{sec:holonomy}

Let $(\F,h)$ be a local system over $\rG$. For every edge $e\in \rE$, we define the holonomy along $e$ as
\[h_{e}=h_{\ol{e},e} \circ h_{e,\ul{e}}:\F_{\ul{e}}\to \F_{\ol{e}}.\]
Given a path in $\rG$, that is, a finite well-chained sequence $p=(e_{1},e_{2},\ldots, e_{n})$
of edges, and with the notation $\ul{p}=\ul{e_{1}}$ and $\ol{p}=\ol{e_{n}}$,
we define the {\em holonomy of $h$}, or {\em parallel transport} induced by $h$, along $p$ as
\[\hol_{h}(p)=h_{e_{n}}\circ \ldots \circ h_{e_{1}}: \F_{\ul{p}}\to \F_{\ol{p}}\,.\]
The basic properties of the holonomy imply that for every path $p$, the relation
\[\hol_{h}(p^{-1})=\hol_{h}(p)^{-1},\]
where $p^{-1}=(e_{n}^{-1},\ldots,e_{1}^{-1})$ denotes the path $p$ traced backwards. Moreover, whenever two paths $p_{1}$ and~$p_{2}$ can be concatenated, the relation 
\[\hol_{h}(p_{1}p_{2})=\hol_{h}(p_{2})\hol_{h}(p_{1})\]
holds.

A path $\gamma$ such that $\ol{\gamma}=\ul{\gamma}=v$ for some vertex $v$ is said to be a loop based at $v$, and $\hol_{h}(\gamma)$ belongs to the group $\Un(\F_{v})$ of unitary linear transformations of $\F_{v}$.

For every vertex $v$, we denote by $\Loop_{v}(\rG)$ the set of all loops based at $v$. It is a monoid for the operation of concatenation, the unit element being the constant loop at $v$. 

\begin{definition} \label{def:holonomygroup} Let $v$ be a vertex of $\rG$. The \emph{holonomy group} of the local system $(\F,h)$  at $v$ is the subgroup of $\U(\F_v)$ defined by 
\[\Hol_{h,v}=\{\hol_{h}(\gamma) : \gamma\in \Loop_{v}(\rG)\}\, .\] 
\end{definition}

If $\rG$ is connected, the holonomy groups at any two vertices are isomorphic, an isomorphism being given by the conjugation by the parallel transport (or holonomy) along any path joining the two vertices.

The holonomy group $\Hol_{h,v}$ is attached to the local system $(\F,h)$ and acts on the fiber $\F_{v}$. We call the space $\F_{v}$, acted on by the group $\Hol_{h,v}$, the {\em holonomy module at $v$} of $(\F,h)$. 

We will be interested in the subspaces of $\F_{v}$ that are stable under the action of $\Hol_{h,v}$, that we call {\em holonomy submodules at $v$} of $(\F,h)$.

\subsection{Local subsystems and holonomy submodules}

Let $(\F,h)$ be a local system over $\rG$. Let $v$ be a vertex of $\rG$ and $\Hol_{h,v}$ the holonomy group of $(\F,h)$ at $v$. 
There is a one-to-one correspondence between the local subsystems of $(\F,h)$ and its holonomy submodules at $v$.

\begin{proposition} \label{prop:localsysfibre}
Let $(\F,h)$ be a local system over the connected graph $\rG$. Let $v$ be a vertex of~$\rG$. \\
\indent 1. If $(\F',h')$ is a local subsystem of $(\F,h)$, then $\F'_{v}$ is a holonomy submodule at $v$ of $\F_{v}$.\\
\indent 2. For all holonomy submodule $W$ of $\F_{v}$, there exists a unique local subsystem $(\F',h')$ of $(\F,h)$ such that $\F'_{v}=W$.\\
\indent 3.~Two local subsystems $(\F',h')$ and $(\F'',h'')$ of $(\F,h)$ are isomorphic if and only if $\F'_{v}$ and $\F''_{v}$ are isomorphic holonomy submodules of $\F_{v}$.
\end{proposition}

We will use that fact that an isomorphism between local systems intertwines their holonomies, in the sense that, with the notation of the paragraph containing \eqref{eq:isomlocalsys}, and for all path $p$ joining a vertex $v$ to a vertex~$w$, we have the equality
\begin{equation}\label{eq:isomhol}
\hol_{h'}(p) \circ f_{v}=f_{w}\circ \hol_{h}(p) : \F_{v}\to \F'_{w}.
\end{equation}

\begin{proof}
1. Let $(\F',h')$ be a local subsystem of $(\F,h)$. For all path $p$ joining $v$ to some vertex $w$, we have 
\begin{equation}\label{eq:F'h}
\F'_{w}=\hol_{h'}(p)(\F'_{v})=\hol_{h}(p)(\F'_{v}).
\end{equation}
In particular, for every loop $\gamma$ in $\rG$ based at $v$, the equality
\[\hol_{h}(\gamma)(\F'_{v})=\F'_{v}\]
holds, and this is exactly the content of the first assertion.

2. For all vertex $w$, let us choose a path $p$ from $v$ to $w$, and define $\F'_{w}=\hol_{h}(p)(W)$. We claim that this definition does not depend on the choice of the path $p$. Indeed, if $q$ is another path from $v$ to $w$, then the concatenation $qp^{-1}$ is a loop based at $v$, so that $\hol_{h}(qp^{-1})(W)$ is equal to $W$, and
\[\hol_{h}(q)(W)=\hol_{h}(p)\big( \hol_{h}(p)^{-1}\hol_{h}(q)(W)\big)=\hol_{h}(p)\big( \hol_{h}(qp^{-1})(W)\big)=\hol_{h}(p)(W).\]
With this definition, for all vertices $x$ and $y$, and any path $q$ joining $x$ to $y$, one has $\F'_{y}=\hol_{h}(q)(\F'_{x})$. Indeed, given a path $p$ from $v$ to $x$, the path $pq$ is a path from $v$ to $y$, and
\[\hol_{h}(q)(\F'_{x})=\hol_{h}(pq)(W)=\F'_{y}.\]
In the case where $x$ and $y$ are the starting and finishing points of an edge $e$, and the path $p$ is taken to be the edge $e$, one has in particular $h_{e}(\F'_{x})=\F'_{y}$, or equivalently,
\[h_{e,x}(\F'_{x})=h_{e,y}(\F'_{y})\]
and we define $\F'_{e}$ as the common value of the two sides of this equality.

The vector bundle $\F'$ thus defined is, by construction, stable by $h$, and once endowed with the restriction of $h$, that we denote by $h'$, becomes a local subsystem of $(\F,h)$ of which the fibre over $v$ is $W$. 

Equation \eqref{eq:F'h} implies that any local subsystem which has fibre over $v$ equal to $W$ coincides with $\F'$ over all vertices of $\rG$, and \eqref{eq:defsubsystem} implies that it also coincides with $\F'$ over edges. This proves the uniqueness. 

3. Consider two local subsystems $(\F',h')$ and $(\F'',h'')$ of $(\F,h)$. If $(\F',h')$ and $(\F'',h'')$ are isomorphic, then \eqref{eq:isomhol} applied to loops based at $v$ expresses the fact that the isometry $f_{v}:\F'_{v}\to \F''_{v}$ intertwines the actions of the holonomy groups of $h'$ and $h''$. These actions are the restrictions of the action of $\Hol_{h,v}$ to the submodules $\F'_{v}$ and $\F''_{v}$, which are therefore isomorphic.

Conversely, if $\F'_{v}$ and $\F''_{v}$ are isomorphic submodules of $\F'_{v}$, then an isomorphism $f:\F'_{v}\to \F''_{v}$ can be extended to an isomorphism of local subsystems using the holonomy in a way that is very similar to the constructions already made in this proof. 
\end{proof}

The holonomy group $\Hol_{h,v}$ is countable and needs not have a particularly simple representation theory, but as far as the space $\F_{v}$ is concerned, its action has exactly the same stable subspaces, and the same notion of equivalence of submodules, as does its closure in $\U(\F_{v})$, which is compact. Therefore, any holonomy submodule admits a supplementary submodule, the fibre $\F_{v}$ admits a canonical decomposition into isotypical submodules, and this decomposition can be refined, non uniquely in general, into a decomposition as a direct sum of irreducible submodules. 

\begin{definition}\label{def:reducible}
We say that a local system $(\F,h)$ is \emph{reducible} if it is a direct sum of local subsystems of positive rank. Otherwise, we say that $(\F,h)$ is \new{irreducible}.

We say that $(\F,h)$ is \emph{isotypical} if it is a direct sum of isomorphic irreducible local systems. 
\end{definition}

The reducibility of a local system can be decided by studying the action of the holonomy group on any fibre. 

\begin{proposition}\label{prop:irredirred} Let $v$ be a vertex of the connected graph $\rG$. A local system $(\F,h)$ over $\rG$ is irreducible, resp. isotypical, if and only if $\F_{v}$ is an irreducible, resp. isotypical holonomy module.
\end{proposition}

\begin{proof} This follows directly from Proposition \ref{prop:localsysfibre}.
\end{proof}

Let us give an application to a special family of graphs, namely unicycles, this, is, connected graphs with a single simple cycle.

\begin{proposition} \label{prop:unicycrang1}
Over a unicycle, any irreducible local systems has rank $1$.
\end{proposition}

\begin{proof} The reason is that the holonomy group at any vertex $v$ of a local system over a unicycle consists in the powers of the holonomy along the single cycle, based at $v$. The irreducible holonomy submodules at $v$ are therefore the lines generated by eigenvectors, and they have dimension $1$.
\end{proof}

\subsection{Inner products}

Let us endow $\Omega^{0}(\F)$ and $\Omega^{1}(\F)$ with inner products. For this, let us assign a positive weight, which we call a \new{Euclidean weight}, to each edge of the graph. 

\begin{definition}
A collection of \new{Euclidean weights} on $\rG$ is a positive real-valued function $\x:\rE\to (0,+\infty)$ such that $\x\circ i=\x$, that is $x_{e^{-1}}=x_e$ for all $e\in \rE$.  
We let $\ul{x}$ denote the collection $(x_e)_{e\in \rE}$.
\end{definition} 

Given a local system $(\F,h)$ over a finite graph $\rG$ with Euclidean weights $\ul{x}$, the spaces $\Omega^{0}(\F)$ and $\Omega^{1}(\F)$ carry inner products, defined as follows:
\begin{equation}\label{eq:SPOmega0}
\forall f_{1},f_{2}\in \Omega^{0}(\F), \quad
(f_{1},f_{2})=\sum_{v\in \rV}  \langle f_{1}(v),f_{2}(v)\rangle_{v}
\end{equation}
and
\begin{equation}\label{eq:SPOmega1}
\forall \omega_{1},\omega_{2}\in \Omega^{1}(\F), \quad
(\omega_{1},\omega_{2})=\frac{1}{2}\sum_{e\in \rE} x_{e} \langle \omega_{1}(e),\omega_{2}(e)\rangle_{e}\,.
\end{equation}

Note that, in contrast to the choice made in \cite[Section 1.2]{KL1}, the inner product on $\Omega^{0}(\F)$ does not depend on the Euclidean weights.

\subsection{Splittings}
Splittings of vector spaces, that is, direct sum decompositions, will play a central role in this paper. In this section, we assume that $\rG$ is finite. The total space $\Omega^{\vsbullet}(\F)$ of a local system $(\F,h)$ comes with a splitting
\begin{equation}\label{eq:graphicalsplitting}
\Omega^{0}(\F)\overset{\gamma_{0}}{=}\bigoplus_{v\in \rV}\F_{v} \ , \ \  \Omega^{1}(\F)\overset{\gamma_{1}}{=} \bigoplus_{e\in \gE} \F_{e}\,    
\end{equation}
that we call the {\em graphical splitting} and that we denote by $\gamma=(\gamma_{0},\gamma_{1})$. Let us observe that the splittings $\gamma_{0}$ and $\gamma_{1}$ are orthogonal with respect to the inner products introduced in the previous section. In this paper, we will devote more attention to the splitting $\gamma_{1}$ of $\Omega^{1}(\F)$.

In the case where the local system $(\F,h)$ is the direct sum $(\F',h')\oplus (\F'',h'')$, the direct sums $\Omega^{0}(\F)=\Omega^{0}(\F')\oplus \Omega^{0}(\F'')$ and $\Omega^{1}(\F)=\Omega^{1}(\F')\oplus \Omega^{1}(\F'')$ need not be orthogonal. In fact, looking at Proposition~\ref{prop:localsysfibre} and the way in which local subsystems are determined by their fibre over a vertex, we see that 
\[\Omega^{0}(\F') \perp \Omega^{0}(\F'') \ \Leftrightarrow \ \Omega^{1}(\F') \perp \Omega^{1}(\F'')  \ \Leftrightarrow \  \exists v\in \rV, \; \F'_{v}\perp \F''_{v}   \ \Leftrightarrow \  \exists e\in \gE, \; \F'_{e}\perp \F''_{e}.\]
We say that the direct sum $(\F',h')\oplus (\F'',h'')$ is orthogonal if these equivalent conditions are satisfied. 

Incidentally, non-isomorphic irreducible local subsystems, which correspond to non-isomorphic irreducible holonomy submodules, are orthogonal.

\begin{definition}
A {\em splitting} of a local system is an orthogonal decomposition  as a direct sum of local subsystems. We call {\em split local system} a local system endowed with a splitting.
\end{definition}

We are not assuming that the summands of a splitting are irreducible. For instance, $(\F,h)=(\F,h)$ is a trivial example of a splitting of the local system $(\F,h)$. However, let us stress that it is part of the definition that the splitting is orthogonal. 

A split local system $(\F,h)=(\F^{1},h^{1})\oplus \ldots \oplus (\F^{r},h^{r})$ induces a {\em horizontal splitting} of $\Omega^{\vsbullet}(\F)$, namely
\begin{equation}\label{eq:horizontalsplitting}
\Omega^{\vsbullet}(\F)\overset{\eta}{=}\Omega^{\vsbullet}(\F^{1})\oplus \ldots \oplus \Omega^{\vsbullet}(\F^{r}),
\end{equation}
that we denote by $\eta=(\eta_{0},\eta_{1})$, and which combines with the graphical splitting to produce the {\em fine splitting} $\sigma=(\sigma_{0},\sigma_{1})$
\begin{equation}\label{eq:holsplit}
\Omega^{0}(\F)\overset{\sigma_{0}}{=}\bigoplus_{v\in \rV} \bigoplus_{i=1}^{r} \F^{i}_{v} \ , \ \ \Omega^{1}(\F)\overset{\sigma_{1}}{=} \bigoplus_{e\in \gE} \bigoplus_{i=1}^{r} \F^{i}_{e}.
\end{equation} 

We will use the letter $\sigma$ to denote the splitting of the local system $(\F,h)$ as well as the splitting of~$\Omega^{\vsbullet}(\F)$ that it induces. Accordingly, we will often denote by $(\F,h,\sigma)$ the split local system. 

\subsection{Covariant derivative}

The connection $h$ on the vector bundle $\F$ gives rise to a linear operator $d_{h}:\Omega^{0}(\F)\to \Omega^{1}(\F)$ defined as follows (see Figure~\ref{fig:derivative} for an illustration).

\begin{definition}[Covariant differential]\label{def:dh}
The covariant derivative, covariant differential, twisted coboundary operator, or simply \new{differential} is the linear operator $d_{h}:\Omega^{0}(\F) \to \Omega^{1}(\F)$ defined by setting, for every $f\in \Omega^{0}(\F)$ and every $e\in \rE$,
\begin{equation}\label{eq:defdh}
(d_{h}f)(e)=h_{e,\ol{e}} f(\ol{e})-h_{e,\ul{e}} f(\ul{e})\,.
\end{equation}
\end{definition}

\begin{figure}[!h]
\centering
\includegraphics[scale=1]{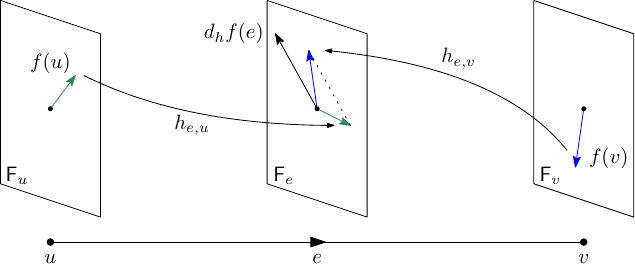}
\caption{\small Discrete covariant derivative acting over an edge $e$ connecting vertices $u$ and $v$.}\label{fig:derivative}
\end{figure}

The covariant derivative $d_{h}$ comes with a companion operator, the covariant coderivative $d_{h}^{*}$.

\begin{definition}[Covariant codifferential]\label{def:dh*}
The \new{covariant codifferential}, or simply codifferential, is the operator $d_{h}^{*}:\Omega^{1}(\F) \to \Omega^{0}(\F)$ defined by setting, for all $\omega \in \Omega^{1}(\F)$ and all $v\in \rV$,
\begin{equation}\label{eq:d-star}
(d_{h}^{*} \omega)(v)=\sum_{\substack{e\in \rE \\ \ol{e}=v}} x_{e} h_{v,e}\omega(e)\,.
\end{equation}
\end{definition}

As the notation suggests, and provided $\rG$ is finite, the map $d_{h}^{*}$ is the adjoint of $d_{h}$ with respect to the inner products \eqref{eq:SPOmega0} and \eqref{eq:SPOmega1}. For a proof of this elementary fact, see \cite[Lemma 3.11]{KL1} (with the very minor change due to the difference in inner products on $\Omega^0(\F)$).

%%%%%%%%%%%%%%%%%%%%%%%%%%%%%%%%%%%%%%%%

\section{Holonomic constructions}\label{sec:holonomic-constructions}

\subsection{Horizontal sections and flat local systems} \label{sec:horsec}

Among the isotypical local subsystems of a local system, the one that corresponds to the trivial representation of the holonomy group plays a special role. In a sense that we will now explain, it is closely related to the kernel of $d_{h}$. 

\begin{definition} Let $(\F,h)$ be a local system over $\rG$.
A section $f\in\Omega^0(\F)$ is \emph{horizontal} if $d_{h}f=0$. 
\end{definition}

The notion of horizontal section depends, of course, on the connection $h$.

\begin{proposition}\label{prop:hor}
A section $f\in \Omega^{0}(\F)$ is horizontal if and only if for all vertices $u$ and $v$ in the same connected component of $\rG$ and all path $\gamma$ joining $u$ to $v$, one has 
\begin{equation}\label{horedge}
f(v)=\hol_h(\gamma) f(u).
\end{equation}
In particular, if $f$ is a horizontal section, then $\|f\|$ is constant.
\end{proposition}

\begin{proof}
Consider a section $f$ and an edge $e$ with source $u$ and target $v$. Then 
\begin{equation}\label{compdf}
d_{h}f(e)=h_{e,v}f(v)-h_{e,u}f(u)=h_{e,v}(f(v)-h_{e}f(u)).
\end{equation}
If \eqref{horedge} holds for every path $\gamma$, it holds in particular for $\gamma=(\ul{e},e,\ol{e})$ for every edge $e$, and \eqref{compdf} shows that $d_{h}f(e)=0$. Thus, $f$ is horizontal.

Conversely, if $f$ is horizontal, then \eqref{compdf} shows that \eqref{horedge} holds for every path of the form $(\ul{e},e,\ol{e})$. Using the multiplicativity of the holonomy, one can then extend \eqref{horedge} to an arbitrary path, by induction on the number of edges which it traverses.

The last assertion is a consequence of the unitarity of the connection.
\end{proof}

An immediate consequence of this proposition is that on a connected graph, a horizontal section vanishes either nowhere on the graph, or everywhere. This can be reformulated by saying that for every vertex $v\in \rV$, the evaluation map $\ev_{v}:\Omega^{0}(\F)\to \F_{v}$ restricts to an injective map $\ev_{v}:\ker d_{h} \hookrightarrow \F_{v}$. In other words, on a connected graph, a horizontal section is entirely determined by its value at any single vertex. In particular, on a connected graph, we have the simple relation
\begin{equation}\label{eq:kerN}
\dim(\ker d_{h})\leq N.
\end{equation}

\begin{definition}
We say that a local system $(\F,h)$ over a connected graph is {\em flat} if $\dim \ker d_{h}=N$. We say that it is {\em fully twisted} if $\dim \ker d_{h}=0$.
\end{definition}

The next proposition describes the image of the map $\ev_{v}:\ker d_{h} \hookrightarrow \F_{v}$, that is, the set of possible values at $v$ of a horizontal section. More specifically, we prove that this image is the invariant holonomy submodule of $\F_{v}$.

For each vertex $v$, we denote by $(\F_{v})^{\Hol_{h,v}}$ the subspace of $\F_{v}$ that is fixed by every element of the holonomy group $\Hol_{h,v}$.

\begin{proposition}\label{prop:evkerd}
Assume that the graph $\rG$ is connected. Let $u\in\rV$ be a vertex. Then
\begin{equation}\label{eq:evkerd}
\ev_{u}(\ker d_{h})=( \F_{u})^{\Hol_{h,u}}=\bigcap_{\gamma\in \Loop_u(\rG)}\ker(\hol_h(\gamma)-\id_{\F_u})\, .
\end{equation}
\end{proposition}

\begin{proof} Applying Proposition \ref{prop:hor} to the case where $v=u$ yields the inclusion of the left-hand side in the right-hand side. 

To prove the reverse inclusion, consider an element $\xi$ of the right-hand side. For every vertex $v\in \rV$, choose a path $p$ joining $u$ to $v$ and define
\[f(v)=\hol_{h}(p) (\xi).\]
The fact that $\xi$ belongs to the right-hand side of \eqref{eq:evkerd} implies that the definition of $f(v)$ does not depend on the choice of the path $p$. By Proposition \ref{prop:hor}, the section $f$ thus defined is horizontal, and verifies $\ev_{u}(f)=\xi$.
\end{proof}

\begin{corollary} \label{cor:horizholo} Assume that $\rG$ is connected. The following assertions are equivalent.
\begin{enumerate}[\indent 1.]
\item The local system $(\F,h)$ is flat.
\item For all $u\in \rV$ and all $\gamma \in \Loop_{u}(\rG)$, one has $\hol_{h}(\gamma)=\id_{\F_{u}}$.
\item For some $v\in \rV$ and all $\gamma \in \Loop_{v}(\rG)$, one has $\hol_{h}(\gamma)=\id_{\F_{v}}$.
\end{enumerate}
\end{corollary}

\begin{proof} One has $\dim \ker d_{h}=N$ if and only if $\ev_{u}(\ker d_{h})=\F_{u}$ for all vertex $u$, and the equivalence of the first two assertions follows from the proposition. The equivalence of $2$ and $3$ is a consequence of the fact that the holonomy along any loop based at a vertex $u$ is conjugated, by the holonomy along a path from $u$ to $v$, to the holonomy along some loop based at $v$. 
\end{proof}

According to this corollary, a local system is flat if and only if the holonomy along every loop is the identity. There are situations, for example when the graph $\rG$ is embedded in a surface, where a different notion of flatness would naturally emerge: one would call flat those connections that have trivial holonomy along loops that are contractible in the surface, but not necessarily along all loops. One way of reconciling the two pictures is to think, in this paper, of the graph as embedded in a simply connected ambient space.

\begin{corollary} \label{cor:flattrivial}
A local system of rank $N\geq 1$ over $\rG$ is flat if and only if it is isomorphic to the trivial local system $(\rG\times \C^N,I_N)$ with connection identically equal to the identity.
\end{corollary}

\begin{proof} The trivial bundle with connection identically equal to the identity has trivial holonomy along any loop, and is therefore flat.

Conversely, consider a flat local system $(\F,h)$ and assume, without loss of generality, that $\rG$ is connected. Let us choose a vertex $v$ and an orthonormal basis of $\F_v$. Let us parallel transport this orthonormal basis to all other vertices, and all the fibres over edges. The flatness of the local system ensures that this is possible, and does not depend on the paths that we choose to transport the basis. We obtain a basis for each fibre of $\F$ and by construction, the matrix of the holonomy along every half-edge with respect to these bases is the identity.
\end{proof}

The intersection on the right-hand side of \eqref{eq:evkerd} can be restricted to any subset of $\Loop_{u}(\rG)$ with the property that any element of $\Loop_{u}(\rG)$ can be written as the concatenation (and reduction) of elements of this subset and their inverses. Here, reduction means that one can ignore pieces of a path that go along an edge and immediately backtrack this edge. 

Such a subset can be constructed from a spanning tree $T$ of $\rG$, as follows.  
For all vertices $v,w\in \rV$, let $[v,w]_{T}$ denote the unique simple path from $v$ to $w$ inside $T$. For every edge $e\in \rE$, let us define the loop  
\[\gamma_{T,u}(e)=[u,\ul{e}]_{T}e[\ol{e},u]_{T}.\]
Any loop on $\rG$ based at $u$ is equal, up to adding or removing some backtracking parts, to a word in the loops $\{\gamma_{T,u}(e) : e\in \rE\setminus T\}$. Therefore, 
\begin{equation}\label{eq:evkerdtree}
\ev_{u}(\ker d_{h})=\bigcap_{e\in \rE\setminus T}\ker(\hol_h(\gamma_{T,u}(e))-\id_{\F_u}).
\end{equation}

\begin{corollary}\label{coro:treeN} Every local system over an acyclic graph is flat.
\end{corollary}

\begin{proof} If $\rG$ is acyclic, then the right-hand side of \eqref{eq:evkerdtree} is empty, therefore equal to $\F_{u}$.
\end{proof}

\subsection{The canonical splitting}

\begin{proposition} \label{prop:canosplit}
Let $(\F,h)$ be a local system. There is a unique orthogonal splitting $(\F,h)=(\F',h)\oplus (\F'',h)$, where $(\F',h)$ is flat and $(\F'',h)$ is fully twisted.
\end{proposition}

\begin{proof} Let us consider such a decomposition. Every horizontal section of $\F$ must be a section of $\F'$. Furthermore, any vector of any fibre of $\F'$ extends to a horizontal section of $\F'$, and therefore of $\F$. From these observations, it follows that for every $v\in \rV$, we must have $\F'_{v}=\ev_{v}(\ker d_{h})$. This shows the uniqueness of the decomposition.

To prove its existence, let us choose a vertex $v$ and split the fibre $\F_{v}$ as $\F_{v}=\F'_{v}\oplus \F''_{v}$, where $\F'_{v}=(\F_{v})^{\Hol_{h,v}}$ is the subspace of holonomy fixed vectors of $\F_{v}$, and $\F''_{v}$ its orthogonal. Then $\F''_{v}$ is the sum of the non-trivial holonomy submodules of $\F_{v}$. By Proposition \ref{prop:localsysfibre}, this yields a decomposition of the local system $(\F,h)$ as $(\F',h)\oplus (\F'',h)$. 

By construction, the holonomy along any loop based at $u$ in the local system $(\F',h)$ is the identity. Therefore, by Corollary \ref{cor:horizholo}, $(\F',h)$ is flat. Moreover, if $f$ is a horizontal section of $\F''$, then by Proposition ~\ref{prop:evkerd}, we have $f_{u}\in \F'_{u}$, so that $f_{u}=0$, and $f$ vanishes. 
\end{proof}

\begin{corollary} \label{cor:alternative}
An irreducible local system is either flat and of rank $1$, or fully twisted.
\end{corollary}

\begin{proof} Consider an irreducible local system $(\F,h)$ that is not fully twisted. Then the flat summand of its canonical splitting has positive rank, and must be equal to $(\F,h)$ by irreducibility. Pick any non-zero vector in any fibre of $\F$ and extend this vector to a horizontal section. Consider, in the fibre over each vertex, the line generated by the value at this vertex of the horizontal section. These lines form a rank $1$ local subsystem of $(\F,h)$, which again by irreducibility must be $(\F,h)$ itself.
\end{proof}

Proposition \ref{prop:evkerd} provides us with a description of the flat part of a local system. Similarly, for all vertex $v$, the fibre at $v$ of the fully twisted part of $(\F,h)$ is  
\begin{equation}\label{eq:Hperpv}
\F''_{v}=\Big(\bigcap_{\gamma\in \Loop_v(\rG)}\ker(\hol_h(\gamma)-\id_{\F_v})\Big)^{\perp}=\sum_{\gamma\in \rL_v(\rG)}\im (\hol_h(\gamma)^{*}-\id_{\F_v})=\sum_{\gamma\in \rL_v(\rG)}\im (\hol_h(\gamma)-\id_{\F_v}),
\end{equation}
where the last equality follows from the fact that $\hol_{h}(\gamma)^{*}=\hol_{h}(\gamma^{-1})$.

\begin{proposition} \label{prop:twistimage} Let $(\F,h)$ be a local system over a connected graph. The following assertions are equivalent.
\begin{enumerate}[\indent 1.]
\item $(\F,h)$ is fully twisted.
\item For all  $v\in \rV$, the equality $\sum_{\gamma\in \rL_v(\rG)}\im (\hol_h(\gamma)-\id_{\F_v})=\F_{v}$ holds.
\item For all  $v\in \rV$, there exists a finite collection $\gamma_{1},\ldots,\gamma_{r}$ of loops based at $v$ such that
\begin{equation}\label{eq:sumimagesF}
\sum_{i=1}^{r} \im (\hol_h(\gamma_{i})-\id_{\F_v})=\F_{v}.
\end{equation}
\item For some $v\in \rV$, there exists a finite collection $\gamma_{1},\ldots,\gamma_{r}$ of loops based at $v$ such that \eqref{eq:sumimagesF} holds.
\end{enumerate}
\end{proposition}

\begin{proof} The equivalence of assertions 1 and 2 follows from \eqref{eq:Hperpv}. The equivalence of 2 and 3 follows from the fact that $\F''_v$ is finite dimensional.
The assertion 3 implies tautologically 4. Assume 4 and consider any vertex $u$. Choose a path $p$ from $u$ to $v$. Then the loops $p\gamma_1p^{-1},\ldots, p\gamma_rp^{-1}$ based at $u$ fulfill \eqref{eq:sumimagesF}.
\end{proof}

If $\rG$ is connected, it follows from the proof that the smallest value of the integer $r$ for which the fourth assertion holds at a vertex $v$ does not depend on $v$. 

Therefore, if $\rG$ is connected, the common value over all vertices of this smallest integer is attached to the twisted local system itself. It is an integer between $1$ and the rank of the bundle, and it is equal to $1$ if and only if there exists a loop $\gamma$ in $\rG$ such that $\hol_{h}(\gamma)-\id$ is invertible. It is useful to have a name for this property.

\begin{definition} We say that the local system $(\F,h)$ is {\em fully twisted along a loop} $\gamma$ if $\hol_{h}(\gamma)-\id$ is invertible. We say that it is {\em fully twisted along a loop} if there exists a loop $\gamma$ such that it is fully twisted along $\gamma$.
\end{definition}

The following example shows that a local system can be fully twisted without being fully twisted along any loop. In other words, the smallest value of the integer $r$ discussed above can be strictly greater than~$1$.

\begin{example}Assume that $\rG$ is connected with at least two independent cycles. Let $k\ge 1$ be an integer. Endow $\rG$ with a local system given by the trivial bundle $\R^{2k+1}$ and connection given by matrices in the group $SO(2k+1)$. Since any matrix in $SO(2k+1)$ has $1$ as an eigenvalue (there is an odd number of eigenvalues, and they come in conjugate pairs on the unit circle, and their product is $1$), the holonomies of all cycles of $\rG$, which are also in $SO(2k+1)$, have $1$ as an eigenvalue. However, choosing generic matrices for each edge (for instance independently uniformly chosen ones), no two distinct cycles have holonomies which share eigenvectors for the eigenvalue $1$. Therefore $\ker d_h$ is trivial, but all loops $\gamma$ are such that $\ker (\hol_h(\gamma)-\id)$ is non trivial.
\end{example}

\subsection{A twisted Poincaré lemma}

In this section, we will state a result that allows one, in a fully twisted local system, to recover a section from its differential. We start with an elementary, but fundamental lemma.

\begin{lemma}
Let $p=(e_1, \ldots, e_k)$ be a path in $\rG$ based at $u=\ul{e_1}$ and ending at $v=\ol{e_k}$. 
Then for all $f\in \Omega^0(\F)$, we have
\[\hol_{h}(p)^{-1} f(v)-f(u)=\sum_{i=1}^{k} \hol_{h}(p_{|[u,e_{i}]})^{-1} d_{h}f(e_{i}).\]
\end{lemma}

\begin{proof}
By definition, for each edge, we have $d_hf(e) = h_{e, \ol{e}} f(\ol{e}) - h_{e, \ul{e}} f(\ul{e})$.
Hence by multiplicativity of the holonomy under concatenation of paths, we have 
\[\hol_h(p_{|[u,e]})^{-1} d_h f(e)=\hol_h(p_{|[u,\ol{e}]})^{-1} f(\ol{e}) - \hol_h(p_{|[u,\ul{e}]})^{-1} f(\ul{e}).\]
Now by summing over the edges of $e\in p$, we obtain a telescoping sum, and the result follows.
\end{proof}

In particular, for a cycle $\gamma=(e_{1},\ldots,e_{k})$ based at $v$, we obtain, applying $\hol_{h}(\gamma)$ on both sides,
\begin{equation}\label{eq:holo-loop-lemma}
\left(\id_{\F_v}-\hol_h(\gamma)\right) f(v)=\sum_{i=1}^{k} \hol_h(\gamma_{|[e_{i},v]}) d_hf(e_{i})
\end{equation}

A first consequence is the following twisted Poincaré lemma, which yields $f$ in terms of $d_hf$ under the assumption that the local system is fully twisted along a loop.

\begin{proposition}[A first twisted Poincaré lemma]\label{prop:poincare-lemma}
Assume that $(\F,h)$ is fully twisted along a loop. Let $v$ be a vertex and let $\gamma=(e_{1},\ldots,e_{k})$ be a loop based at $v$ such that $\id_{\F_v}-\hol_h(\gamma)$ is invertible. Then, for all $f\in \Omega^0(\F)$, we have
\begin{equation}
    f(v)=\big(\id_{\F_v}-\hol_h(\gamma)\big)^{-1} \sum_{i=1}^{k} \hol_h(\gamma_{|[e_{i},v]}) d_hf(e_{i})\,.
\end{equation}
\end{proposition}

\begin{corollary}\label{cor:bound}
For any vertex $v$ and cycle $\gamma_{v}=(e_{1},\ldots,e_{k})$ based at $v$ such that $\id_{\F_{v}}-\hol_{h}(\gamma_{v})$ is invertible, we have for all $f\in \Omega^0(\F)$
\[\| f(v) \| \le  \; \big\| \big(\id_{\F_v}-\hol_h(\gamma_{v})\big)^{-1} \big\| \; \sum_{i=1}^{k} \lVert d_h f(e_{i})\rVert\,  . \]
\end{corollary}

Let us take note of the fact that the assumption that $\id_{\F_{v}}-\hol_{h}(\gamma_{v})$ is invertible is equivalent to the strict inequality $2\id_{\F_v}-\hol_h(\gamma_{v})-\hol_h(\gamma_{v})^*>0$ between self-dual endomorphisms of $\F_v$. 

Let us now treat the general case of a fully twisted local system that needs not be fully twisted along any loop. For this, it is convenient to introduce some notation. Given a path $p=(e_{1},\ldots,e_{k})$ joining a vertex $u$ to a vertex $v$, and given a $1$-form $\omega\in \Omega^{1}(\F)$, we define the (forward) covariant discrete integral of $\omega$ along~$p$ as
\[I_{h}^{\uparrow}(p,\omega)= \sum_{i=1}^{k} \hol_h(\gamma_{|[e_{i},v]}) d_hf(e_{i}) \in \F_{v}.\]
The downward covariant discrete integral is
\[I_{h}^{\downarrow}(p,\omega)=\sum_{i=1}^{k} \hol_{h}(p_{|[u,e_{i}]})^{-1} d_{h}f(e_{i})=\hol_{h}(p^{-1})I_{h}^{\uparrow}(p,\omega) \in \F_{u}.\]

Assume that the local system $(\F,h)$ is fully twisted. Let $v$ be a vertex of $\rG$. Let $\gamma_{1},\ldots,\gamma_{r}$ be loops based at $v$ such that 
\[\sum_{i=1}^{r} \im (\hol_h(\gamma_{i})-\id_{\F_v})=\F_{v}.\]
 For each $i\in \{1,\ldots,r\}$, set $R_{i}=\im (\hol_h(\gamma_{i})-\id_{\F_v})$ and let $\proj{i}:\F_{v}\to R_{i}$ denote the orthogonal projection. Let us define
 \[\proj{}=\bigoplus_{i=1}^{r}\proj{i}\colon \F_{v} \longrightarrow \bigoplus_{i=1}^{r} R_{i} \]
the direct sum of these orthogonal projections. 

\begin{proposition}[Twisted Poincaré lemma] \label{prop:twistagain} With the notation above, $\proj{*}\proj{}$ is an invertible endomorphism of $\F_{v}$ and for all $f\in \Omega^{0}(\F)$, we have the equality
\[f(v)=\big(\proj{*}\proj{} \big)^{-1} \bigg(\bigoplus_{i=1}^{r} 
\Big[\big((\id_{\F_v}-\hol_h(\gamma_{i}))_{|R_{i}}\big)^{-1} \proj{i} \big(I^{\uparrow}_{h}(\gamma_{i},d_{h}f)\big)\Big]
\bigg).
\]
\end{proposition}

\begin{proof} Let us start by proving that $\proj{*}\proj{}$ is invertible. The adjoint $\proj{*}$ of $\proj{}$ is the sum of the inclusion maps $R_{i}\to \F_{v}$ and our main assumption that $R_{1}+\ldots+R_{r}=\F_{v}$ is equivalent to the fact that $\proj{*}$ is onto. Therefore, $\proj{}$ is injective and $\proj{*}\proj{}$ is invertible.

For each $i\in \{1,\ldots,r\}$ let us write \eqref{eq:holo-loop-lemma} along the loop $\gamma_{i}$, then apply $\proj{i}$ on both sides, and then use the fact that $\proj{i}$ commutes with $\hol_{h}(\gamma_{i})$. We find
\[\left(\id_{\F_v}-\hol_h(\gamma_{i})\right)\proj{i} f(v)=\proj{i} \big(I^{\uparrow}_{h}(\gamma_{i},d_{h}f)\big) \in R_{i}.\]
The endomorphism $\id_{\F_v}-\hol_h(\gamma)$ restricts to an invertible endomorphism of $R_{i}$, so that
\[\proj{i} f(v)=\big((\id_{\F_v}-\hol_h(\gamma_{i}))_{|R_{i}}\big)^{-1} \proj{i} \big(I^{\uparrow}_{h}(\gamma_{i},d_{h}f)\big).\]
To recover $f(v)$ from its orthogonal projections on the subspaces $R_{1},\ldots,R_{r}$, it suffices to apply the inverse of $\proj{*}\proj{}$.
\end{proof}

\begin{corollary} \label{cor:bound2} With the current notation, and writing $\gamma_{i}=(e_{i,1},\ldots,e_{i,k_{i}})$
\[\|f(v)\|\leq \big\|\big(\proj{*}\proj{} \big)^{-1}\big\|\ \sum_{i=1}^{r} \Big(\big\|\big((\id_{\F_v}-\hol_h(\gamma_{i}))_{|R_{i}}\big)^{-1}\big\|  \ \sum_{j=1}^{k_{i}} \|d_{h}f(e_{i,j})\|\Big).\]
\end{corollary}

%%%%%%%%%%%%%%%%%%%%%%%%%%%%%%%%%%%%%%%%

\section{Determinantal linear spanning forests (DLSF)}\label{sec:DLSF}

\subsection{Determinantal linear subspaces} \label{sec:dls}

For the sake of completeness, let us recall from \cite{KL3} the definition of determinantal linear subspaces. 

Let $E$ be a finite-dimensional inner product space endowed with an orthogonal splitting $E=E^{1}\oplus \ldots \oplus E^{s}$ that we denote by $\sigma$. Let $H$ be a linear subspace of $E$. The determinantal linear subspace measure $\DLP(E,H,\sigma)$ is a probability measure on the Grassmannian manifold of $E$. More specifically, set $d=\dim E$ and $n=\dim H$. Then the measure is supported by the set $\Gr_{n}(E)$ of $n$-dimensional subspaces of $E$, and even more precisely by the set
\[\Gr_{n}(E,\sigma)=\big\{Q \in \Gr_{n}(E) : Q=(Q\cap E^{1})\oplus \ldots \oplus (Q\cap E^{s})\big\}\]
of linear $n$-subspaces adapted to $\sigma$. The connected components of $\Gr_{n}(E,\sigma)$ are indexed by the $s$-tuples of integers $\ul{n}=(n_{1},\ldots,n_{s})$ such that $n_{1}+\ldots n_{s}=n$ and $0\leq n_{i}\leq d_{i}$ for all $i\in \{1,\ldots,s\}$, where $d_{i}=\dim E_{i}$. They are the sets
\[\Gr_{\ul{n}}(E,\sigma)=\big\{Q \in \Gr_{n}(E,\sigma) : \dim(Q\cap E^{1})=n_{1},  \ldots, \dim(Q\cap E^{s})=n_{s}\big\}.\]
For each $s$-tuple $\ul{n}$, the set $\Gr_{\ul{n}}(E,\sigma)$ carries a unique Borel measure $\nu_{\ul{n}}^{E,\sigma}$ of total mass~$\binom{d_{1}}{n_{1}}\ldots \binom{d_{s}}{n_{s}}$ that is invariant under the action of the group $\U(E^{1})\times \ldots \times \U(E^{s})$. The measure
\[\nu_{n}^{E,\sigma}=\sum_{\ul{n}} \nu_{\ul{n}}^{E,\sigma}\]
is then a Borel measure on $\Gr_{n}(E,\sigma)$ of total mass $\binom{n}{d}$.

Given any two linear $n$-subspaces $F$ and $G$ of $E$, the square cosine of the angle between $F$ and $G$ can be computed by choosing orthonormal bases $(f_{1},\ldots,f_{n})$ and $(g_{1},\ldots,g_{n})$ of $F$ and~$G$ respectively, and setting
\begin{equation}\label{eq:defcos2}
\cos^{2}(F,G)=\Big|\det\big(\langle f_{i},g_{j}\rangle\big)_{1\leq i,j \leq n}\Big|^{2}.
\end{equation}
This number does not depend on the chosen orthonormal bases. 

The determinantal linear subspace measure $\DLP(E,H,\sigma)$ is the probability measure
\[\cos^{2}(Q,H) \;\nu_{n}^{E,\sigma}(\nsd Q)\]
on $\Gr_{n}(E,\sigma)$.

It can be sampled by the following algorithm:
\begin{enumerate}[\indent $\sbullet$]
\item choose uniformly at random independent orthonormal bases of $E^{1},\ldots,E^{s}$, and merging them to produce a random orthonormal basis $(e_{1},\ldots,e_{d})$ of $E$;
\item define $K$ as the matrix in the basis $(e_{1},\ldots,e_{d})$ of the orthogonal projection on $H$;
\item sample a determinantal random subset $\X$ of $\{1,\ldots,d\}$ with kernel $K$;
\item form the random linear subspace
\[\Q =\Vect(e_{i} : i\in \X).\]
\end{enumerate}
Then $\Q$ is distributed according to the determinantal linear subspace measure $\DLP(E,H,\sigma)$, and we say more simply that $\Q$ is $\DLP(E,H,\sigma)$.

\subsection{Determinantal linear spanning forests}

The main definition is the following.

\begin{definition} Let $\rG$ be a graph with Euclidean weights $\ul{x}$. Let $(\F,h,\sigma)=(\F^{1},h^{1})\oplus \ldots \oplus (\F^{r},h^{r})$ be a split local system over $\rG$. The {\em determinantal linear spanning forest} associated to $(\F,h,\sigma)$ is the determinantal linear subspace $\DLP(\Omega^{1}(\F),\im d_{h},\sigma_{1})$.
\end{definition}

To be clear, the determinantal linear spanning forest is the determinantal random linear subspace of $\Omega^{1}(\F)$ associated to the inner product induced defined by \eqref{eq:SPOmega1}, the range of the linear operator $d_{h}$ defined by \eqref{def:dh}, and the splitting $\sigma_{1}$ defined by \eqref{eq:holsplit}.

We denote by $\DLSF(\F,h,\sigma)$ the distribution of the determinantal linear spanning forest. When~$\sigma=\gamma$ is the trivial splitting given by \eqref{eq:graphicalsplitting}, we simply write $\DLSF(\F,h)$. We say that a random linear subspace $\Q$ of $\Omega^{1}(\F)$ is a $\DLSF(\F,h,\sigma)$ instead of saying that it has the distribution $\DLSF(\F,h,\sigma)$.

Let $(\F,h,\sigma)=(\F^{1},h^{1})\oplus \ldots \oplus (\F^{r},h^{r})$ be a split local system over $\rG$. Let $\sigma_{1}$ be, as before, the induced fine splitting of $\Omega^{1}(\F)$. For every linear subspace $Q$ of $\Omega^{1}(\F)$, for all edge $e$ and all integer $i\in \{1,\ldots,r\}$, we write
\[Q_{e}=Q\cap \F_{e}, \ Q^{i}=Q\cap \F^{i}  \ \text{ and } \ Q_{e}^{i}=Q\cap \F_{e}^{i}.\]
Let us accordingly introduce the vectors of integers
\[\dim_{\gamma} Q=(\dim Q_{e})_{e\in \gE}, \ \dim_{\eta} Q = (\dim Q^{i})_{1\leq i \leq r} \ \text{ and } \ \dim_{\sigma} Q=(\dim Q_{e}^{i})_{e\in \gE, 1\leq i\leq  r}\]
which contain an information finer than $\dim Q$. Indeed, if $\Q$ is a $\DLSF(\F,h,\sigma)$, then $\dim \Q$ is almost surely constant, equal to the rank of $d_{h}$. We will often call the vector $\dim_{\gamma}Q$ the {\em trace} of $Q$.

\subsection{First properties} \label{sec:firstprop}

Let us collect some of the properties of the $\DLSF(\F,h,\sigma)$ that follow directly from the general study of determinantal linear subspaces conducted in \cite{KL3}.

As before, let $(\F,h,\sigma)$ be a split local system over a finite graph $\rG$. Let $\Q$ be a $\DLSF(\F,h,\sigma)$. 

\begin{enumerate}[\indent $\sbullet$]
\item The random variable $\Q$ takes its values in $\Gr_{n}(\Omega^{1}(\F),\sigma_{1})$, where $n$ is the rank of $d_{h}$. By \eqref{eq:kerN}, we know that this rank satisfies $N(|\rV|-1)\leq n \leq N |\rV|$. We will explain later, in Section \ref{sec:stratif}, that for a generic connection $h$, it is equal to $N|\rV|$.
\item Almost surely, $\Q$ is a valid substitute to $\im d_{h}$ in the splitting $\Omega^{1}(\F)=\im d_{h} \oplus (\im d_{h})^{\perp}$, in the sense that $\Omega^{1}(\F)=\Q \oplus (\im d_{h})^{\perp}$, or equivalently, since the dimension of $Q$ is the rank of $d_h$ almost surely, $\Q\cap (\im d_{h})^{\perp}=\{0\}$.

Yet another equivalent condition is that $\proj{Q}d_h:\Omega^0(\F)\to Q$ is injective, where $\proj{Q}$ is the orthogonal projection on $Q$.

\item For every vector of integers ${\ul n}_{\sigma}=(n_{e}^{i})_{e\in \gE, 1\leq i \leq r}$, let us define the set
\[\Gr_{{\ul n}_{\sigma}}(\Omega^{1}(\F),\sigma_{1})=\{Q\in \Gr_{n}(\Omega^{1}(\F),\sigma_{1}) : \dim_{\sigma} Q = {\ul n}_{\sigma}\}.\]
The support of the distribution of $\Q$, defined as the smallest closed subset of $\Gr_{n}(\Omega^{1}(\F))$ with full measure, is equal to a union of subsets of the form $\Gr_{{\ul n}_{\sigma}}(\Omega^{1}(\F),\sigma_{1})$, namely those for which there is a positive probability that $\dim_{\sigma} \Q={\ul n}_{\sigma}$ (see \cite[Prop. 5.9]{KL3}). We will characterize these ${\ul n}_{\sigma}$ in a moment.

\item The distribution of $\Q$ depends in a non-trivial way on the splitting of the local system. However, the distribution of the trace of $\Q$ does not. More precisely, if we denote by ${\sf R}$ a $\DLSF(\F,h)$, associated to the trivial splitting of $(\F,h)$, then, according to \cite[Cor. 5.7]{KL3}, we have the equality in distribution
\[\dim_{\gamma} \Q \overset{\text{(d)}}{=} \dim_{\gamma} {\sf R}.\]
\end{enumerate}

\subsection{The case of rank $1$} \label{sec:rank1}

Consider a local system $(\F,h)$ of rank $1$. It is irreducible, hence, by Corollary~\ref{cor:alternative}, it is either flat or fully twisted. Holonomies along loops are linear isomorphisms of lines, that is, scalars, and $(\F,h)$ is flat if the holonomy along every loop is $1$, and fully twisted if there exists at least one loop along which the holonomy differs from $1$.

Another consequence of $\F$ having rank $1$ is the identification
\[\Gr(\Omega^1(\F),\sigma_1)\simeq 2^{\gE},\]
through which the DLSF can be seen as the set of edges of a random spanning subgraph of $\rG$. In what follows, for each subset $S$ of $\gE$, we set $\ul{x}^S=\prod_{\geom{e}\in S}x_e$.

If the local system is flat, this random subgraph, that we denote by $\T$, is distributed on the set of spanning trees of $\rG$ according to 
\begin{equation}\label{eq:law-st}
    \P(\T=T) \propto \ul{x}^T\,,
\end{equation}
by the famous result of Burton and Pemantle \cite{Burton-Pemantle}, see \cite[Section 6]{KL3}.
The normalizing constant is the product of nonzero eigenvalues of the Laplacian, divided by the number of vertices.

If the local system is fully twisted, this random subgraph, that we denote by $\rB$, is a random cycle-rooted spanning forest (CRSF) whose distribution is 
\begin{equation}\label{eq:law-crsf}
    \P(\rB=B) \propto \prod_{\gamma \text{ cycle of } B} \big(2-2\mathrm{Re} (\hol_h(\gamma)) \big) \, \ul{x}^B\,,
\end{equation}
where the product is over the simple cycles of the CRSF $B$ (with an arbitrarily chosen orientation). This is a result of Kenyon \cite{Kenyon}, see \cite[Section 14.2]{KL3}.
We furthermore know that the normalizing constant is the determinant of the twisted Laplacian. This statement will be generalized to higher rank in Section~\ref{sec:characteristic-polynomial}.

Note that the above also holds when the coefficient field is the skew-field of quaternions, in which case the process is a quaternionic DPP, with distribution given by \eqref{eq:law-crsf} (where the real part is now defined on the quaternions), see \cite{Kassel-ESAIM}.

\subsection{Independent splitting of a determinantal linear spanning forest} Let us prove a first property that is more specific to the graphical setting. 

\begin{proposition}\label{prop:independence} Let $(\F,h,\sigma)=(\F^{1},h^{1})\oplus \ldots \oplus (\F^{r},h^{r})$ be a split local system over $\rG$. Let $\Q$ be a $\DLSF(\F,h,\sigma)$. For all $i\in \{1,\ldots,r\}$, define
$\Q^{i}=\Q \cap \Omega^{1}(\F^{i})$.

Then $\Q^{1},\ldots,\Q^{r}$ are independent, and for all $i\in \{1,\ldots,r\}$, the random variable $\Q^{i}$ is a $\DLSF(\F^{i},h^{i})$.
\end{proposition}

For the needs of this proof, and in line with \eqref{eq:horizontalsplitting}, let us denote by $\eta=(\eta_{0},\eta_{1})$ the splitting $\Omega^{\vsbullet}(\F)=\Omega^{\vsbullet}(\F^{1})\oplus \ldots \oplus \Omega^{\vsbullet}(\F^{r})$. The splitting is less fine that $\sigma$, so that $\Q$, which is adapted to $\sigma_{1}$, is in particular adapted to~$\eta_{1}$. Therefore, $\Q=\Q^{1}\oplus \ldots \oplus \Q^{r}$.

\begin{proof} The proof consists in an observation and a computation. The observation is that the subspace $\im d_{h}$ of $\Omega^{1}(\F)$ is adapted to the splitting $\eta_{1}$. Indeed, the operator $d_{h}:\Omega^{0}(\F) \to \Omega^{1}(\F)$ is block-diagonal with respect to the splittings $\eta_{0}$ and $\eta_{1}$, the diagonal blocks being the operators $d_{h^{1}},\ldots,d_{h^{r}}$. Therefore, $\im d_{h}=\im d_{h^{1}}\oplus \ldots \oplus \im d_{h^{r}}$, where for each $i\in \{1,\ldots,r\}$, we have $\im d_{h^{i}}=\im d_{h} \cap \Omega^{1}(\F^{i})$.

The computation is the following. Let $E$ be an inner product space, and let $E=E'\oplus E''$ be a splitting of $E$. Let $F'$ and $G'$ be linear subspaces of $E'$, and $F''$ and $G''$ linear subspaces of~$E''$. Set $F=F'\oplus F''$ and $G=G'\oplus G''$. Assume that $\dim F=\dim G$. Then, choosing orthonormal bases of the spaces $F$ and $G$ adapted to their splittings, it follows directly from the definition of $\cos^{2}(F,G)$ given by \eqref{eq:defcos2} that
\[\cos^{2}(F,G)=\cos^{2}(F',G')\cos^{2}(F'',G'')\]
if $\dim F'=\dim G'$ and $\dim F''=\dim G''$, and $\cos^{2}(F,G)=0$ otherwise. 

For all $i\in \{1,\ldots,r\}$, let us define $n_{i}=\rank (d_{h^{i}})$. From the computation just made and an easy induction argument, we see that the function $Q\mapsto \cos^{2}(Q,H)$ vanishes outside $\Gr_{\ul{n}}(\Omega^{1}(\F),\eta_{1})$, and that on $\Gr_{\ul{n}}(\Omega^{1}(\F),\eta_{1})$, the following equality of measures holds:
\[\cos^{2}(Q,\im d_{h})\;\nu_{n}^{(\Omega^{1}(\F),\eta_{1})}(\nsd Q)=\bigotimes_{i=1}^{r}  \cos^{2}(Q^{i},\im d_{h^{i}}) \;\nu_{n_{i}}^{(\Omega^{1}(\F^{i}),\gamma_{1})}(\nsd Q^{i}),\]
which is the content of the proposition.
\end{proof}

\subsection{Divisibility of the trace}

Proposition \ref{prop:independence} and the property \cite[Cor. 5.7]{KL3} mentioned in Section~\ref{sec:firstprop} imply the following property of {\em divisibility} of the trace of the $\DLSF$. 

\begin{theorem} \label{thm:divtrace} Consider a split local system $(\F,h,\sigma)=(\F^{1},h^{1})\oplus \ldots \oplus (\F^{r},h^{r})$ over $\rG$. Let $\Q$ be a $\DLSF(\F,h)$. Let $\Q^{1},\ldots,\Q^{r}$ be independent $\DLSF(\F^{1},h^{1}),\ldots,\DLSF(\F^{r},h^{r})$. Then the following equality in distribution holds between $\N^{\gE}$-valued random variables:
\[\dim_{\gamma}\Q \overset{(d)}= \dim_{\gamma}\Q^{1}+\ldots + \dim_{\gamma}\Q^{r}.\]
\end{theorem}

\begin{proof} Let $\Q^{\sigma}$ be a $\DLSF(\F,h,\sigma)$. On the one hand, according to Proposition \ref{prop:independence}, we have the equality of distribution of random linear spaces
\[\Q^{\sigma}\overset{(d)}=\Q^{1}\oplus \ldots \oplus \Q^{r}.\]
On the other hand, by \cite[Cor. 5.7]{KL3}, we have the equality of distribution 
\[\dim_{\gamma}\Q\overset{(d)}=\dim_{\gamma}\Q^{\sigma}.\]
The result follows readily.
\end{proof}

Let us stress that the $\DLSF$ itself does not enjoy such a property of divisibility, as the example presented in the next section shows.

\subsection{An example}

Let us illustrate this distinction between divisibility of the trace, and not of the $\DLSF$, on an example. 

Assume that the graph $\rG$ is a triangle with three vertices $u,v,w$ and three geometric edges $[e],[f],[g]$, with $e$ joining $u$ to $v$, $f$ joining $v$ to $w$, and $g$ joining $w$ to $u$. We assume that the three edges are endowed with equal Euclidean weights.

Assume the $\F$ is the trivial real bundle of rank $2$, with all fibres equal to $\R^{2}$, and that all holonomies are the identity. The local system $(\F,h)$ splits as the sum $(\F^{1},h^{1})\oplus (\F^{2},h^{2})$ of two systems of rank $1$, each fibre being split as $\R^{2}=\R\oplus \R$. Let respectively $\Q^1$, $\Q^2$ and $\Q$ be $\DLSF(\F^{1},h^{1})$, $\DLSF(\F^{2},h^{2})$ and $\DLSF(\F,h)$, with $\Q^1$ and $\Q^2$ independent.

Let us compute the common distribution of $\Q^1$ and $\Q^2$. They are $\DLSF$ on a flat local system of rank~$1$, so that, as explained in Section \ref{sec:rank1}, and more precisely in \eqref{eq:law-st}, they are uniform spanning trees of the graph. Therefore, for each $i\in \{1,2\}$,
\[\P(\Q^i=\F^{i}_{e}\oplus \F^{i}_{f})=\P(\Q^i= \F^{i}_{f}\oplus \F^{i}_{g})=\P(\Q^i= \F^{i}_{g}\oplus \F^{i}_{e})=\frac{1}{3}.\]
In particular, the trace of $\Q^i$ is equidistributed on $\{(1,1,0),(1,0,1),(0,1,1)\}$.

\begin{figure}[h!]
\begin{center}
\includegraphics[width=.58\textwidth]{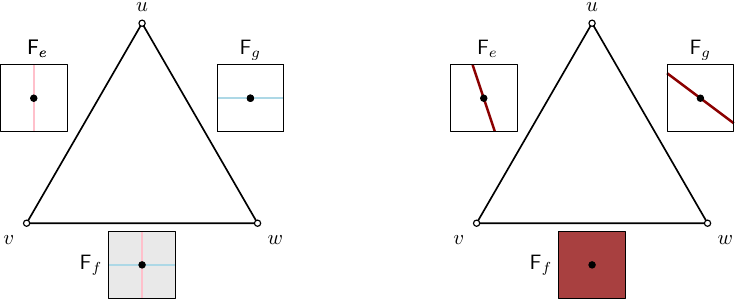}
\caption{\small On the left, in pink and blue, a joint realization of $\DLSF(\F^{1},h^{1})$ and $\DLSF(\F^{2},h^{2})$. On the right, a realization of $\DLSF(\F,h)$ with the same trace.}
\label{fig:triangle}
\end{center}
\end{figure}

Let us now describe the distribution of $\Q$. The range of $d_h$ is the $4$-dimensional space
\[\im \d_{h}=\{(x,y,z)\in \F_{e}\oplus \F_{f} \oplus \F_{g} : x+y+z=0\},\]
and
\[(\im d_{h})^{\perp}=\{(t,t,t) \in \F_{e}\oplus \F_{g}\oplus \F_{h} : t\in \R^{2}\}.\]

Since $\Q$ is an adapted $4$-dimensional subspace of $\F_e\oplus \F_f\oplus \F_g$, it can either be the direct sum of two of the three fibres, or the direct sum of one of the fibres, and two $1$-dimensional subspaces of the other two. 

By definition of a \DLP, The probability that $\Q=\F_e\oplus \F_f$ is the product of $\cos^2(\F_e\oplus \F_f, \im d_h)$, which is equal to $\frac{1}{9}$ as a result of \eqref{eq:defcos2}, by the total mass of the measure $\nu_{(2,2,0)}^{\Omega^1(\F),\gamma}$, which is equal to $1$. Therefore, 
\[\P(\Q=\F_{e}\oplus \F_{f})=\P(\Q= \F_{f}\oplus \F_{g})=\P(\Q= \F_{g}\oplus \F_{e})=\frac{1}{9}.\]

Let us now choose two unit vectors $y\in \F_{f}=\R^2$ and $z\in \F_{g}=\R^2$, set 
\[Q=\F_{e}\oplus \R y \oplus \R z,\]
and determine the density of probability that $\Q=Q$. An application of \eqref{eq:defcos2} yields
\begin{equation}\label{eq:densprobQ}
\cos^{2}(Q,\im d_{h}) = \frac{1}{9} \det(y,z)^{2}.
\end{equation}
The integral of this function over $\Gr_{(2,1,1)}(\Omega^1(\F),\sigma)$ is equal to $\frac{1}{18}$ times the mass of the measure $\nu_{(2,1,1)}^{\Omega^1(\F),\gamma}$, that is, $\frac{2}{9}$. Therefore, 
\[\P\big(\dim_\gamma \Q=(2,1,1)\big)=\P\big(\dim_\gamma \Q=(1,2,1)\big)=\P\big(\dim_\gamma \Q=(1,1,2)\big)=\frac{2}{9}\]
and the distribution of $\Q$ on each of these three events is described by \eqref{eq:densprobQ}.

We check on this example that $\dim_\gamma \Q$ has the same distribution as the sum of two independent copies of $\dim_\gamma \Q^1$ and $\dim_\gamma \Q^2$
respectively. However, $\Q$ does not have the same distribution as the direct sum of independent copies of $\Q^1$ and $\Q^2$. Indeed, for instance, on $\Gr_{(2,1,1)}(\Omega^1(\F),\sigma_1)$, the distribution of $\Q^1\oplus \Q^2$ has two atoms, of mass $\frac{1}{9}$ each, whereas, according to \eqref{eq:densprobQ} the distribution of $\Q$ is non-atomic and has full support. 

Let us emphasize that the distributions $\Q$ and the independent direct sum of $\Q^1$ and $\Q^2$ have distinct support, but they are not mutually singular.

\subsection{A special unitary example}

In the quaternionic case of rank $1$, an interesting property holds: adding two independent copies of a quaternionic CRSF yields the same trace as the rank $2$ DLSF associated with the corresponding connection in $\SU(2)$. We say that a connection is \new{special unitary} if the holonomy of any based loop is an element of the special unitary group of the fibre at which the loop is based.

\begin{proposition}\label{prop:su2}
Assume that $N=2$ and the connection is special unitary. Then the trace of the DLSF has the same law as the sum of two independent random CRSF whose common law is \eqref{eq:law-crsf}.
\end{proposition}
\begin{proof}
This is a specialization of~\cite[Proposition 6.13]{KL3}.
\end{proof}

As this example shows, there are sometimes simplifications to the law of the trace of a DLSF, but which do not immediately imply a simplification to its distribution. 
The secret of those subtleties seems to lie in the spectral properties of the holonomies of the cycles of the graph.

%%%%%%%%%%%%%%%%%%%%%%%%%%%%%%%%%%%%%%%%

\section{DLSF and polymatroids}\label{sec:DLSF-Polymatroids} 

We will describe in some detail the support of the distribution of the random vector of integers $\dim_{\gamma} \Q=(\dim \Q_{e})_{e\in \gE}$. This support is a polytope in $\N^{\gE}$ associated to a certain polymatroid. Let us briefly review the language of matroids and polymatroids. 

\subsection{Definition of polymatroids}

A \new{matroid} (see \cite{Oxley,Schrijver}) is the abstraction of the dependence relations induced by a finite configuration of vectors in a vector space. 
A pair~$\mM=(S,\cI)$ is a matroid on the finite ground set~$S$ if $\cI$ is a non-empty collection of subsets of $S$, called \new{independent} subsets, such that 
\begin{itemize}
\item if $J\subset I$ and $I\in \cI$, then $J\in \cI$,
\item if $I,J\in \cI$ and $\vert J \vert < \vert I \vert$, there exists $i\in I\setminus J$ such that $J\cup \{i\}\in \cI$.
\end{itemize}
Maximal elements of $\cI$ all have the same cardinality, called the \new{rank} of the matroid. Elements of maximal rank are called \new{bases}, and the set of bases of $\mM$ is denoted by $\B$. 

For every subset $A$ of $S$ the set of independent subsets of $S$ contained in~$A$ is a matroid on the ground set $A$, and the rank of this matroid is denoted by $\rk_{\mM}(A)$. The function 
\[\rk_\mM \colon 2^S\to \N,\ A\mapsto \max\{ |I| : I\subset A, I\in \cI \}\]
is the {\em rank function} of the matroid $\mM$ and it characterizes $\mM$ completely. Moreover, a function $\submod:2^{S}\to \N$ is the rank function of a matroid on $S$ if and only if
\begin{enumerate}[\indent 1.]
\item $\submod(\varnothing)=0$,
\item $\submod$ is non-decreasing,
\item $\submod$ is sub-modular: $\submod(A\cup B)+\submod(A\cap B) \leq \submod(A)+\submod(B)$ for all $A,B\subset S$,
\item $\submod(\{i\})\leq 1$ for all $i\in S$.
\end{enumerate}

\begin{definition} By a {\em polymatroid}, we mean a pair $\mP=(S,\submod)$, where $S$ is a finite set and $\submod:S\to \N$ a function, called the {\em rank function} of $\mP$, which satisfies the properties 1, 2 and 3 above.
\end{definition}

To a polymatroid $P=(S,\rho)$, which may well be a matroid, two convex polytopes in $\R_+^S$ are attached. Here and thereafter, by $\R_{+}^{S}$, we mean $(\R_{+})^{S}$. For all $m\in \R_{+}^{S}$ and all $T\subseteq S$, we set $m(T)=\sum_{i\in T} m(i)$.

The first polytope associated to the polymatroid $P$ is its independence polytope, defined by
\[\pPind(\submod)=\big\{m\in \R_+^S: \forall T\subseteq S, \ m(T)\le \submod(T)\big\}.\]
The second is its basis polytope, which is a face of the independence polytope, defined by 
\[\pPbas(\submod)=\{m\in \pPind(\mP): m(S)=\submod(S)\}.\]

For the sets of integer points of these polytopes, we introduce the notation
\[\nPind(\submod)=\pPind(\submod)\cap \N^S \ \text{ and} \ \nPbas(\submod)=\pPbas(\submod)\cap \N^S.\]
The polytopes $\pPind(\submod)$ and $\pPbas(\submod)$ are integer in the sense that they are respectively the convex hulls of $\nPind(\submod)$ and $\nPbas(\submod)$ (see \cite[Cor. 44.3d]{Schrijver}).

There is a natural partial order on the set of polymatroids on a common ground set. It is defined by $\mP_1=(S,\submod_{1})\preceq \mP_2=(S,\submod_{2})$ if $\submod_{1}\leq \submod_{2}$. This holds if and only if $\pPind(\submod_{1})\subseteq \pPind(\submod_{2})$. 

In the case of two matroids $\mM_{1}$ and $\mM_{2}$ with the same rank, the relation $\mM_1\preceq \mM_2$ holds if and only if $\B_1\subseteq \B_2$. In that case, $\pPbas(\mM_{1})\subseteq \pPbas(\mM_{2})$ and we say that $\mM_1$ is a \new{submatroid} of $\mM_2$.

\subsection{Linear polymatroids} \label{sec:linpol}

A fundamental example of a matroid is given by a a finite set $S$ of elements of a vector space and the collection $\cI$ of all linearly independent subsets of $S$. For such a linear matroid, the rank function associates to every subset of $S$ the dimension of its linear span. 

Similarly, a fundamental example of a polymatroid on a set $S$ is given by attaching to every element~$i$ of $S$ a finite dimensional subspace $F_{i}$ of some fixed vector space, and defining, for every subset $A$ of~$S$, the number $\submod(A)$ as the dimension of the sum $F_{A}=\sum_{i\in A}F_{i}$. That the function $\submod$ thus defined is sub-modular follows from the fact that for all $A,B\subseteq S$, one has the two relations $F_{A\cup B}=F_{A}+F_{B}$ and $F_{A\cap B}\subseteq F_{A}\cap F_{B}$, and the identity $\dim(F_{A}+F_{B})+\dim (F_{A}\cap F_{B})=\dim F_{A}+\dim F_{B}$. A polymatroid of this form is called a linear polymatroid. 

\begin{figure}[h!]
\begin{center}
{\includegraphics[width=4cm]{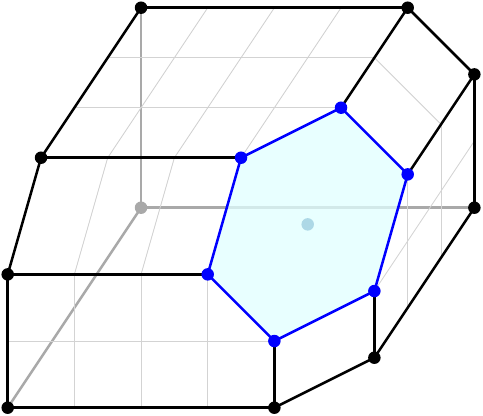}}
\caption{\small The independent polytope of the linear polymatroid on $S=\{1,2,3\}$ associated to the subspaces $F_{1}={\rm Vect}(e_{1},e_{2},e_{3},e_{4})$, $F_{2}={\rm Vect}(e_{4},e_{5},e_{6},e_{7},e_{8})$ and $F_{3}={\rm Vect}(e_{1},e_{8},e_{9})$ of $\C^{9}$. The blue face is the corresponding basis polytope, with its unique interior integral point $(3,4,2)$.}
\end{center}
\end{figure}

\subsection{Bases of linear polymatroids} For a linear matroid $M=(S,\cI)$, corresponding to a family of vectors $(v_{i} : i\in S)$ of some vector space, the independent sets (resp. bases) are the sets $I\subseteq S$ such that $\{v_{i} : i \in I\}$ is linearly independent (resp. a basis of ${\rm Vect}(v_{i} : i\in S)$). Let us give an analogous picture for linear polymatroids. 

Let $(F_{i}:i\in S)$ be a family of finite-dimensional subspaces of a vector space. We already explained that the function $\submod:2^{S}\to \N$ defined by $\submod(A)=\dim \sum_{i\in A}F_{i}$ is sub-modular, and we call $(S,\submod)$ the linear polymatroid associated to $(F_{i}:i\in S)$.

\begin{proposition}\label{prop:meuse} Let $(S,\submod)$ be the linear polymatroid associated to a family $(F_{i}:i\in S)$ of finite-dimensional subspaces of a vector space.
\begin{enumerate}[\indent 1.]
\item A point $(n_{i})_{i\in S}\in \N^{S}$ belongs to $\pPind(\submod)$ if and only if for every $i\in S$ there exists $n_{i}$ vectors $v_{i,1},\ldots,v_{i,n_{i}}$ of $F_{i}$ such that the family $(v_{i,j})_{i \in S, 1\leq j \leq n_{i}}$ is linearly independent.

\item A point $(n_{i})_{i\in S}\in \N^{S}$ belongs to $\pPbas(\submod)$ if and only if for every $i\in S$ there exists $n_{i}$ vectors $v_{i,1},\ldots,v_{i,n_{i}}$ of $F_{i}$ such that the family $(v_{i,j})_{i \in S, 1\leq j \leq n_{i}}$ is a basis of $\sum_{i\in S} F_{i}$.

\item The previous statements remain true if for every $i\in S$ a basis $B_{i}$ of $F_{i}$ is fixed in advance, and each vector $v_{i,j}$ is constrained to belong to $B_{i}$.
\end{enumerate}
\end{proposition}

This proposition is well known, and seems to be an elementary result in linear algebra. Nevertheless, we are not aware of a short direct proof. In \cite{KL3}, we referred to the proof given, with a slightly different wording, in \cite[Proposition 1.1]{Lafforgue}. For the sake of completeness, we give here another proof, based on a classical theorem of Hall and Rado about matroids, that we start by recalling.

Let $S$ be a finite set and $\cA=(A_v:v\in V)$ a collection of subsets of $S$.
We say that a subset $J$ of $S$ is \emph{transverse} to $\cA$ if there is a bijection $\psi:V\to J$ such that for all $v\in V$, we have $\psi(v)\in A_v$.

For each $W\subseteq V$, we let $A_W=\cup_{v\in W} A_v$. Let $\mM=(S,\cI)$ be a matroid on $S$, whose rank function is denoted $\rk_\mM$. 

Then Rado's theorem (see \cite[Theorem 11.2.2]{Oxley}) states that there exists an element of $\cI$ which is transverse to $\cA$ if and only if for all $W\subseteq V$, we have
\begin{equation}
{\rk}_M(A_W)\ge \vert W\vert\,.
\end{equation}

\begin{proof}[Proof of Proposition \ref{prop:meuse}] 
Let us consider $n=(n_{i})_{i\in S}\in \N^{S}$. We will first prove that if there exists a linearly independent family $(v_{i,j})$ as in the first assertion, then $n$ belongs to $\pPind(\rho)$. Then, we will fix a basis $B_{i}$ of each space $F_{i}$ and prove that if $n$ belongs to $\pPind(\rho)$, then it is possible to choose a linearly independent family $(v_{i,j})$ with $v_{i,j}\in B_{i}$ for all $i,j$. All the assertions of the proposition follow immediately from these two.
\smallskip

$\sbullet$ Let us start and assume that there exists a linearly independent family $(v_{i,j})$ as in the first assertion. Then for every subset $T$ of $S$,
\[n(T)=\sum_{i\in T} n_{i}=|\{v_{i,j} : i\in T, 1\leq j \leq n_{i}\}|\leq \dim \sum_{i\in T} F_{i}=\submod(T).\]
Therefore, $n$ belongs to $\pPind(\submod)$.
\smallskip

$\sbullet$ Let us now assume that $n$ belongs to $\pPind(\submod)$. Let us fix, for each $i\in S$, a basis $B_{i}$ of $F_{i}$. Let us form the disjoint union $B=\bigsqcup_{i\in S}B_{i}$. The set $B$ naturally indexes a family of vectors and we denote by $M=(B,\cI)$ the corresponding linear matroid on $B$. 

Let us now define the set $V=\{(i,j) : i\in S, 1\leq j \leq n_{i}\}$ and consider the family
\[\cA=(A_{v}: v\in V)=(B_{i} : (i,j)\in V)\]
of subsets of $B$, in which each basis $B_{i}$ appears $n_{i}$ times.

We claim that the family $\cA$ satisfies the assumption of Rado's theorem for the matroid $M$. Indeed, let us consider $W\subseteq V$ and define $T=\{i\in S : \exists j\in \{1,\ldots,n_{i}\}, (i,j)\in W\}$. Then, on the one hand,
\[\rk_{M}(A_{W})=\dim {\rm Vect}\bigcup_{(i,j)\in W} B_{i}=\dim \sum_{(i,j)\in W} F_{i}=\dim \sum_{i\in T} F_{i}=\submod(T).\]
On the other hand, 
\[|W| \leq \sum_{i\in T} n_{i}\]
and the assumption that $n\in \pPind(\submod)$ allows us to conclude that $\rk_{M}(A_{W})\geq |W|$.

Therefore, by Rado's theorem, there exists a subset $J$ of $B$ that belongs to $\cI$ and is transverse to~$\cA$. Since $J$ is transverse to $\cA$, it contains $n_{i}$ elements of $B_{i}$ for each $i\in S$ and can be enumerated as $J=\{v_{i,j} : i\in S, 1\leq j \leq n_{i}\}$, in such a way that $v_{i,j}\in F_{i}$ for all $(i,j)$. Moreover, since $J$ belongs to $\cI$, the family $(v_{i,j})$ is linearly independent. It has therefore all the desired properties.
\end{proof}

\subsection{The polymatroid of a local system} \label{sec:polyloc}

Let $(\F,h)$ be a local system over a finite graph $\rG$. For every subset $A$ of~$\gE$, let us define
\[\Omega^{0}(A)=\bigoplus_{v\in \rV(A)} \F_{v} \ \text{ and }\   \F_{A}=\Omega^{1}(A)=\bigoplus_{e\in A} \F_{e}.\]
The adjoint $d_{h}^{*}:\Omega^{1}(\F)\to \Omega^{0}(\F)$ of the covariant differential and the orthogonal projection $\proj{\im d_{h}}:\Omega^{1}(\F)\to \Omega^{1}(\F)$ on $\im d_{h}$ are distinct linear maps but they share the same kernel. Therefore, for every $A\subseteq \gE$, the spaces $d_{h}^{*}(\F_{A})$ and $\proj{\im d_{h}}(\F_{A})$ have the same dimension. We define  a function $\submod_{h}:2^{\gE} \to \N$ by setting, for all $A\subseteq \gE$,
\[\submod_{h}(A)=\dim d_{h}^{*}(\F_{A})=\dim \proj{\im d_{h}}(\F_{A}).\]
This function $\submod_{h}$ is the rank function of the linear polymatroid associated with the linear subspaces $(d_{h}^{*}(\F_{e}) : e\in \gE)$ of $\Omega^{0}(\F)$. We set $P_{h}=(\gE,\submod_{h})$, and say that this polymatroid is associated with the local system $(\F,h)$.

\begin{proposition} \label{prop:supporttrace} The support of the trace of a $\DLSF(\F,h)$ is the set of integer points of $\pPbas(\submod_{h})$.
\end{proposition}

\begin{proof} Let $\Q$ be a $\DLSF(\F,h)$. A point $\ul{n}=(n_{e})_{e\in \gE}\in \N^{\gE}$ is in the support of the distribution of the random integer vector $\dim_{\gamma}\Q$ if and only if $\sum_{e\in \gE} n_{e}=\dim \im d_{h}$ and there exists a point $Q=\bigoplus_{e\in \gE}Q_{e}$ of $\prod_{e\in \gE} \Gr_{n_{e}}(\F_{e})$ such that $\cos^{2}(Q,\im d_{h})>0$, that is, such that $Q\cap (\im d_{h})^{\perp}=\{0\}$, or equivalently such that $Q\cap \ker d_{h}^{*}=\{0\}$. This in turn is equivalent to the fact that $d_{h}^{*}$ sets a linear isomorphism between $Q$ and $\im d_{h}^{*}$.

Thus, $\ul{n}$ is in the support of the distribution of $\dim_{\gamma}\Q$ if and only if it is possible, for each edge $e\in \gE$, to find $n_{e}$ vectors $u_{e,1},\ldots,u_{e,n_{e}}$ in $\F_{e}$ such that the family $(v_{e,i})=(d_{h}^{*}u_{e,i})$ is a basis of $\im d_{h}^{*}$. 

According to Proposition \ref{prop:meuse}, this is equivalent to the fact that $\ul{n}$ belongs to the basis polytope of the polymatroid $\rho_{h}$.
\end{proof}

Let us give an expression of the rank function $\submod_{h}$. Let us consider a subset $A$ of $\gE$. This subset induces a subgraph $\rG^{A}=(\rV(A),A)$ of $\rG$, and we think of $\rG^{A}$ as a graph in its own right. We denote by $d_{h}^{A}:\Omega^{0}(A)\to \Omega^{1}(A)$ the covariant differential on $\rG^{A}$ and by $(d_{h}^{A})^{*}$ its adjoint. 

We will use the fact that the restriction of ${d_{h}^{*}}$ to $\Omega^{1}(A)$ sends $\Omega^{1}(A)$ to $\Omega^{0}(A)$ and, seen as a linear map from $\Omega^{1}(A)$ to $\Omega^{0}(A)$, this restriction coincides with $(d_{h}^{A})^{*}$.

By contrast, note that the restriction of $d_{h}$ to $\Omega^{0}(A)$ does not send $\Omega^{0}(A)$ to $\Omega^{1}(A)$ in general, but this will not be a problem. 

\begin{theorem}\label{theo:computerank} For every subset $A$ of $\gE$, we have
\begin{equation}\label{eq:calcrhoh}
\submod_{h}(A)=\rk (d_{h}^{A})=N |\rV(A)|- \dim \ker (d_{h}^{A}).
\end{equation}
\end{theorem}

\begin{proof} By definition of the polymatroid $\submod_{h}$, we have
\[\submod_{h}(A)=\dim d_{h}^{*}(\Omega^{1}(A))=\dim \res{d_{h}^{*}}{\Omega^{1}(A)}(\Omega^{1}(A)).\]
Thus, by the observation made just before the statement of the proposition, 
\[\submod_{h}(A)=\dim (d_{h}^{A})^{*}(\Omega^{1}(A))=\rk((d_{h}^{A})^{*}).\]
We now use the fact that adjoint linear maps have the same rank and conclude that 
\[\submod_{h}(A)=\rk (d_{h}^{A}) =\dim \Omega^{0}(A)-\dim \ker (d_{h}^{A})=N|\rV(A)|-\dim \ker (d_{h}^{A}),\]
which establishes the equality. 
\end{proof}

Recall that \eqref{eq:evkerdtree} gives a fairly concrete way of computing the dimension of $\ker (d_{h}^{A})$ when $\rG^{A}$ is connected, as the dimension of the common eigenspace for the eigenvalue $1$ of the holonomies along the loops $\gamma_{T}(e)$, based at a fixed vertex of $\rG$, where $T$ is a spanning tree of $A$ and $e$ spans $A\setminus T$. 

Let us define the numbers
\begin{itemize}
\item $c_{0}(A)=|\rV(A)|$, the number of vertices of this subgraph,
\item $b_{0}(A)$ the number of its connected components,
\item $b_{0}'(A)$ the number of its acyclic connected components.
\end{itemize}

\begin{proposition} \label{prop:rhocircrho}
For every subset $A$ of $\gE$,
\begin{equation}\label{eq:circrhoh}
N (c_{0}(A)-b_{0}(A))\leq \submod_{h}(A)\leq N (c_{0}(A)-b'_{0}(A)).
\end{equation}
\end{proposition}

\begin{proof}  We deduce the inequalities from Theorem \ref{theo:computerank}, using three facts. The first is that the kernel of $d_{h}^{A}$ is the direct sum of the kernels of the restrictions of $d_{h}^{A}$ to the connected components of $\rG^{A}$. The second fact is that on any connected graph, according to \eqref{eq:kerN}, the kernel of the covariant derivative has dimension at most $N$. Therefore, $\dim \ker (d_{h}^{A}) \leq N b_{0}(A)$ and the lower bound follows. 

To prove the upper bound, we use a third fact, namely that on a tree, according to Corollary \ref{coro:treeN}, the kernel of the covariant derivative has exactly dimension $N$. Therefore, $\dim \ker (d_{h}^{A}) \geq N b'_{0}(A)$ and the upper bound follows.
\end{proof}

\subsection{Addition of polymatroids}

Given two polymatroids $P_{1}=(S,\submod_{1})$ and $P_{2}=(S,\submod_{2})$ on the same ground set, one can form the polymatroid $P=(S,\submod_{1}+\submod_{2})$. The inclusion $\pPbas(\mP_{1})+\pPbas(\mP_{2})\subseteq \pPbas(\mP)$ is readily checked, and it is in fact an equality. 

\begin{proposition}\label{prop:Schrijver}
The  basis polytope of the sum of two polymatroids is the Minkowski sum of their basis polytopes.
\end{proposition}

\begin{proof} By \cite[Theorem 44.6]{Schrijver}, the equality $\pPind(\mP_{1})+\pPind(\mP_{2})= \pPind(\mP)$ holds. Let us consider a vector $v\in \pPbas(\mP)$ and write it as $v=v_{1}+v_{2}$ with $v_{1}\in \pPind(\mP_{1})$ and $v_{2}\in \pPind(\mP_{2})$. Then from the relations $v_{1}(S)\leq \submod_{1}(S)$, $v_{2}(S)\leq \submod_{2}(S)$ and $v_{1}(S)+v_{2}(S)=v(S)=(\submod_{1}+\submod_{2})(S)$, we deduce that $v_{1}(S)=\submod_{1}(S)$ and $v_{2}(S)=\submod_{2}(S)$, that is, $v_{1}\in \pPbas(\submod_{1})$ and $v_{2}\in \pPbas(\submod_{2})$.
\end{proof}

In particular, if $\submod$ is a polymatroid on $S$, and $k\geq 1$ an integer, then the function $k\submod$ is also a polymatroid on $S$, and the real polytope $\pPbas(k\submod)$ is the dilation by a factor $k$ of the real polytope $\pPbas(\submod)$. However, the same equality does not hold for integer points but, in general, only the inclusion 
\[k \nPbas(k\submod)  \subseteq \nPbas(k\submod).\]
 The correct equality for integer points is 
\[\nPbas(k\submod)=\underbrace{\nPbas(\submod)+\ldots +\nPbas(\submod)}_{k \text{ times}},\]
the set of sums of $k$ elements of $\nPbas(\submod)$. 

\begin{corollary} Let $(\F,h,\sigma)=(\F^{1},h^{1})\oplus \ldots \oplus (\F^{r},h^{r})$ be a split local system. Then the support of the trace of a $\DLSF(F,h)$ is 
\[\pPbas(\submod_{h})=\pPbas(\submod_{h^{1}}+\ldots+\submod_{h^{r}})=\pPbas(\submod_{h^{1}})+\ldots+\pPbas(\submod_{h^{r}}).\] 
\end{corollary}

\begin{proof} For every subset $A$ of $\gE$, we have
\begin{align*}
\submod_{h}(A)=\dim d_{h}^{*}(\F_{A})&=\dim(d_{h^{1}}^{*}(\F^{1}_{A})\oplus \ldots \oplus d_{h^{r}}^{*}(\F^{r}_{A}))\\
&=\dim d_{h^{1}}^{*}(\F^{1}_{A})+\ldots+\dim d_{h^{r}}^{*}(\F^{r}_{A})=\submod_{h^{1}}(A)+\ldots + \submod_{h^{r}}(A)
\end{align*}
and the result follows from Proposition \ref{prop:Schrijver}.
\end{proof}

This corollary also follows from Theorem \ref{thm:divtrace} and Proposition \ref{prop:supporttrace}, because the support of the distribution of a sum of independent random variables (in this case with values in $\N^{\gE}$) is the sum of the supports of each individual random variable. 

\subsection{Circular and bicircular matroids} 

On the set of edges of the finite connected graph $\rG$, there are two matroids that play a special role, of which the rank functions already appeared in Proposition \ref{prop:rhocircrho}.

The first is the circular matroid, of which the independent sets are the acyclic sets of edges, that is, the forests. The bases of this matroid are the spanning trees of $\rG$. We denote its rank function by $\submod_{\o}:2^{\gE}\to \N$. It is given, for every subset $A\subseteq \gE$, by
\begin{equation}\label{eq:rankcirc}
\submod_{\o}(A)=c_{0}(A)-b_{0}(A).
\end{equation}
We denote by $\cT(\rG)$ the set of all spanning trees of $\rG$.

The second matroid on $\gE$ is the bicircular matroid, of which the independent sets are the sets of edges which do not contain a connected subgraph with two independent cycles. If $\rG$ is a tree, this matroid is equal to the circular matroid. Otherwise, if $\rG$ has at least one cycle, the bases of the bicircular matroid are the cycle-rooted spanning forests (CRSF) of $\rG$, namely the subgraphs of which each connected component contains exactly one cycle. We denote the rank function of this matroid by $\submod_{\oo}:2^{\gE}\to \N$. It is given, for every subset $A\subseteq \gE$, by
\begin{equation}\label{eq:rankbicirc}
\submod_{\oo}(A)=c_{0}(A)-b'_{0}(A).
\end{equation}
We denote by $\B(\rG)$ the set of all CRSFs of $\rG$.

Proposition \ref{prop:rhocircrho} can be rewritten as
\begin{equation}\label{eq:itisjustthat}
N \submod_{\o}\leq \submod_{h}\leq N \submod_{\oo},
\end{equation}
and we reinterpret these inequalities using Proposition \ref{prop:Schrijver}, as follows.

\begin{theorem} \label{thm:maininclusion}
Let $(\F,h)$ be a local system of rank $N$ over $\rG$. The following inclusions of convex polytopes hold:
\begin{equation}
\pPind(N\submod_{\o})\subseteq  \pPind(\submod_{h})\subseteq \pPind(N\submod_{\oo}).
\end{equation}
\end{theorem}

Let us emphasize that the same inclusions for the basis polytopes do not hold in general, for the ranks of the polymatroids $N\submod_{\o}$, $\submod_{h}$, and $N\submod_{\oo}$ need not agree.

\subsection{Support of the distribution of the $\DLSF$}

Let us gather the previous results and give as precise as possible a description of the support of the distribution of a $\DLSF$. 

The starting point to define a $\DLSF$ is a split local system, but by Proposition \ref{prop:independence}, we lose nothing by assuming that we are working with a local system $(\F,h)$ with the trivial splitting. 

By \cite[Prop. 5.9]{KL3}, the support of the distribution of the $\DLSF$ on $(\F,h)$ is completely characterised by the support of the distribution of its trace. 

By \cite[Cor. 5.7]{KL3}, the distribution of the trace of the $\DLSF$ is not changed by changing the splitting of $(\F,h)$, and we find it convenient to replace the trivial splitting by a splitting into irreducible local systems. 

By Proposition \ref{prop:independence} again, it is enough to study each irreducible local subsystem separately. Altogether, we have reduced the problem to the case where $(\F,h)$ is an irreducible local system. Recall from Corollary~\ref{cor:alternative} that such a local system is either flat of rank $1$, or fully twisted.

\begin{proposition} \label{prop:polydirred}
Let $(\F,h)$ be an irreducible local system of rank $N$ over $\rG$.
\begin{enumerate}[\indent 1.]
\item If $(\F,h)$ is flat, of rank $1$, then $\submod_{h}=\submod_{\o}$ and $\pPbas(\submod_{h})=\pPbas(\submod_{\o})$.
\item If $(\F,h)$ is fully twisted, then $\submod_{h}\leq N\submod_{\oo}$, and $\submod_{h}(\gE)=N\submod_{\oo}(\gE)$, and $\pPbas(\submod_{h})\subseteq \pPbas(N\submod_{\oo})$.
\end{enumerate}
\end{proposition}

\begin{proof} 1. We know from Section \ref{sec:rank1} that $\DLSF(\F,h)$ is a probability of full support on the set of spanning trees of $\rG$, which is the set of bases of the circular matroid $\submod_\o$. Besides, by Proposition~\ref{prop:supporttrace}, this support is $\pPbas(\submod_{h})$. Therefore, $\submod_h=\submod_\o$. 

This equality can also be checked directly from \eqref{eq:calcrhoh} and \eqref{eq:rankcirc}.

2. The first inequality is true by Proposition \ref{prop:rhocircrho}. The second equality follows from Theorem~\ref{theo:computerank} and the fact that $\ker d_{h}=\{0\}$. Now, by definition of the independence polytope, we have the inclusion $\pPind(\submod_{h})\subseteq \pPind(N\submod_{\oo})$. Since the polymatroids $\submod_{h}$ and $N\submod_{\oo}$ have the same rank, the same inclusion holds for their basis polytopes. 
\end{proof}

The last inclusion of the second assertion is an equality for almost all connections $h$, in a sense made precise by Proposition \ref{prop:negligible} below. However, it is not always an equality, as one sees already in rank $1$. Indeed, in this case, the integer points of $\pPbas(\submod_h)$ are the indicator functions of CRSFs, all of whose cycles have holonomy different from $1$, and even a fully twisted connection can have holonomy equal to $1$ along some cycles. Proposition \ref{prop:crit=} below gives a criterion of equality.

Let us summarize this discussion.

\begin{theorem}\label{thm:superposition} Let $(\F,h,\sigma)$ be a split local system of rank $N$ over $\rG$. Set $k=N-\dim \ker d_{h}$. Then
\[\pPbas(\submod_h) \subseteq (N-k) \pPbas( \submod_\o) + k \pPbas( \submod_{\oo})=\pPbas((N-k)\submod_\o+k\submod_{\oo}).\]
In particular, almost surely, $\dim_\gamma \Q$, as a vector of $\set{N}^{\gE}$, is the sum of $N-k$ indicators of spanning trees and $k$ indicators of CRSF.
\end{theorem}

\begin{proof} This follows from Proposition \ref{prop:polydirred} and the fact that both sides of the inclusion are additive, in the sense of the Minkowski sum of polytopes, with respect to the direct sum of local systems.
\end{proof}

This theorem raises at least two natural questions. Is it possible, given a $\DLSF$, to find spanning trees and CRSFs, the superposition of which give rise to the same trace? We expect an application of Edmonds' greedy algorithm to give a positive answer to this question.

Conversely, is it possible to construct in a sensible way a $\DLSF$ from spanning trees, CRSFs, and possibly some extra information? 

\subsection{Atoms} \label{sec:atoms}

The distribution of the $\DLSF(\F,h)$ is, in general, a mixture of an atomic and a continuous part. In this paragraph, we give a full description of the atomic part. 

\begin{theorem}\label{thm:atoms}
Let $(\F,h)$ be a local system over the connected graph $(\rG,\ul{x})$ and $\Q$ be a $\DLSF(\F,h)$. The distribution of $\Q$ admits atoms if and only if $(\F,h)$ is either flat, or fully twisted along a simple loop. 

If $(\F,h)$ is flat, then the atoms are the sums $\F_T=\bigoplus_{e\in T} \F_{e}$, where $T$ is a spanning tree of $\rG$, and
\begin{equation}
    \P\big(\Q= \F_T\big) = \bigg(\frac{\ul{x}^T}{\T(\ul{x})}\bigg)^{N} \, ,
\end{equation}
where $\T(\ul{x})=\sum_{T\in \cT(\rG)} \ul{x}^T$ is the generating polynomial of spanning trees. 

If $(\F,h)$ is fully twisted, then the atoms are of the form
$\F_B=\bigoplus_{e\in B} \F_{e}$ for $B$ a CRSF of $\rG$, and for all such $B$, 
\begin{equation}\label{eq:law-atoms}
    \P\big(\Q= \F_B \big)= \frac{1}{\det \Delta_h} \prod_{\gamma \ \text{\rm cycle of } B} \! \big|\det(\id-\hol_h(\gamma))\big|^2 \  (\ul{x}^{B})^{N} \,.
\end{equation}
\end{theorem}

Note that there are examples of local systems which are not fully twisted along a simple loop, but which are fully twisted along a loop: take a graph with one vertex and two self-loops, with respective holonomies which have $1$ as eigenvalue, but not their product.  In this case, the local system is fully twisted along the product of the two self-loops. In such a situation, it follows from the theorem that the $\DLSF$ has no atom.

\begin{proof} According to its definition recalled in Section \ref{sec:dls}, the distribution of a $\DLP$ on a split inner product space $(E,\sigma)$ is supported by a disjoint union of Cartesian products of Grassmannians, and it has a continuous density with respect to the invariant measure that we denoted by $\nu_n^{E,\sigma}$. Therefore, it can only have atoms on connected components that are points. 

In our situation, the isolated points of $\Gr(\Omega^1(\F),\sigma)$ are the subspaces
\[\F_A=\bigoplus_{e\in A} \F_e, \ \ A \subseteq \gE,\]
and the atoms of the $\DLSF(\F,h)$ must be of this form. 

By Theorem \ref{thm:superposition}, any $A$ such that $\F_A$ is an atom of $\DLSF(\F,h)$ must satisfy
\[N \1_A=\1_{T_1}+\ldots+\1_{T_{N-k}}+\1_{B_1}+\ldots+\1_{B_k}\]
for some spanning trees $T_1,\ldots,T_{N-k}$ and CRSFs $B_1,\ldots,B_k$, and with $k=N-\dim \ker d_h$.

The equality implies in particular that $N|A|=(N-k)(|\rV|-1)+k|\rV|=N|\rV|-k$, so that either $k=0$, and $(\F,h)$ is flat, or $k=N$, and $(\F,h)$ is fully twisted. 

If $k=0$, then the function $N\1_A=\1_{T_1}+\ldots+\1_{T_{N}}$ takes only the values $0$ and $N$, so that $T_2,\ldots,T_N$ must contain $T_1$, hence $T_1=\ldots=T_N$, and $A$ is this spanning tree. 

When $k=N$, the same argument shows that $A$ is a CRSF.

Let us assume that $(\F,h)$ is flat, choose a spanning tree $T$ and compute $\P(\Q=\F_T)$. By Corollary \ref{cor:flattrivial}, since $(\F,h)$ is flat, it is isomorphic to the trivial bundle $\rG\times \C^N$ with the connection identically equal to the identity. It is therefore isomorphic to the direct sum of $N$ copies of the trivial bundle of rank $1$ with trivial holonomies. We denote by $\sigma$ the corresponding splitting of $(\rG\times \C^N,I_N)=(\rG\times \C,1)\oplus \ldots \oplus (\rG\times \C,1)$. We can now do the computation:
\begin{align*}
\P_{(\F,h)}(\Q=\F_T)&=\P_{(\rG\times \C^N,I_N)}(\Q=\F_T)& \text{by the isomorphism of local systems}\\
&=\P_{(\rG\times \C^N,I_N)}(\dim_\gamma \Q=N\1_T) & \text{$\F_T$ is characterized by $\dim_\gamma \F_T=N\1_T$}\\
&=\P_{(\rG\times \C^N,I_N,\sigma)}(\dim_\gamma \Q=N\1_T)
& \text{by \cite[Cor. 5.7]{KL3}}
\\
&=\P_{(\rG\times \C,1)}(\dim_\gamma \Q=\1_T)^N & \text{by Theorem \ref{thm:divtrace}}\\
&=\bigg(\frac{\ul{x}^T}{\T(\ul{x})}\bigg)^{N} & \text{by \eqref{eq:law-st}.}
\end{align*}

Let us now turn to the fully twisted case. Let $B$ be a CRSF of $\rG$. By definition of the $\DLSF$ and by \cite[Prop. 4.6 (3)]{KL3}, we have
\[\P(\Q=\F_B)=\cos^2(\F_B,\im d_h)=\det({\sf 1}^{\F_B} \proj{\im d_h} {\sf 1}_{\F_B}),\]
where we borrow from \cite[Section 2.4]{KL3} the notation ${\sf 1}_{\F_B}$ for the inclusion $\F_B\to \Omega^1(\F)$ and ${\sf 1}^{\F_B}:\Omega^1(\F)\to \F_B$ for the orthogonal projection. 

Expressing the projection in terms of the inverse of the Laplacian, we find
\[\P(\Q=\F_B)=\det({\sf 1}^{\F_B} d_h \Delta_h^{-1} d_h^* {\sf 1}_{\F_B})=\det(d_h^* {\sf 1}_{\F_B} {\sf 1}^{\F_B} d_h )/\det(\Delta_h)=\det(d_h^* \proj{\F_B} d_h )/\det(\Delta_h).\]
The numerator is the determinant of the Laplacian $\Delta^B_h$ on the local system induced by $(\F,h)$ on the graph $\rG^B=(\rV,B)$.

This Laplacian factorizes over the connected components of $\rG^B$. Let us consider one connected component $C$ of $B$, with simple cycle $\gamma_C$, oriented in an arbitrary way. By Proposition \ref{prop:unicycrang1}, the local system over $C$ splits into a direct sum of irreducible local systems of rank $1$. On each of them, the holonomy $z$ along $\gamma_C$ is one of the eigenvalues of $\hol_h(\gamma_C)$, and, according to \cite{Forman}, the determinant of the corresponding Laplacian is $\ul{x}^{C} |1-z|^2$. 

Collecting all the rank $1$ contributions, we find that 
\[\P(\Q=\F_B)=\frac{1}{\det \Delta_h} \prod_{C } \prod_{z} \big(\ul{x}^C |1-z|^2\big),\]
from which \eqref{eq:law-atoms} follows immediately.

It follows from this explicit expression that if the distribution of $\Q$ has an atom, then there exists a loop $\gamma$ along which $\id -\hol_h(\gamma)$ is invertible. Thus, $(\F,h)$ is fully twisted along $\gamma$.

Conversely, if $(\F,h)$ is twisted along the simple loop $\gamma$, then there exists a connected CRSF~$B$, that is, a spanning unicycle, with unique cycle $\gamma$, and $\P(\Q=\F_B)>0$. 
\end{proof}

Recall that $\B(\rG)$ denotes the set of all CRSFs of $\rG$. From Theorem \ref{thm:atoms} it follows that the total mass of the atoms is
\begin{itemize}
\item $\displaystyle \T(\ul{x})^{-N} \sum_{T\in \cT(\rG)} (\ul{x}^T)^N$ if $(\F,h)$ is flat,
\item $\displaystyle \det \Delta_h^{-1}
\sum_{B\in \B(\rG)} (\ul{x}^B)^N \prod_{\gamma \ \text{\rm cycle of } B} \vert \det(\id-\hol_h(\gamma))\vert^2$ if $(\F,h)$ is fully twisted along a simple loop,
\end{itemize}
and $0$ if $(\F,h)$ is neither flat nor fully twisted along a simple loop.

An interesting consequence of Theorem \ref{thm:atoms} is the following. If $(\F,h)$ is fully twisted along a simple loop, there is a probability distribution on $\B(\rG)$ which assigns each CRSF $B$ a probability proportional to 
\begin{equation}\label{eq:law-crsf-n}
    \P(\rB=B) \propto (\ul{x}^B)^N \, \prod_{\gamma \ \text{\rm cycle of } B} \vert \det(\id-\hol_h(\gamma))\vert^2\,.
\end{equation}
It is the distribution of the $\DLSF$ conditioned to take its values in its set of atoms.  
This may allow to deduce geometric information about uniform CRSF using higher rank bundles, in the spirit of \cite{Kenyon, Kassel-Kenyon} which were restricted to probability distributions on CRSF twisted by rank $1$ local systems.

\subsection{Stratification of the set of connections} \label{sec:stratif}

Over the graph $\rG$, let us consider a vector bundle $\F$, without choosing a connection. The set of polymatroids on the ground set $\gE$ is partially ordered as a subset of $\N^{\gE}$ and we proved that for every connection $h\in \A(\F)$, the inequality $N\submod_{\o} \leq \submod_{h} \leq N \submod_{\oo}$ holds. This inequality means that $\submod_{h}$ belongs to the interval $[N\submod_{\o},N\submod_{\oo}]$ of the poset of polymatroids on the ground set $\gE$. There is therefore a partition of the set of all unitary connections on $\F$ according to their associated polymatroid:
\[\A(\F)=\bigsqcup_{N\submod_{\o} \leq \submod \leq N\submod_{\oo}} \A_{\submod}(\F),\]
with $\A_{\submod}(\F)=\{h\in \A(\F) : \submod_{h}=\submod\}$ for every polymatroid $\submod$ on $\gE$.

A less precise stratification of $\A(\F)$ can be obtained by defining, for each integer $k\in \{0,\ldots,N\}$, \[\A_{k}(\F)=\{h\in \A(\F) : \dim \ker d_{h}=N-k\}.\]
With our terminology, $(\F,h)$ is flat if and only if $h$ belongs to $\A_0(\F)$, and $(\F,h)$ is fully twisted if and only if $h$ belongs to $\A_N(\F)$. 

For each $k\in \{0,\ldots,N\}$, 
\[\A_{k}(\F)=\bigsqcup_{\submod : \submod(\gE)=N (|\rV|-1)+k} \A_{\submod}(\F).\]

\begin{proposition} We have the equality $\A_{0}(\F)=\A_{N\submod_{\o}}(\F)$ and the inclusion $\A_{N}(\F)\supset \A_{N\submod_{\oo}}(\F)$. Moreover, this inclusion is strict on any graph that has at least two independent cycles.  
\end{proposition}

\begin{proof} The first assertion comes from Proposition \ref{prop:polydirred}. 
For the second, consider two independent loops in the graph. Any connection that makes the local system twisted along the first loop and has holonomy equal to the identity belongs to $\A_{N}(\F)$ but not to $\A_{N\submod_{\oo}}(\F)$.
\end{proof}

Choosing an orthonormal basis of each fibre of $\F$ puts the set $\A(\F)$ of all connections on $\F$ in one-to-one correspondence with a finite Cartesian product of unitary groups. Changing the bases alters the correspondence, but preserves both the smooth structure and the Haar probability measure. This allows us to speak of negligible subsets of $\A(\F)$.

\begin{proposition}\label{prop:negligible}
The set $\A(\F)\setminus \A_{N\submod_{\oo}}(\F)$ is a closed negligible subset of $\A(\F)$.
\end{proposition}

\begin{proof} Let us consider $h\in \A_{N\submod_{\oo}}(\F)$. By definition, this means that for all subset $A$ of $\gE$, we have $\dim d_{h}^{*}(\F_{A})=N\submod_{\oo}(A)$. Therefore, by lower semi-continuity of the rank, any connection $h'$ close enough to~$h$ satisfies $\dim d_{h'}^{*}(\F_{A})\geq N\submod_{\oo}(A)$, and $\submod_{h'}\geq N\submod_{\oo}$. Since the reverse inequality holds for all connections, we must have equality. Finally, any connection close enough to $h$ belongs to $\A_{N\submod_{\oo}}(\F)$, which is therefore an open set.

When a connection is picked uniformly at random in $\A(\F)$, almost surely, along no loop in the graph does the holonomy have $1$ as an eigenvalue. In particular, on any connected subgraph of~$\rG$ that is not a tree, there is no non-zero horizontal section. Therefore, by Theorem \ref{theo:computerank}, almost surely, we have $\rho_{h}(A)=N\submod_{\oo}(A)$. This proves that $\A_{N\submod_{\oo}}(\F)$ has full measure in $\A(\F)$. 
\end{proof}

Let us conclude by giving a criterion for a connection to belong to $\A_{N\submod_{\oo}}(\F)$.

\begin{proposition} \label{prop:crit=} 
Let $(\F,h)$ be a fully twisted local system of rank $N$ over $\rG$. The following assertions are equivalent.
\begin{enumerate}[\indent 1.]
    \item $\pPbas(\submod_{h})= \pPbas(N\submod_{\oo})$.
    \item For all $B\in \B(\rG)$, one has $N\1_B \in \pPbas(\submod_h)$.
     \item For all $B\in \B(\rG)$, one has $\P(\Q=\F_B)>0$.
    \item $(\F,h)$ is fully twisted along every simple loop in $\rG$.
\end{enumerate}
\end{proposition}

\begin{proof} For every $B\in \B(\rG)$, the vector $N\1_B$ belongs to $\pPbas(N\submod_{\oo})$, hence $1$ implies $2$. 

By definition of the bicircular matroid,
$\nPbas(\submod_{\oo})$ is the set of indicator functions of CRSFs. We also know that $\pPbas(\submod_{\oo})$ is the convex hull of this set. Moreover, no strict convex combination of two indicator functions of distinct CRSFs can be integer-valued. Therefore, the indicator functions of CRSFs are exactly the extremal point of $\nPbas(\submod_{\oo})$. Therefore, the functions $N\1_B$ are the extremal points of $\nPbas(N\submod_{\oo})$. Any convex polytope which contains them contains $\pPbas(N\submod_{\oo})$. Together with the second assertion of Proposition \ref{prop:polydirred}, this proves that $2$ implies $1$.

Let $B$ be a CRSF of $\rG$. By Proposition \ref{prop:supporttrace}, the function $N\1_B$ belongs to $\pPbas(\submod_h)$ if and only if $\P(\dim_\gamma \Q=N\1_B)>0$ and since the stratum of the Grassmannian of $\Omega^1(\F)$ corresponding to the dimension $N\1_B$ consists of the single point $\F_B$, this is equivalent to $\P(\Q=\F_B)>0$. Therefore, $2$ and $3$ are equivalent.

Finally, the equivalence of $3$ and $4$ follows from Theorem \ref{thm:atoms} and the fact that every single loop of $\rG$ is included in a CRSF. 
\end{proof}

This proposition does not say in general that $\pPbas(\submod_h)$ is the convex hull of the vectors of the form $N\1_B, B \in \B(\rG)$ that it contains. For instance, consider the case where $\rG$ has only one vertex and two self-loops $e$ and $f$. Consider two local systems $(\F,h)$ and $(\F',h')$ of rank $1$ over~$\rG$. Assume that $h_e=h'_f=1$ and $h_f=h'_e=-1$. Then the dimension of $\DLSF(\F,h)$ is almost surely $(0,1)$ and the dimension of $\DLSF(\F',h')$ is almost surely $(1,0)$. Therefore, the dimension of $\DLSF(\F\oplus \F',h\oplus h')$ is almost surely $(1,1)$. Therefore, $\pPbas(\submod_{h\oplus h'})=\{(1,1)\}$ does not contain any vector of the form $2\1_B$ for any $B\in \B(\rG)=\{\{e\},\{f\}\}$. 

%%%%%%%%%%%%%%%%%%%%%%%%%%%%%%%%%%%%%%%%

\section{The characteristic polynomial of the Laplacian}\label{sec:characteristic-polynomial}

From \cite{KL3} it follows that $\det (\Delta_h)_{\im d_h^*}^{\im d_h^*}$ can be interpreted as a partition for the $\DLSF$.
Indeed, this follows from the fact that the density of the distribution of the $\DLSF$ with respect to the reference measure $\nu_n$ on the Grassmannian (see Section \ref{sec:dls}) is 
\[\frac{d\P}{d\nu_n}(Q)=\cos^2(Q, \im d_h)=\frac{\det( d_h^* \proj{Q} d_h)}{\det (\Delta_h)_{\im d_h^*}^{\im d_h^*}}\,,\]
with normalizing constant
\[\det (\Delta_h)_{\im d_h^*}^{\im d_h^*} = \int_{\Gr_n(\Omega^1(\F),\sigma_1)} \det( d_h^* \proj{Q} d_h) d\nu_n\,.\]

The goal of this section is to give a combinatorial expression for this partition function, that is, for the product of nonzero eigenvalues of the twisted Laplacian $\Delta_h$ (Theorem \ref{thm:mtt}). We will obtain it as a corollary of an expression for the characteristic polynomial of the Laplacian (Theorem \ref{thm:carlap}). 

The problem of finding ``any combinatorial meaning" to the determinant of the twisted Laplacian in rank greater or equal to $2$
was posed by Kenyon in 
\cite[Section 14, Open question~2]{Kenyon}. In \cite{KL2,KL2b,KL5}, we proposed possible answers to this question, which we did not find completely satisfactory. The answer proposed here has the advantage over those previous ones of expressing a positive spectral quantity (the product of nonzero eigenvalues) as a sum over a combinatorial object (integer points of the Minkowski sum of $(N-k)$ copies of the base polytope of the circular matroid, and $k$-copies of the base polytope of the bicircular matroid) of a manifestly gauge-invariant and positive quantity.

In order to prove our result, we first set up some notations and results from abstract theory of representable polymatroids on a finite set. The generalized matrix-tree theorem (Theorem~\ref{thm:mtt}) is then seen as a special case of this result.

\subsection{Measured polymatroids}

Let $(E_{1},(\cdot,\cdot))$ and $(E_{0},(\cdot,\cdot))$ be finite dimensional inner product spaces. Let $E_{1}=\bigoplus_{i\in S} E_{1,i}$ be an orthogonal splitting of $E_{1}$, that we denote by $\sigma$. Let $\partial \colon E_{1} \to E_{0}$ a linear map. According to Section \ref{sec:linpol}, the family of linear subspaces $(E_{0,i}=\partial E_{1,i} : i\in S)$ determines a polymatroid on~$S$, of which we denote by $\submod$ the rank function: for every subset $I$ of $S$, 
\[\submod(I)=\dim \sum_{i\in I} E_{0,i}.\]

The dual spaces $E^{1}$ and $E^{0}$ of $E_{1}$ and $E_{0}$ inherit inner products from $E_{1}$ and $E_{0}$. Moreover,~$E^{1}$ inherits the orthogonal splitting $E^{1}=\bigoplus_{i\in S} E^{1,i}$, where for each $i\in S$,
\[E^{1,i}=\{\varphi \in E^{1} : \forall j\in S \setminus  \{i\}, \, \varphi_{|E_{1,j}}=0\}.\]

Given a collection of positive reals $\ul{x}=(\x_i:i\in S)$ that we call {\em Euclidean weights}, we twist the inner product on $E^{1}$ into a new inner product $\lanx \cdot,\cdot \ranx$ by setting, for all $i,j\in S$, and all $\alpha_{i}\in E^{1,i}$ and $\alpha_{j}\in E^{1,j}$,
\begin{equation}
\lanx u_{i}, u_{j} \ranx = x_{i} (\alpha_{i},\alpha_{j}).
\end{equation}
Note that this scalar product vanishes if $i\neq j$.

For every subset $I$ of $S$, we will use the notation $\ul\x^{I}=\prod_{i\in I}x_{i}$ and $\ul\x^{-I}=(\ul\x^{I})^{-1}$.

Let us now consider the linear map $d:E^{0}\to E^{1}$, defined as the transposed map of $\partial$, and the map $d^{*}:E^{1}\to E^{0}$, defined as the adjoint of $d$.

For all split inner product space $W=\bigoplus_{i\in S} W_{i}$ and all family $\ul{t}=(t_{i},i\in S)$ of scalars, we define the endomorphism $m^{W}_{\ul{t}}=\sum_{i\in S} t_{i} {\rm id}_{W_{i}}$ of $W$. In the next results, we denote by $\ul{k}$ a vector of $\N^{S}$, and we use the notation $\ul{t}^{\ul{k}}=\prod_{i\in S}t_{i}^{k_{i}}$. 

Recall from Section \ref{sec:dls} the definitions of various subsets of the Grassmannian of $W$, and the measures on these sets, associated to an orthogonal splitting $\sigma$ of $W$. Let us also introduce the following notation: if $\sR$ is a linear subspace of $W$, we denote by ${\sf 1}_{\sR}\colon \sR \to W$ the inclusion map of $\sR$ into this space, and by ${\sf 1}^{\sR}\colon W \to \sR$ the adjoint of ${\sf 1}_{\sR}$, that is, the orthogonal projection onto $\sR$. It will always be clear from the context which space plays the role of $W$.

We will use the following lemma.

\begin{lemma} \label{lem:det1+ta} Let $W=\bigoplus_{i\in S} W_{i}$ be an orthogonal splitting of an inner product space $W$. 
For all $a\in \End(W)$ and all family $\ul{t}=(t_{i},i\in S)$ of scalars, 
\[\det(1+m^{W}_{\ul{t}}a)=\sum_{\ul{k}\geq 0} \ul{t}^{\ul{k}} \int_{\Gr_{\ul{k}}(W ,\sigma)}\det({\sf 1}^{R}a {\sf 1}_{R}) \d \nu_{\ul{k}}(R).\]
\end{lemma}

\begin{proof} This follows from \cite[Prop. 3.8]{KL3}.
\end{proof}

We denote by ${\sf J}^{0}:E_{0}\to E^{0}$ the antilinear isomorphism induced by the inner product on $E_{0}$. To be clear, for all $u,v\in E_{0}$, we have $\big({\sf J}^{0} u\big)(v)=(u,v)$. We define similarly ${\sf J}^{1}:E_{1}\to E^{1}$ induced by the inner product on $E_{1}$, of which we stress that it does not depend on $\ul{x}$. We set ${\sf J}_0=({\sf J}^0)^{-1}$ and ${\sf J}_1=({\sf J}^1)^{-1}$.

We will use the fact that the adjoint of $\partial$ is equal to 
\[\partial^{*}=\sJ_{1} d \sJ^{0}.\]

Let us define, for all $\ul{k} \in \N^{S}$
\[w(\ul{k})=\int_{\Gr_{\ul{k}}(E_{1},\sigma)} \det({\sf 1}^{U}\partial^{*} \partial {\sf 1}_{U}) \d \nu_{\ul{k}}(U).\]

\begin{proposition} \label{prop:det1+tdd*} For all scalars $\ul{t}=(t_{i},i\in S)$, the following equality holds:
\[\det(1+m^{E^{1}}_{\ul{t}}dd^{*})=\sum_{\ul{k}\in \pPind(\submod)}  \ul{t}^{\ul{k}} \ul{x}^{\ul{k}} w(\ul{k}).\]
\end{proposition} 

\begin{proof} By \cite[Lemma 13.4]{KL5}, we have the identity
\[d^{*}={\sf J}^{0} \partial \sJ_{1} m^{E^{1}}_{\ul{x}}.\]
Let us choose $\ul{k}\in \N^{S}$ and $R\in \Gr_{\ul{k}}(E^{1}\!,\sigma)$. Then $m_{\ul{x}}^{E^{1}}{\sf 1}_{R}={\sf 1}_{R}m_{\ul{x}}^{R}$, so that
\[\det({\sf 1}^{R}d d^{*} {\sf 1}_{R})=\det({\sf 1}^{R}d {\sf J}^{0} \partial \sJ_1 {\sf 1}_{R})\det(m^{R}_{\ul{x}})=\ul{x}^{\ul{k}} \det({\sf 1}^{R}d {\sf J}^{0} \partial \sJ_1 {\sf 1}_{R}).\]
Let us define $U=\sJ_1(R)$. Then $\sJ_1 {\sf 1}_{R}={\sf 1}_{U}\sJ_1$ and ${\sf 1}^{R}{\sf J}^{1}={\sf J}^{1}{\sf 1}^{U}$. Therefore, 
\begin{align*} 
\det({\sf 1}^{R}d {\sf J}^{0} \partial \sJ_1 {\sf 1}_{R})&=\det({\sf 1}^{\sR}{\sf J}^{1}\sJ_1d {\sf J}^{0} \partial \sJ_1 {\sf 1}_{\sR})\\
&=\det({\sf J}^{1}{\sf 1}^{U}\sJ_1d {\sf J}^{0} \partial  {\sf 1}_{U} \sJ_1)=\det({\sf 1}^{U}\sJ_1d {\sf J}^{0} \partial  {\sf 1}_{U})=\det({\sf 1}^{U}\partial^{*} \partial  {\sf 1}_{U}).
\end{align*}

Finally, we apply Lemma \ref{lem:det1+ta} to $W=E^{1}$ and $a=dd^{*}$ and we use the fact that the isomorphism~${\sf J}^{1}$ preserves the measures on the Grassmannians to change the variable from $R$ to $U$, and we arrive at
\[\det(1+m_{\ul{t}}dd^{*})=\sum_{\ul{k}\geq 0}  \ul{t}^{\ul{k}} \ul{x}^{\ul{k}} w(\ul{k}).\]
There remains to restrict the sum to the $\ul{k}\geq 0$ for which the integral is not zero. By \cite[Prop. 3.5]{KL3}, they are precisely the points of the independence polytope of the polymatroid $\submod$.
\end{proof}

\begin{corollary} With the notation above, the following equality holds:
\[\det\big((d^{*}d)_{(\ker d)^{\perp}}^{(\ker d)^{\perp}}\big)=\sum_{\ul{k}\in \pPbas(\submod)} \ul{x}^{\ul{k}}w(\ul{k}).\]
\end{corollary}

\begin{proof} The determinant of the compression of $d^{*}d$ on the orthogonal of its kernel is the leading coefficient of the polynomial $\det(1+ t d^{*}d)=\det(1+ t dd^{*})$. It is therefore the same sum as above, restricted to the~$\ul{k}$ of maximal weight, that is, the bases of the polytope $\submod$.
\end{proof}

\subsection{The Laplacian of a local system}\label{sec:Laplocsys}

The Laplacian associated to a local system $(\F,h)$ over the graph $\rG$ is the operator 
\[\Delta_{h}=d_{h}^{*}d_{h} \in \End\big( \Omega^{0}(\F) \big).\]
We are interested in the characteristic polynomial of this operator
\begin{align*}
\det (1+t\Delta_{h})=\det (1+td_{h}^{*}d_{h})=\det (1+td_{h}d_{h}^{*}).
\end{align*}

Let us introduce the dual spaces $\Omega_{0}(\F)$ and $\Omega_{1}(\F)$ of $\Omega^{0}(\F)$ and $\Omega^{1}(\F)$ respectively. As vector spaces, we identify these dual spaces with the original ones. The only difference is that the inner product on $\Omega_{1}(\F)$ is not twisted by the Euclidean weights:
\begin{equation*}\label{eq:SPOmega_{1}}
\forall c_{1},c_{2}\in \Omega_{1}(\F), \quad
(c_{1},c_{2})=\frac{1}{2}\sum_{e\in \rE}  \langle c_{1}(e),c_{2}(e)\rangle_{e}\,.
\end{equation*}

Let $\partial_{h}:\Omega_{1}(\F)\to \Omega_{0}(\F)$ be the transposed map of $d_{h}$ and $\partial_{h}^{*}:\Omega_{0}(\F)\to \Omega_{1}(\F)$ its adjoint. To be explicit, if $e$ is an edge and $\xi\in \F_{e}$, then
\begin{equation}\label{eq:defpartialh}
\partial_{h}(\xi)=h_{\ol{e},e}(\xi)-h_{\ul{e},e}(\xi) \in \F_{\ol{e}}\oplus \F_{\ul{e}} \subseteq \Omega_{0}(\F),
\end{equation}
and if $v$ is a vertex and $\zeta\in \F_{v}$,
\begin{equation} \label{eq:defpartialh*}
\partial^{*}_{h}(\zeta)=\sum_{e : \ol{e}=v} h_{e,v}(\zeta)\in \bigoplus_{e:\ol{e} = v}\F_{e}\subseteq \Omega_{1}(\F).
\end{equation}

For all $\ul{k}\in \N^{\gE}$, let us define
\begin{equation}\label{eq:wh}
w_{h}(\ul{k})=\int_{\Gr_{\ul{k}}(\Omega_{1}(\F),\gamma_{1})}\det({\sf 1}^{U}\partial_{h}^{*}\partial_{h} {\sf 1}_{U}) \d \nu_{\ul{k}}(U),
\end{equation}
where $\gamma_{1}$ is the graphical splitting defined by \eqref{eq:graphicalsplitting}. 

\begin{theorem}\label{thm:carlap}
Let $\Delta_{h}$ be the Laplacian associated to a local system $(F,h)$. For all $t$,
\[\det(1+t\Delta_{h})=\sum_{\ul{k}\in \pPind(N \submod_{\oo})} t^{|\ul{k}|}\ul{x}^{\ul{k}} w_{h}(\ul{k}),\]
with $|\ul{k}|=\sum_{e\in \gE} k_{e}$.
\end{theorem}

\begin{proof} We specialize Proposition \ref{prop:det1+tdd*} and use Theorem \ref{thm:maininclusion}. 
\end{proof}

Let $\T(\ul{x})$ denote the generating function of spanning trees of $\rG$, that is, the polynomial
\[\T(\ul{x})=\sum_{T\in \cT(\rG)} \ul{x}^{T},\]
where $\cT(\rG)$ is the set of spanning trees of $\rG$. 

\begin{theorem}\label{thm:mtt}
Set $k=\dim \ker d_{h}$. Then the product of the non-zero eigenvalues of $\Delta_{h}$ is
\[{\det}^{*}(\Delta_{h})=|\rV|^{k}\T(\ul{x})^{k}\sum_{\ul{k}\in \pPbas((N-k) \submod_{\oo})} \ul{x}^{\ul{k}} w_{h}(\ul{k}).
\]
\end{theorem}

\begin{proof} To start with, we split $(\F,h)$ as the sum of its trivial isotypical part (see end of Section~\ref{sec:horsec}) and a local subsystem on which $d_{h}$ is injective. 

The trivial isotypical part splits into a sum of $k$ trivial rank $1$ subsystems, each with one zero eigenvalue of the Laplacian.

The product of the non-zero eigenvalues factorizes along the splitting. Each trivial rank 1 system yields $|\rV| \T(\ul{x})$.

For the other subsystem, the result is a consequence of Theorem \ref{thm:carlap} and the fact that $\submod_{h}\leq (N-k)\submod_{\oo}$ and the fact that the rank of $\submod_{h}$ is equal to the rank of $(N-k)\submod_{\oo}$.
\end{proof}

%%%%%%%%%%%%%%%%%%%%%%%%%%%%%%%%%%%%%%%%

\section{Infinitesimal connections}\label{sec:infinitesimal-connections}

In this section, we assume that $\rG$ is connected. Recall that, given a vector space $W$, we denote by $\rG\times W$ the vector bundle over $\rG$ of which all fibres are equal to $W$, and that we denote by~$\Omega^{\sbullet}(W)$ the corresponding spaces of $0$- and $1$-forms.

Let $\u(N)$ be the space of skew-Hermitian $N\times N$ matrices. Let $A$ be an element of $\Omega^1(\u(N))$, that is, a $\u(N)$-valued $1$-form on $\rG$. We think of~$A$ as an \emph{infinitesimal} connection on~$\rG$ which induces a genuine connection in the following way: for each real $t$, we define a connection $h_t$ on the bundle $\rG\times \C^{N}$ by setting, for each edge $e$ joining $u=\ul{e}$ to $v=\ol{e}$,
\begin{equation}
(h_{t})_{e,u}=(h_{t})_{v,e}=e^{\frac{t}{2} A_e}.
\end{equation}
We are going to study the local system $(\F=\rG\times \C^{N}, h_{t})$ as $t$ tends to $0$.

For every $1$-chain $c=\sum_{e\in \gE} c_{e} e$ in $\rG$, we define $A(c)=\sum_{e\in \gE} c_{e} A_e\in \u(N)$. 

In Theorem~\ref{thm:partition-limit} below, we compute the leading term of $\det \Delta_{h_t}$ as $t$ tends to $0$. It involves a family of subgraphs of $\rG$ which we introduce now.

Let $\cC(\rG,N)$ denote the set of $N$-tuples of connected spanning subgraphs of $\rG$ such that the sum of their first Betti numbers is~$N$, i.e. the total amount of edges is $N\vert\rV\vert$. Using the notation~$\cC_{r}(\rG)$ for the set of connected spanning subgraphs~$K$ of $\rG$ with $r$ independent cycles, that is, such that $b_1(K)={\rm rk} \,H_{1}(K,\Z)=r$, we have
\[\cC(\rG,N)=\bigsqcup_{\ul{r}\geq 0, |\ul{r}|=N} \cC_{\ul{r}}(\rG,N)\,.\]

To each $\underline{K}=(K_{1},\ldots,K_{N})\in\cC(\rG,N)$, and $A\in\Omega^1(\u(N))$ we will associate an $N\times N$ matrix. To this end, for each $j\in \{1,\ldots,N\}$, we set $r_{j}=b_{1}(K_{j})$ and we choose an integral basis $(c_{1,j},\ldots,c_{r_{j},j})$ of the free abelian group~$Z_{1}(K_{j},\Z)$. This gives us $N$ cycles along each of which we integrate~$A$, to produce $N$ skew-Hermitian matrices. From each of these matrices $A(c_{i,j})$, with $1\leq i \leq r_{j}$, we keep just one column, that with the label~$j$. With these $N$ columns, we form the matrix
\begin{equation}\label{eq:matriceA}
M_{A}(\ul{K})=(A(c_{1,1})_{1}\ldots A(c_{r_{1},1})_{1}\ldots\ldots A(c_{1,N})_{N}\ldots A(c_{r_{N},N})_{N}).
\end{equation}
This matrix depends on the choices of integral bases that we made, but the modulus of its determinant does not. Indeed, two different choices lead to two matrices that differ by multiplication on the right by a block-diagonal matrix with integer entries, the $j$-th block being the change of basis in $Z_{1}(K_{j},\Z)$ whose determinant is thus~$\pm 1$. Therefore, the square of the modulus of the determinant of $M_{A}(\ul{K})$ is independent of the choice of integral bases. 

\begin{theorem}\label{thm:partition-limit}Let $A\in \Omega^1(\u(N))$ and define $h_t$ as above.
As $t$ tends to $0$, we have
\begin{equation}\label{eq:partition-limit}
\det \Delta_{h_t}=t^{2N} \; \sum_{\underline{K}\in \cC(\rG,N)} |\det M_{A}(\ul{K})|^{2} \; \ul{x}^{\underline K}+O(t^{2N+1})
\,.
\end{equation}
\end{theorem}

Let us break down the proof of Theorem \ref{thm:partition-limit} in two steps, namely Propositions \ref{prop:Schur-expansion} and \ref{prop:KL5-13-14}, from which the result follows readily. 

In the following, given a connection $h$, we will use the short-hand notation $\boxminus_{h}$ for $\ker d_{h}^*$.

Recall that $\F$ is the vector bundle $\rG\times \C^{N}$. 

\begin{proposition}\label{prop:Schur-expansion}
Let $A\in \Omega^1(\u(N))$ and define $h_t$ as above. For each $j\in \set{N}$, let $A_j\in \Omega^1(\F)$ be the $j$-th column of $A$, which is an element of $\Omega^{1}(\rG\times \C^{N})=\Omega^{1}(\F)$. Define $\alpha=A_1\wedge \ldots \wedge A_N\in \ext^N \Omega^1(\F)$. 
As $t$ tends to~$0$, we have
\begin{equation}\label{eq:prop-schur}
\det \Delta_{h_t} = t^{2N} \vert \rV\vert^N \T(\ul{x})^N \lVert (\proj{\ker d^*_{h_0}})^{\wedge N}(\alpha) \rVert^2+O(t^{2N+1})\,.
\end{equation}
\end{proposition}

\begin{proof}
Let us apply the Schur complement formula, in the guise of \cite[Lemma 5.7]{KL5}, to the operator $d_{h_{t}}:\Omega^{0}(\F)\to \Omega^{1}(\F)$ and the subspace~$H_{0}=\ker d_{h_{0}}$ of $\Omega^0(\F)$ consisting in constant $0$-forms. We find
\begin{equation}\label{eq:schur-deltat}
\det \Delta_{h_t} = \det\left(d_{h_t}^* d_{h_t}\right)=\det\left(1^{H_0^\perp}d_{h_t}^* d_{h_t} 1_{H_0^\perp}\right) 
\det \left( 1^{H_0} d_{h_t}^* \proj{\ker(1^{H_{0}^{\perp}}d_{h_{t}}^{*})} d_{h_t} 1_{H_0} \right)\,.
\end{equation}

Let us write a Taylor expansion of $d_{h_{t}}$ in a neighborhood of $t=0$. We split $\Omega^0(\F)$ as $H_0^{\perp} \oplus H_0$ and analyse the map $d_{h_t}$ on each summand. We find
\begin{equation}\label{eq:dht-expansion}
d_{h_t}=(d_{h_{0}} + O(t)) 1_{{H_0}^{\perp}} + (t {\omega}_A +O(t^2)) 1_{H_0},
\end{equation}
where $ \omega_{A}:H_{0} \subset \Omega^0(\F) \to \Omega^{1}(\F)$ is the linear map defined as follows: for every horizontal section $f\in H_{0}$, of which we denote by $f\in \C^{N}$ the unique value, and every edge~$e$, define
\begin{equation}\label{eq:omegaA}
(\omega_{A}(f))_{e}=A_{e}(f).
\end{equation}
Taking the adjoint of \eqref{eq:dht-expansion} yields
\begin{equation}\label{eq:dht*}
d^*_{h_t}=1^{H_0^{\perp}} (d_{h_{0}}^*+O(t)) + 1^{H_0} (t {\omega}_A^*  +O(t^2)).
\end{equation}
 
We now claim that
\begin{equation}\label{eq:proj-expansion}
\proj{\ker(1^{H_{0}^{\perp}}d_{h_{t}}^{*})} =\proj{\ker d^*_{h_0}}+O(t)\,.
\end{equation}
Let us show the analogous result for the complementary projections for which we have an explicit form. First, note that \begin{equation}\label{eq:orthogonals}
\left(\ker(1^{H_{0}^{\perp}}d_{h_{t}}^{*})\right)^{\perp}=\im(d_{h_{t}} 1_{H_{0}^{\perp}}) \; \text{and} \;  (\ker d_{h_{0}}^{*})^{\perp}=\im d_{h_0}=\im (d_{h_0} 1_{H_0^\perp})\,.
\end{equation} 
Secondly, for any injective linear map $u$ between inner product spaces $E^{0}$ and $E^{1}$, we may write 
\[\proj{\im u}=u (u^* u)^{-1} u^*\,.\] 
This formula shows that the map $u \mapsto \proj{\im u}$ from $\mathcal L(E^{0}, E^{1})$ to $\End(E^{1})$ is differentiable on the set of injective linear maps. 

The map $u_{0}=d_{h_0} 1_{H_0^\perp}$ is injective by construction, and by continuity, the map $u_{t}=d_{h_t} 1_{H_0^{\perp}}$ is injective for $t$ small enough. Therefore, 
\begin{equation}\label{eq:proj-Ot}
\proj{\im\big(d_{h_{t}} 1_{H_{0}^{\perp}}\big)}=\proj{\im d_{h_0}}+O(t)\,.
\end{equation}
From \eqref{eq:orthogonals} and \eqref{eq:proj-Ot} we thus deduce
\[\proj{\ker(1^{H_{0}^{\perp}}d_{h_{t}}^{*})}= \id_{\Omega^1(\F)}- \proj{\im\big(d_{h_{t}} 1_{H_{0}^{\perp}}\big)}= \id_{\Omega^1(\F)} - \proj{\im d_{h_0}}+O(t)=\proj{\ker d_{h_{0}}^{*}}+O(t)\,. \]
and the claim \eqref{eq:proj-expansion} is proved. 

From \eqref{eq:schur-deltat}, \eqref{eq:dht-expansion}, \eqref{eq:dht*}, and \eqref{eq:proj-expansion} we thus find
\begin{equation}
\det \Delta_{h_t}= \det\big( ({\Delta_{h_{0}}+O(t)})_{H_0^{\perp}}^{H_0^{\perp}}\big) \, \det( t  {\omega}_A^* \proj{\ker d_{h_{0}}^{*}} t {\omega}_A +O(t^3)).
\end{equation}
By factoring out the leading power of $t$, we obtain
\begin{equation}\label{eq:step2}
\det \Delta_{h_t}=t^{2N}\, \det\big({\Delta_{h_{0}}}_{H_0^\perp}^{H_0^\perp}\big) \, 
\det(\omega_A^* \proj{\ker d_{h_{0}}^{*}} \omega_A)+O(t^{2N+1})\,.
\end{equation}
We recognize the first term: letting $\Delta$ be the Laplacian on the trivial bundle of rank $1$, we have
\[\det\Big((\Delta_{h_{0}})_{H_0^\perp}^{H_0^\perp}\Big)=\left(\det \big(\Delta_{(\ker \Delta)^{\perp}}^{(\ker \Delta)^{\perp}}\big)\right)^{N}=(\vert \rV\vert \T(\ul{x}))^N\,,\]
where we used the classical matrix-tree theorem in the last equality (see for instance \cite[Section 6]{KL5}).
In order to compute $\det(\omega_A^* \proj{\ker d_{h_{0}}^{*}} \omega_A)$, we pick the basis of $H_0$ given by the canonical basis of~$\C^N$. Set $\alpha=A_1\wedge \ldots\wedge A_N$. Then, a computation similar to that done in the proof of \cite[Prop. 13.14]{KL5} yields
\[\det(\omega_A^* \proj{\ker d_{h_{0}}^{*}} \omega_A)=\lVert (\proj{\ker d_{h_{0}}^{*}})^{\wedge N}(\alpha) \rVert^2\]
and this finishes the proof.
\end{proof}

In order to compute $\lVert (\proj{\ker d^*_{h_0}})^{\wedge N}(\alpha) \rVert^2$, we rely on 
\cite[Prop. 13.14]{KL5}.

\begin{proposition}\label{prop:KL5-13-14}
We have
\begin{equation}\label{eq:Lx}
\sum_{\ul{K}\in \cC(\rG,N)} \ul{x}^{\ul{K}} |\det M_{A}(\ul{K})|^{2} = \T(\ul{x})^N \lVert (\proj{\ker d^*_{h_0}})^{\wedge N}(\alpha) \rVert^2\,.
\end{equation}
\end{proposition}
\begin{proof}
We consider the matroid on $\gE\times \set{N}$ defined by the map $d^*_{h_{0}}$ in the canonical basis. We think of it as $N$ copies of the graph $\rG$ on which we consider the circular matroid. More precisely, let $\rG_1,\ldots,\rG_N$ be these $N$ ordered copies, and let $\rG_{\set{N}}$ be their union.

The circular matroid on $\rG_{\set{N}}$ is the union of circular matroids on each of the $\rG_i$. In particular, its bases are the union of spanning trees of each $\rG_i$ (recall that we assume $\rG$ to be connected), which we denote $\ul{T}=(T_1,\ldots, T_N)$.

There is an isomorphism 
\[\Omega^1(\rG_{\set{N}})=\Omega^1(\rG_1)\oplus \ldots \oplus \Omega^1(\rG_N)\,.\]
We shall denote $\theta=\theta_{1}\oplus \ldots \oplus \theta_{N}$ the associated decomposition of a $1$-form $\theta$ on $\rG_{\set{N}}$.

Let us now define $N$ $1$-forms $\theta^{1},\ldots, \theta^{N}$ by taking the rows of $A$, that is, $\theta^{i}=A^i$.
Note that with this correspondence, the matrix $M_A(\ul{K})$ is $\theta^{i}(c_j)$ for $c_j$ a cycle basis of $\ul{K}$.

For this matroid, we have the following polynomials: 
\[\rB(\ul{x})=\det (d_{h_{0}}^* d_{h_{0}})_{\im d_{h_{0}}^*}^{\im d_{h_{0}}^*}=\left(\vert \rV\vert \T(\ul{x})\right)^N\]
and, with $\vartheta=\theta^{1}\wedge \ldots \theta^{N}$,
\[\rL_{N}(\vartheta,\ul{x})=\sum_{\ul{K}\in \cC(\rG,N)} w(\ul{K}) \ul{x}^{\ul{K}} \,.\]
where 
\[w(\ul{K})=w(\ul{T}) \det(\vartheta, \mathscr Z_{\ul{T}}^{\ul{K}})^2 \,.\]
Here $\ul{T}$ is a union of a tree on each copy of $\rG$, contained in $\ul{K}$, and $w(\ul{T})=\vert \rV \vert^N$.
Moreover, it follows from the above observation that $\det(\vartheta, \mathscr Z_{\ul{T}}^{\ul{K}})^2=\vert \det M_A(\ul{K})\vert^2$. In addition, we have
\[\lVert (\proj{\ker d^*_{h_0}})^{\wedge N}(\vartheta) \rVert^2=\lVert (\proj{\ker d^*_{h_0}})^{\wedge N}(\alpha) \rVert^2\]
since $\theta^{i}=A^i$.
The result thus follows from \cite[Prop. 13.14]{KL5} which states that $\rL_N(\vartheta,\ul{x})=\rB(\ul{x})\lVert (\proj{\ker d^*_{h_0}})^{\wedge N}(\vartheta) \rVert^2$, upon diving both sides by $\vert\rV\vert^N$.
\end{proof}

Plugging \eqref{eq:Lx} into \eqref{eq:prop-schur} yields Theorem \ref{thm:partition-limit}. 
Let us now turn to the specialization of the other equality given in \cite[Proposition 13.14]{KL5}.
Recall \eqref{eq:omegaA} and define
\[d_{h_0,A}=(d_{h_0})_{|(\ker d_{h_0})^\perp}\oplus \omega_A \colon \Omega^0(\F)= (\ker d_{h_0})^\perp \oplus \ker d_{h_0} \to \Omega^1(\F).\] 

\begin{proposition}\label{prop:KL5-13-14-2}
We have
\begin{equation}
\det\left(d_{h_{0},A}^* \circ d_{h_{0},A}\right) = \vert \rV\vert^N \sum_{\ul{K}\in \cC(\rG,N)} \ul{x}^{\ul{K}} |\det M_{A}(\ul{K})|^{2}\,.
\end{equation}
\end{proposition}
\begin{proof}
We simply need to check that with the above consideration of the circular matroid on~$N$ copies of~$\rG$, we have
$d_{h_0,A}=d\oplus \omega_\vartheta$ and then apply the second equality of \cite[Proposition~13.14]{KL3} which states that $\rL_N(\vartheta,\ul{x})=\det \Delta_\vartheta$ where $\Delta_\vartheta=(d\oplus \omega_\vartheta)^*\circ (d\oplus \omega_\vartheta)$.
\end{proof}

The next result follows directly from the previous discussion and, provided $d_{h_0,A}$ is injective, expresses the limit as $t$ tends to $0$ of the distribution of the trace of the $\DLSF(\F,h_t)$.

\begin{theorem} Assume that $A$ is such that $d_{h_0,A}$ is injective. Then, for all $\ul{n}\in \N^{\gE}$ such that $\sum n_e=N|\rV|$, we have
\[\lim_{t\to 0} \P(\dim_\gamma \Q=\ul{n})= \frac{|\rV|^N}{\det(d_{h_{0},A}^* d_{h_{0},A})} \sum_{\ul{K}\in \cC_{\ul{n}}(\rG,N)} \ul{x}^{\ul{K}} |\det M_{A}(\ul{K})|^{2}.\]
\end{theorem}

A case where the assumption that $d_{h_0,A}$ is injective is not satisfied is when the $1$-form $A$ is exact, in which case the curve $t\mapsto h_t$ is tangent at $t=0$ to the gauge orbit of the trivial connection $h_0$. 

%%%%%%%%%%%%%%%%%%%%%%%%%%%%%%%%%%%%%%%%

\section{Electrical networks}\label{sec:electrical-networks}

\subsection{The setting}\label{sec:nomenclature} 
In this section, we explain how the classical interpretation of graphs with Euclidean weights as electrical networks can be adapted to the framework of this paper, namely to local systems. For this, we start by making little adjustments in our formalism, introducing some notation, and explaining our convention. 

In contrast to what we did in Section \ref{sec:Laplocsys}, we find it more convenient here to define, for each edge $e$, the Euclidean weight $y_e=x_e^{-1}$, and to endow the space $\Omega_1(\F)$ of $1$-chains with the inner product
\begin{equation*}
\forall c_{1},c_{2}\in \Omega_{1}(\F), \quad
(c_{1},c_{2})=\frac{1}{2}\sum_{e\in \rE} y_e \langle c_{1}(e),c_{2}(e)\rangle_{e}\,.
\end{equation*}
Just as the weights $x_e$ could be interpreted as conductances (or inverse lengths), the weights $y_e$ represent resistances (or lengths).

The spaces $\Omega^0(\F)$ and $\Omega_0(\F)$ are the same vector space, but we keep the distinction between them for the sake of consistency. They are endowed with the untwisted inner product \eqref{eq:SPOmega0}. The inverse isometries~${\sf J}^0$ and ${\sf J}_0$ are, therefore, the identity map (see the diagram below).

We keep the definition \eqref{eq:defpartialh} for the twisted boundary operator $\partial_h:\Omega_1(\F) \to \Omega_0(\F)$, the definition of which does not involve any Euclidean weight.

Finally, let us define the inverse isometries $\sR:\Omega_1(\F)\to \Omega^1(\F)$ and $\sC:\Omega^1(\F)\to \Omega_1(\F)$ by setting
\[\sR=\bigoplus_{e\in \gE} y_e {\id}_{\F_e} \ \text{ and } \ \sC=\bigoplus_{e\in \gE} x_e {\id}_{\F_e}.\]

These definitions are summarized in the following commutative diagram:
\[\xymatrix{ \Omega^{0}(\F) \ar[r]^-{d_{h}} \ar@<0.5ex>[d]^{\sJ_{0}} & (\Omega^{1}(\F),\langle \cdot,\cdot\rangle_{\ul{x}}) \ar@<0.5ex>[d]^{\sC} \\
\Omega_{0}(\F) \ar@<0.5ex>[u]^{\sJ^{0}} & (\Omega_{1}(\F),\langle \cdot,\cdot\rangle_{\ul{y}}) \ar[l]_-{\partial_{h}} \ar@<0.5ex>[u]^{\sR}.
}\]

As far as an electric interpretation is involved, we think of the elements of $\Omega_1(\F)$ as currents, the elements of $\Omega^1(\F)$ as differences of potential, the elements of $\Omega^0(\F)$ as potentials, and the elements of~$\Omega_0(\F)$ as incoming currents. The electrical energy (or more properly power) of a current $i\in \Omega_1(\F)$ is the number $\|i\|^2_{\ul{y}}=\|\sR i \|^2_{\ul x}$.

\subsection{The classical problem}

One of the original motivations for the study of the Laplacian on networks with weighted edges is the study of electrical networks \cite{Kirchhoff} (see \cite{Weyl} for one of the first historical reformulations of this problem in the language of Analysis Situs, which became algebraic graph theory). The problem addressed by Kirchhoff is the following. Consider a network of electric generators, each consisting of a perfect voltage generator and a resistor, and find the current in each generator. The network is represented by a finite connected graph $\rG=(\rV,\rE)$, with one edge for each generator. On each oriented edge $e$ of the graph, the tension of the perfect generator is denoted by $\omega_{e}$, and the resistance by~$y_{e}$. Then $\omega$ is a $1$-form on the graph and $\ul{y}=(y_{e})_{e\in \gE}$ is a collection of Euclidean weights. The current~$i$ to be determined is a $1$-chain on $\rG$, determined by the two Kirchhoff laws, namely
\begin{equation*}
\left\{\begin{array}{l}
     \omega-\sR i \text{ sums to } 0 \text{ along every cycle, that is, belongs to } \im d,\\
     \partial_{h}i=j.
\end{array} \right. 
\end{equation*}

The unique solution $i$ to this problem is given by the decomposition of $\omega$ along $\Omega^{1}=\im d\oplus \ker d^{*}$:
\[i=\sC \sigma, \ \text{ where } \ \omega= d \eta + \sigma, \ \text{ with } \ d^{*}\sigma=0.\]
In other words,
\begin{equation}\label{eq:solkirch}
  i=\sC \proj{\ker d^{*}}\omega.  
\end{equation}

It is also possible to further impose at each node of the network an incoming current, represented by a $0$-chain $j$. This $0$-chain must satisfy the condition $j\in \im \partial$, which is equivalent to the fact that the total incoming current in the network is zero. Now, the equations become

\begin{equation*}
\left\{\begin{array}{l}
     \omega-\sR i\in \im d,\\
     \partial i=j.
\end{array} \right.
\end{equation*}

To solve these inhomogeneous equations, we introduce an arbitrary current $c\in \Omega_{1}(\F)$ such that $\partial c=j$. Using $c$, the equations can be rewritten as 
\[\omega-\sR (i-c) \in \sR c+\im d \ \text{ and } \ \partial (i-c)=0.\]
To solve these, we write the decomposition of $\omega-\sR c$ along $\Omega^{1}=\im d\oplus \ker d^{*}$ and find
\[\omega-\sR c= d \eta + \sigma, \ \text{ with } \ d^{*}\sigma=0.\]
Since $\partial =d^* \sR$, the second Kirchhoff equation yields $d^{*}(\sR(c-i))=0$, so that $\sR(i-c)=\sigma$. Therefore,
\begin{equation}
i=c+\sC \proj{\ker d^{*}}(\omega-\sR c)=\sC \proj{\ker d^{*}}\omega + ({\rm id}-\proj{\ker \partial})c.
\end{equation}

From a probabilistic perspective, the operator ${\rm id}-\proj{\ker \partial}$ can be nicely expressed in terms of the uniform spanning tree of the graph. Indeed, if $\T$ is a random spanning tree of $\rG$ distributed according to \eqref{eq:law-st}, we have
\[\proj{\ker \partial}=\E[\sP^{\ker \partial}_{\| \sT}],\]
by the mean projection theorem \cite[Thm. 5.5]{KL5} conjugated by the inverse isometries $\sR$ and~$\sC$, and \cite[Section 6]{KL5}.

This can be used to give a probabilistic solution to Kirchhoff's equations in certain cases. 
For example, let $a$ and $b$ be two nodes of the network, that is, two vertices of $\rG$. Let us consider the case where $\omega=0$ and $j=a-b$, representing one unit of current entering the network at $a$ and exiting at $b$. Let us choose a path $c$  from $a$ to $b$. Then for every spanning tree~$T$, the chain~$\proj{\ker \partial} c$ is the loop $c [b,a]_{T}$. Therefore, $c-\proj{\ker \partial} c=[a,b]_{T}$ and it is well known (since~\cite{Pemantle}) that when~$T$ follows the above law, this random path is the loop erased random walk from $a$ to~$b$. Therefore, the current $i$ is the expectation of the $1$-chain given by the oriented loop-erased random walk from $a$ to $b$.

\subsection{The case of local systems} Let us proceed by analogy in the case where the graph is endowed with a local system. Let us consider a local system $(\F,h)$ over a weighted graph $(\rG,\ul{x})$.

Let us choose $\omega\in \Omega^{1}(\F)$. Let us also take $j\in \im \partial_{h}\subseteq \Omega_{0}(\F)$. We are looking for $i\in \Omega_{1}(\F)$ such that
\begin{equation}\label{eq:kir-problem}
\left\{\begin{array}{l}
     \omega-\sR i\in \im d_{h},\\
     \partial_{h}i=j.
\end{array} \right.
\end{equation}

Let us choose $c\in \Omega_{1}(\F)$ such that $\partial_{h}c=j$. Following step by step the reasoning of the previous section, we find that the unique solution is $i=\sC \proj{\ker d_{h}^{*}}\omega + ({\rm id}-\proj{\ker \partial_{h}})c$,
which we may also rewrite as
\begin{equation}\label{eq:kir-solution}
i=\proj{\ker \partial_{h}}\sC \omega + ({\rm id}-\proj{\ker \partial_{h}})c.
\end{equation}

This solution can also be characterized by a minimization problem.

\begin{proposition}[Thomson's minimization principle] Given $j$ and $\omega$, the solution $i$ to \eqref{eq:kir-problem} given by \eqref{eq:kir-solution} satisfies
\[i=\arg\min\big\{ \|\iota-\sC\omega\|_{\ul{y}}^{2} : \iota \in \Omega^{1}(\F), \partial_{h}\iota=j\big\}.\]
\end{proposition}

\begin{proof} We have $i-c=\proj{\ker \partial_{h}}(\sC \omega-c)$. Therefore,
\[i-c=\arg\min \{ \| \sC \omega-(c+\iota) \|_{\ul{y}}^{2} : \partial_{h}(c+\iota)=j\}=\arg\min \{ \| \sC \omega-\iota \|_{\ul{y}}^{2} : \partial_{h}\iota=j\}-c,\]
the second equality being an instance of a fully general identity about $\arg\min$ and translations. 
\end{proof}

The following result expresses two monotonicity properties of the solution of \eqref{eq:kir-problem} with respect to the electric resistances in the network.

\begin{proposition}[Rayleigh's monotonicity law] Consider $\omega\in \Omega^1(\F)$ and $j\in \im \partial_h$.
\begin{enumerate}[\indent 1.]
    \item If $\omega=0$, then the map $\ul{x}\mapsto \|\sR i\|_{\ul x}^{2}$ is non-increasing. 
    \item If $j=0$, then the map $\ul{x}\mapsto \|\sR i\|_{\ul x}^{2}$ is non-decreasing. 
\end{enumerate}
\end{proposition}

\begin{proof} Consider two conductances $\ul x$ and $\ul x\leq \ul x'$. We have, in general,
\[\|i-\sC \omega\|^{2}_{\ul y}\geq \|i-\sC \omega\|^{2}_{\ul y'}\geq \|i'-\sC \omega\|^{2}_{\ul y'}.\]
If $\omega=0$, then we find
\[\|\sR i \|^{2}_{\ul x}=\|i\|^{2}_{\ul y}\geq \|i'\|^{2}_{\ul y'}=\|\sR i' \|^{2}_{\ul x'},\]
which proves the first assertion.

If $j=0$, the decomposition $\sC \omega=(\sC \omega -i )+i$ is $\ul y$-orthogonal, so that 
\[\|i-\sC \omega\|^{2}_{\ul y}=\|\sC \omega\|^{2}_{\ul y}-\|i\|^{2}_{\ul y}.\]
The same holds for $i'$, and we find $\|\sR i \|^{2}_{\ul x}=\|i\|_{\ul y}^{2}\leq \|i'\|^{2}_{\ul y'}=\|\sR i' \|^{2}_{\ul x'}$.
\end{proof}

\begin{remark} \label{rem:inject}
For a generic connection $h$, the operator $\partial_{h}$ is onto, so that the condition $j\in \im \partial_h$ is satisfied by every $0$-chain. Therefore, in electrical terms, it is possible to inject current at a single vertex of the graph, and to let it find its way inside the network.
\end{remark}

A probabilistic interpretation of the projection is given by the equality
\[\id-\proj{\ker \partial_{h}}=\id -\E[\sP^{\ker \partial_{h}}_{\| \S}]=\E[\sP^{\S}_{\| \ker \partial_{h}}],\]
where $\S$ is the $\DLP(\Omega_{1}(\F),\im \partial_{h}^{*},\gamma_{1})$, and we emphasize that only the distribution of $\S$ depends on the Euclidean weights in the last expression. The projection is purely topological. Moreover, note that $\Q=\sR(\S)$ is the $\DLSF(\F,h)$.

Then the solution \eqref{eq:kir-solution} to \eqref{eq:kir-problem} is 
\begin{equation}\label{eq:kir-solution-proba}
i=\E\big[\sP^{\S}_{\| \ker \partial_{h}}(c)\big]
\end{equation}
for any $c\in \Omega_1(\F)$ such that $\partial_h c=j$.

A different choice $c'$ of $1$-chain such that $\partial_h c'=j$ differs from $c$ by an element of $\ker \partial_h$, and therefore satisfies $\sP^{\S}_{\| \ker \partial_{h}}(c')=\sP^{\S}_{\| \ker \partial_{h}}(c)$ almost surely. 

Let us explain how to find a $1$-chain $c$ such that $\partial_{h}c=j$ in the case where $(\F,h)$ is fully twisted along a loop. By linearity, we lose no generality by assuming that $j$ is supported by a unique vertex $v$, where it is equal to $j_v$. Let us choose a loop $\gamma=(e_1,\ldots,e_k)$ based at $v$ such that $\hol_h(\gamma)$ does not have $1$ as an eigenvalue. Then we define $c$ as the parallel transport along~$\gamma$ of $({\rm id}-\hol_h(\gamma))^{-1}(j_{v})$, that is,
\[c=\sum_{i=1}^k \underbrace{\big(\hol_h(\gamma_{|[v,e_i]})- \hol_h({\gamma_{|[e_i,v]}}^{-1})\big)(j_v)}_{\in\,  \F_{e_i}}.
\]

\subsection{Effective resistance} It seems natural to call the operator 
\[R_{h}=\proj{\im d_{h}}\sR:\Omega_{1}(\F)\to \Omega^{1}(\F)\]
the effective resistance of the network. When $\omega=0$ and $j=\partial_{h} c$, the solution $i$ to the Kirchhoff problem satisfies
\[\sR i=R_{h}c.\]

The following proposition is a generalization of the Foster--Tetali \cite{Tetali} theorem.

\begin{proposition}[Foster's network theorem]
Let $\Q$ be the $\DLSF(\F,h)$ and let $\ul n = \dim_{\gamma} \Q$ its trace. Then
\[\det\big({\rm id}_{\Omega^{1}(\F)} +m_{\ul z}^{\Omega^{1}(\F)} R_{h} \sC\big)=\E\Big[\prod_{e\in \gE} (1+z_{e})^{n_{e}}\Big].\]
In particular,
\[\det\big({\rm id}_{\F_{e}} +zx_{e} (R_{h})^{\F_{e}}_{\F_{e}} \big)=\E\big[(1+z)^{n_{e}}\big],\]
and 
\[\E[n_{e}]=x_{e}\Tr \big((R_{h})^{\F_{e}}_{\F_{e}}\big) \ \text{ and } \ \P(\Q_{e}=\F_{e})=x_{e}^{N} \det \big((R_{h})^{\F_{e}}_{\F_{e}}\big).\]
Moreover,
\[\sum_{e\in \gE}x_{e} \Tr \big((R_{h})^{\F_{e}}_{\F_{e}}\big) =\rk(d_{h})=N|\rV|-\dim (\ker d_{h}).\]
\end{proposition}

\begin{proof} The first relation is a direct application of one of the expression of the generating function of a DLP, namely, in this setting
\[\det\big({\rm id}_{\Omega^{1}(\F)} +m_{\ul z}^{\Omega^{1}(\F)} \proj{\im d_{h}}\big)=\E\Big[\prod_{e\in \gE} (1+z_{e})^{n_{e}}\Big].\]
The others follow by successive specialization. Finally, the last identity expresses the fact that $\dim \Q$ is almost surely equal to the rank of $d_h$.
\end{proof}

\subsection{Relation with the Gaussian free vector field} \label{sec:GFVF}

To the local system $(\F,h)$, there is an associated Gaussian random vector $\sPhi$ of $\Omega^{0}(\F)$ such that for all $\alpha, \beta\in \Omega^{0}(\F)$, 
\[\E[\langle \alpha, \sPhi\rangle \langle \sPhi,\beta \rangle]=\langle \alpha, \Delta_{h}^{-1} \beta\rangle.\]

Since the projection on $\im d_{h}$ can be written as $d_{h}\Delta_{h}^{-1}d_{h}^{*}$, a small computation gives the distribution of $d_{h}\sPhi$, which is a Gaussian random vector in $\Omega^{1}(\F)$, namely, for all $\eta,\xi\in \Omega^{1}(\F)$,
\[\E[\langle \eta, d_{h}\sPhi\rangle \langle d_{h}\sPhi,\xi \rangle]=\langle \eta, \proj{\im d_{h}}\xi\rangle.\]
From this relation, we deduce the following expression of the effective resistance.

\begin{proposition} Let $\eta$ and $\xi$ be elements of $\Omega^{1}(\F)$. Let $e$ and $f$ be two edges. Then
\[\E[\langle \eta_{e}, d_{h}\sPhi\rangle_{e} \langle d_{h}\sPhi,\xi_{f} \rangle_{f}]=\langle \eta_{e}, R_{h} \xi_{f}\rangle_{e}.\]
\end{proposition}

\begin{proof} From the equation above and the definition of the inner products, we find
\[x_{e}x_{f}\E[\langle \eta_{e}, d_{h}\sPhi\rangle_{e} \langle d_{h}\sPhi,\xi_{f} \rangle_{f}]=x_{e}\langle \eta_{e}, \proj{\im d_{h}}\xi_{f}\rangle_{e}=x_{e}x_{f}\langle \eta_{e}, R_{h} \xi_{f}\rangle_{e},\]
hence the result.
\end{proof}

%%%%%%%%%%%%%%%%%%%%%%%%%%%%%%%%%%%%%%%%

\section{Infinite graphs}\label{sec:infinite-graphs}

In addition to its mathematical interest, it is useful from the point of view of statistical physics to consider determinantal linear spanning forests on infinite graphs. For a treatment of this question in the classical case of spanning forests, see~\cite[Chapters 9 and 10]{Lyons-Peres}.

We assume that our graphs are countably infinite and locally finite, in the sense that $\rV$ is countable and each vertex has finite degree.

\subsection{Hilbert spaces and splitting}

Even when $\rG$ is infinite, the equations \eqref{eq:defdh} and \eqref{eq:d-star} define linear operators between the spaces 
\[\Omega^{0}(\F)=\prod_{v\in \rV} \F_{v} \ \text{ and } \ \Omega^{1}(\F)=\prod_{e\in \gE} \F_{e}.\]
We keep the notation $d_{h}$ and $d_{h}^{*}$ for these operators, but emphasize that $d_{h}^{*}$ is not the adjoint of~$d_{h}$ anymore, simply because there is no inner product structure on $\Omega^{0}(\F)$ and $\Omega^{1}(\F)$. We keep the notation $\Delta_{h}=d_{h}^{*}d_{h}$ for the linear endomorphism of $\Omega^{0}(\F)$.

Let us now define Hermitian structures on subspaces of $\Omega^0(\F)$ and $\Omega^1(\F)$. For this purpose, we use the inner product \eqref{eq:SPOmega1} on $\Omega^{1}(\F)$, and a twisted version of  \eqref{eq:SPOmega0} on $\Omega^{0}(\F)$. For every vertex~$v$, we define $\lambda_{v}=\sum_{\ol{e}=v}x_{e}$ and we set 
\begin{equation}\label{eq:SPOmega0lambda}
\forall f_{1},f_{2}\in \Omega^{0}(\F), \quad
(f_{1},f_{2})=\sum_{v\in \rV}  \lambda_{v} \langle f_{1}(v),f_{2}(v)\rangle_{v}.
\end{equation}
We define the two Hilbert spaces $\Omega^0_{\ell^2}(\F)=\{f\in\Omega^0(\F) : \|f\|^{2}<\infty\}$ endowed with the inner product~\eqref{eq:SPOmega0lambda}, and $\Omega^1_{\ell^2}(\F)=\{\omega\in\Omega^1(\F) : \|\omega\|^{2}<\infty\}$ endowed with \eqref{eq:SPOmega1}.

\begin{proposition} The operator $d_{h}$ sends $\Omega^0_{\ell^2}(\F)$ to $\Omega^1_{\ell^{2}}(\F)$. By restriction, it defines a bounded linear operator $D_{h}$, with operator norm $\|D_{h}\|\leq \sqrt{2}$.
\end{proposition}

\begin{proof} Let $f$ be an element of $\Omega^0_{\ell^2}(\F)$. For each edge $e$, the triangular inequality combined with the unitarity of the connection $h$ implies the inequalities 
\[\|d_{h}f(e)\|^{2}\leq (\|h_{e,\ul{e}}f(\ul e)\|+\|h_{e,\ol{e}}f(\ol e)\|)^{2}\leq (\|f(\ul e)\|+\|f(\ol e)\|)^{2}\leq 2 (\|f(\ul e)\|^{2}+\|f(\ol e)\|^{2}).\]
Thus,
\[\|d_{h}f\|^{2}=\frac{1}{2}\sum_{e\in \rE} x_{e} \|d_{h}f(e)\|^{2}\leq \sum_{e\in \rE} (x_{e} \|f(\ul e)\|^{2} + x_{e} \|f(\ol e)\|^{2})=2 \sum_{v\in \rV} \lambda_{v} \|f(v)\|^{2}.\]
This proves that $d_{h}f$ belongs to $\Omega^1_{\ell^{2}}(\F)$ and $\|d_{h}f\|^{2}\leq 2 \|f\|^{2}$.
\end{proof}

We denote by $D_{h}^{*}$ the adjoint of $D_{h}$, which is also a bounded linear operator, from $\Omega^1_{\ell^2}(\F)$ to~$\Omega^0_{\ell^{2}}(\F)$, with the same operator norm as $D_{h}$. It is given by the formula 
\begin{equation}\label{eq:d-starlambda}
(D_{h}^{*} \omega)(v)=\frac{1}{\lambda_{v}} (d_{h}^{*}\omega) (v)=\frac{1}{\lambda_{v}}\sum_{\substack{e\in \rE \\ \ol{e}=v}} x_{e} h_{v,e}\omega(e)\,.
\end{equation}

Let us define the bounded linear endomorphism 
\[L_{h}=D_{h}^{*}D_{h}\]
of $\Omega^{0}_{\ell^{2}}(\F)$.

Let us borrow and adapt some more notation from \cite{BLPS}. Let us denote respectively by~$\Omega^{0}_{0}(\F)$ and $\Omega^{1}_{0}(F)$ the spaces of $0$-forms and $1$-forms with finite support. 

\[\Omega^{0}_{0}(\F) \subset \Omega^{0}_{\ell^{2}}(\F) \subset \Omega^{0}(\F)\]

\begin{definition}
Let us define
\[\bigstar_{h}=\overline{\im D_{h}}=\overline{D_{h}(\Omega^{0}_{0}(\F))} \ \text{ and } \  \boxminus_{h}=\overline{\Omega^{1}_{0}(\F)\cap \ker D^{*}_{h}},\]
the closures being in the $\ell^{2}$ sense.
\end{definition}

Let us make a few simple remarks. Firstly, the space $\bigstar_{h}$ is equal to the closure of the range of~$D_{h}$. Indeed, the direct inclusion follows by definition, and reciprocally, any element of $\Omega^{0}_{\ell^{2}}(\F)$ is the limit of a sequence of elements with finite support, so that the range of $D_{h}$, hence its closure, is included in $\bigstar_{h}$.

Secondly, the space $\boxminus_{h}$ is included in the closed subspace $\ker D_{h}^{*}$, but this inclusion can be strict. Thirdly, the spaces $\bigstar_{h}$ and $\boxminus_{h}$ are orthogonal. However, their sum is not, in general, equal to $\Omega^1_{\ell^{2}}(\F)$. 

Let us define the space of harmonic functions with finite Dirichlet energy
\[\H_{h}=\{f\in \Omega^{0}(\F) : \lapp_h f=0 \text{ and } d_hf\in \Omega^{1}_{\ell^{2}}(\F)\}\]
and the space of their differentials
\[\D_h=d_h(\H_{h})\subseteq \Omega^1_{\ell^2}(\F).\]

\begin{theorem}\label{thm:splitting}
The following three orthogonal decompositions of $\Omega^{1}_{\ell^{2}}(\F)$ hold.
\[\Omega^1_{\ell^2}(\F)=\bigstar_{h} \oplus \ker D_{h}^{*}=\bigstar_{h}\oplus \D_h \oplus \boxminus_{h}=\big(d_h(\Omega^{0}(\F))\cap \Omega^{1}_{\ell^{2}}(\F)\big)\oplus \boxminus_{h}\,.\]
\end{theorem}

There is an analog of this proposition in \cite[Chapter 9, Eq. (9.6)]{Lyons-Peres} in the case of a trivial bundle of rank $1$, but the proof here is a little more intricate due to the presence of holonomy along loops.

\begin{proof} Let $\omega\in \Omega^1_{\ell^2}(\F)$ be orthogonal to $\bigstar_{h}$. Then $\omega$ is orthogonal to the range of $D_{h}$, so that it belongs to the kernel of $D_{h}^{*}$. This proves the first decomposition.

Let now $\omega\in \Omega^1_{\ell^2}(\F)$ be orthogonal to $\boxminus_{h}$. We will construct a function $f\in \Omega^{0}(\F)$ such that $d_{h}f=\omega$. This suffices to prove the third decomposition. If we assume moreover that $\omega$ is orthogonal to $\bigstar_{h}$, then we already argued that $D_{h}^{*}\omega=0$. This implies $d_{h}^{*}\omega=0$, so that $\Delta_{h}f=0$, and the second decomposition is proved.

This discussion shows that the proof reduces to the following statement: any $\omega\in \Omega^1_{\ell^2}(\F)$ orthogonal to $\boxminus_{h}$ is of the form $d_{h}f$ for some $f\in \Omega^{0}(\F)$. 

Let us give ourselves a $1$-form $\omega\in \Omega^{1}_{\ell^{2}}(\F)$ that is orthogonal to $\boxminus_h$. The existence of $f\in \Omega^{0}(\F)$ such that $d_{h}f=\omega$ can be proved separately for each part of the decomposition $\omega=\omega_{\rH}+\omega_{\rH^{\perp}}$ along the canonical decomposition $\Omega^{1}(\F)=\Omega^{1}(\rH)\oplus \Omega^{1}(\rH^{\perp})$. In other words, it suffices to treat the cases where the local system is either flat, or fully twisted.

The flat case is easy, as it is possible to choose a basis of each fibre in such a way that all holonomies are the identity. In these bases, the function $f$ is just the discrete integral of $\omega$ along any path from a fixed origin.

Let us turn to the fully twisted case. We are looking for a section $f\in \Omega^{0}(\F)$ and for all $\alpha\in \Omega^{1}_{0}(\F)$, the equality $(d_{h} f,\alpha)=(\omega,\alpha)$ holds. Since $\alpha$ has finite support, we have $(d_{h} f,\alpha)=(f,d_{h}^{*}\alpha)$, and the equality is equivalent to
\[(f,d_{h}^{*}\alpha)=(\omega,\alpha).\]

The assumption that $\omega$ is orthogonal to $\boxminus_{h}$ implies that the kernel of the linear form
$(\omega, \cdot) : \Omega^{1}_{0}(\F)\to \C$ contains the kernel of the linear map  $d_{h}^{*}:\Omega^{1}_{0}(\F)\to \Omega^{0}_{0}(\F)$.
Therefore, there exists a linear form $L$ on $\Omega^{0}_{0}(\F)$ such that $(\omega, \alpha)=L(d_{h}^{*} \alpha)$ for all $\alpha \in \Omega^{1}_{0}(\F)$. 

As any linear form on $\Omega^{0}_{0}(\F)$, the linear form $L$ is represented by an element $f$ of $\Omega^{0}(\F)$. For every $\theta\in \Omega^{1}_{0}(\F)$, we have
\[(f,d_{h}^{*}\alpha)=L(d^{*}\alpha)=(\omega,\alpha)\]
and the result is proved.
\end{proof}

\begin{lemma} Let $(\F,h)$ be a fully twisted local system over $\rG$. Then 
\[d_{h}^{*}(\Omega^{1}_{0}(\F))\supseteq \Omega^{0}_{0}(\F).\]
\end{lemma}

\begin{proof} Let $\varphi$ be an element of $\Omega^{0}_{0}(\F)$. Let $R\subset \rV$ be the support of $\varphi$. Let $v$ be an element of~$R$, and $\gamma_{1},\ldots,\gamma_{r}$ be $r$ loops given by Proposition \ref{prop:twistimage}. Let $S$ be a connected subgraph of $\rG$ containing $R$ and the loops $\gamma_{1},\ldots,\gamma_{r}$. Then by Proposition \ref{prop:twistimage} again, the restriction $(\F_{|S},h)$ is also fully twisted. Therefore, the codifferential of this local system over a finite graph is onto. Moreover, this codifferential coincides with the restriction to $\F_{S}$ of the codifferential of $\F$. Therefore, there exists a preimage of $\varphi$ by $d_{h}^{*}$, which moreover has finite support.
\end{proof}

\subsection{The space of harmonic functions with finite Dirichlet energy}\label{sec:dirichlet-harmonic}

Let us take a closer look at the space $\H_h$. Any horizontal section is harmonic and we would like to know when the inclusion $\ker d_{h}\subseteq \H_h$ is an equality.

This is true on a finite graph, but may not hold on an infinite graph. In any case, the question can be treated  separately on each summand of the canonical splitting of the local system into a flat and a fully twisted system (see Proposition \ref{prop:canosplit}).

In the classical case $h=1$, conditions are discussed in \cite[Chapters 9 and 10]{Lyons-Peres} for $\H_h$ to be equal to the constant functions. This is the case on all amenable graphs such as $\Z^d$. However, there are graphs for which $\H_h$ consists of more than just the constant functions: it is the case of all hyperbolic planar graphs, see \cite{Benjamini-Schramm} and references therein. For instance, on the bi-infinite binary tree, it is possible to work out a simple explicit example of an $f\in \Omega^0(\F)\setminus \Omega^{0}_{\ell^2}(\F)$ that is harmonic and such that $d_h f\in \Omega^{1}_{\ell^2}(\F)$.  

An easy observation is that if we consider a graph where $\D_1$ is non trivial, then any connection which admits horizontal sections will give rise to a non trivial $\D_h$.

In the fully twisted case, the question is to know whether the local system admits a non-zero harmonic function. We do not answer this question in full generality, but we will give a condition (satisfied by numerous examples) where this is the case. 

The first observation is that if a Poincaré inequality holds, then $\D_{h}=0$: harmonic functions must be in $\ell^{2}$, and are horizontal. 

In the next definition, we consider norms of elements of $\Omega^{0}(\F)$ and $\Omega^{1}(\F)$, allowing the value~$+\infty$.

\begin{definition}[Twisted Poincaré inequality] The local system $(\F,h)$ over the weighted graph $(\rG,\ul{x})$ is said to satisfy a twisted Poincaré inequality if there exists a a constant $C>0$ such that for all $f\in \Omega^0(\F)$ 
\[\lVert f \rVert^2 \le  C \lVert d_hf \rVert^2\, .\]
\end{definition}

A twisted Poincaré inequality can only hold on a fully twisted local system. Indeed, in a flat local system, there exists non-zero horizontal sections, which cannot satisfy the inequality for any $C$.

If a twisted Poincaré inequality is satisfied, then any $f\in \Omega^{0}(\F)$ such that $d_{h}f\in \Omega^{1}_{\ell^{2}}(\F)$ belongs to $\Omega^{0}_{\ell^{2}}$. Therefore, $d_{h}f$ is equal to $D_{h}f$. Moreover, the assumption implies that $D_h$ is injective and has closed range. In particular, its inverse is bounded on its range. 

\begin{proposition} \label{prop:HD=0}
If the local system $(\F,h)$ is fully twisted and satisfies a twisted Poincaré inequality over $(\rG,\ul{x})$, then $\H_{h}=\{0\}$.
\end{proposition}

\begin{proof} Let us consider $f\in \H_{h}$. By definition, $d_h f\in \Omega^1_{\ell^2}(\F)$, so that the Poincaré inequality implies that $f\in \Omega^{0}_{\ell^{2}}(\F)$. Now, the harmonicity of $f$ implies
\[\|D_{h}f\|^{2}=(f,L_{h}f)=0,\]
so that $D_{h}f=0$, and therefore, by the Poincaré inequality again, $f=0$.
\end{proof}

We will now give a condition on the holonomy of the connection $h$ which implies that $(\F,h)$ satisfies a twisted Poincaré inequality. Let us introduce a notation. For a loop $\gamma=(e_{1},\ldots,e_{k})$ and an edge $e\in \rE$, we denote by
\[m(e,\gamma)=|\{i\in \{1,\ldots,k\}: e_{i}=e\}\]
the number of times where $\gamma$ crosses the oriented edge $e$. We also define
\[m([e],\gamma)=m(e,\gamma)+m(e^{-1},\gamma),\]
the number of times where $\gamma$ crosses the geometric edge $e$.

Moreover, we define the length of $\gamma$ as
\[\ell(\gamma)=\sum_{i=1}^{k}\frac{1}{x_{e_{i}}}=\sum_{e\in \gE}\frac{m([e],\gamma)}{x_{e}}.\]

\begin{definition}[Uniform holonomic spectral gap for local systems fully twisted along a loop]\label{def:geom-reg}
We say that the local system $(\F,h)$ over the weighted graph $(\rG,\ul{x})$ admits a {\em uniform holonomic spectral gap} if it is possible to choose, for every vertex $v$, a loop~$\gamma_{v}$ based at $v$ such that $\id_{\F_v}-\hol_h(\gamma_{v})$ is invertible, in such a way that the following conditions hold. 
\begin{enumerate}[1.]
\item There exists a constant $C_{1}>0$ such that for all $v\in \rV$, 
\[ \big\| \big(\id_{\F_v}-\hol_h(\gamma_{v})\big)^{-1}\big\|^{2} \leq \frac{C_{1}}{\ell(\gamma_{v})}.\]
\item There exists a constant $C_{2}>0$ such that for all $e\in \rE$,
\[\sum_{v\in \rV} m([e],\gamma_{v}) \lambda_{v}\leq C_{2}.\]
\end{enumerate}
\end{definition}

This assumption is for example satisfied on $\Z^d$ with conductances uniformly bounded above and away from zero, and periodic $h$ such that, along the boundary of any plaquette, the holonomy does not have ~$1$ as an eigenvalue.

We will propose a weaker condition in a moment, that also implies a Poincaré inequality, but can be verified by fully twisted local systems that are not fully twisted along a loop. 

\begin{proposition}\label{prop:bound}
Assume that the local system $(\F, h)$ is fully twisted along a loop and satisfies the assumption of uniform holonomic spectral gap over $(\rG,\ul{x})$.
Then $(\F,h)$ satisfies a twisted Poincaré inequality over $(\rG,\ul{x})$ with constant $C_{1}C_{2}$.
\end{proposition}

\begin{proof} For each vertex $v$, let $\gamma_{v}$ be a loop given by the assumption. By Corollary \ref{cor:bound},
\[\|f\|^{2}=\sum_{v\in \rV} \lambda_{v} \|f(v)\|^{2}\leq \sum_{v\in \rV} \lambda_{v} 
\big\| \big(\id_{\F_v}-\hol_h(\gamma_{v})\big)^{-1}\big\|^{2}
\Big(\sum_{e\in \gE} m([e],\gamma_{v}) \|d_{h}f(e)\|\Big)^{2}.\]
Let us apply the Cauchy--Schwarz inequality to the last sum, weighting each term by $\sqrt{m([e],\gamma_{v}) x_{e}}$. We find 
\begin{align*}
\|f\|^{2}&\leq \sum_{v\in \rV} \lambda_{v} \big\| \big(\id_{\F_v}-\hol_h(\gamma_{v})\big)^{-1}\big\|^{2} \ell(\gamma_{v}) \sum_{e\in \gE} m([e],\gamma_{v}) x_{e}\|d_{h}f(e)\|^{2}\\
&\leq C_{1}  \sum_{v\in \rV} \lambda_{v} \sum_{e\in \gE} m([e],\gamma_{v}) x_{e} \|d_{h}f(e)\|^{2}\\
&\leq C_{1}  \sum_{e\in \gE} \bigg(\sum_{v\in \rV}m([e],\gamma_{v})\lambda_{v}\bigg)  x_{e} \|d_{h}f(e)\|^{2}\\
&\leq C_{1} C_{2} \|d_{h}f\|^{2},
\end{align*}
which proves the result.
\end{proof}

\begin{definition}[Uniform holonomic spectral gap for fully twisted local systems]\label{def:geom-reg2}
We say that the local system $(\F,h)$ over the weighted graph $(\rG,\ul{x})$ admits a {\em uniform holonomic spectral gap} if it is possible to choose, for every vertex $v$, some loops $\gamma_{v,1},\ldots,\gamma_{v,r_{v}}$ based at $v$ such that 
\[\sum_{i=1}^{r} \im (\hol_h(\gamma_{v,i})-\id_{\F_v})=\F_{v},\]
in such a way that the following conditions hold.

For all $v\in \rV$ and $i\in \{1,\ldots,r_{v}\}$, we set $R_{v,i}= \im (\hol_h(\gamma_{v,i})-\id_{\F_v})$.

\begin{enumerate}[1.]
\item There exists a constant $C_{1}>0$ such that for all $v\in \rV$, 
\[\big\|\big((\id_{\F_v}-\hol_h(\gamma_{v,i}))_{|R_{v,i}}\big)^{-1}\big\|^{2} \leq \frac{C_{1}}{\ell(\gamma_{v,i})}. 
\]
\item There exists a constant $C_{2}>0$ such that for all $e\in \rE$,
\[\sum_{\substack{v\in \rV \\ i=1,\ldots,r_{v}}} m([e],\gamma_{v,i}) \lambda_{v}\leq C_{2}.\]
\item There exists a constant $C_{3}>0$ such that for all $v\in \rV$,
\[ \big\|\big(\proj{*}_{v}\proj{}_{v} \big)^{-1}\big\|\ \leq C_{3}.\]
\end{enumerate}
\end{definition}

This assumption is for example satisfied on $\Z^d$ with conductances uniformly bounded above and away from zero, and a periodic connection $h$ taking its values in $\SO(2d+1)$ such that, along the boundary of any plaquette, the holonomy is alternatively one of two rotations which have $1$ as simple eigenvalue, with different eigenspaces.

\begin{proposition}\label{prop:bound2}
Assume that the fully twisted local system $(\F, h)$ satisfies the assumption of uniform holonomic spectral gap over $(\rG,\ul{x})$. Then $(\F,h)$ satisfies a twisted Poincaré inequality over $(\rG,\ul{x})$ with constant $C_{1}C_{2}C_{3}$.
\end{proposition}

\begin{proof} The proof is almost the same as that of Proposition \ref{prop:bound}, with the small difference that it relies on Corollary \ref{cor:bound2} instead of Corollary \ref{cor:bound}.
\end{proof}

\subsection{Recurrent weighted graphs}

In this section and the next, we will consider an infinite connected graph $\rG=(\rV,\rE)$ endowed with Euclidean weights $\ul{x}$. These weights allow one to define a nearest neighbor random walk on $\rG$, the probability of crossing of an edge $e$ being proportional to $x_e$. We say that the weighted graph $(\rG,\ul{x})$ is recurrent or transient according to the recurrence or transience of this random walk.

\begin{theorem} \label{thm:dh-rec}
For any local system $(\F,h)$ over a recurrent weighted graph $(\rG,\ul{x})$, the equality $\D_h=\{0\}$ holds.
\end{theorem}

In the proof of this theorem and of the next, we will consider the trivial real rank $1$ local system over $(\rG,\ul{x})$, with holonomies equal to the identity. The sections of this vector bundle are simply real-valued functions on $\rV$, and the Laplacian of this local system is the usual Laplacian on scalar functions over $\rG$ determines by the weights $\ul{x}$. We denote it by $\Delta_1$. To avoid any confusion, we write {\em $1$-harmonic} when we mean {\em harmonic with respect to $\Delta_1$}. We keep the word {\em harmonic} for sections of $\F$, and this is relative to the operator $\Delta_h$.

\begin{proof} The equality $\D_h=\{0\}$ is equivalent to the fact that any function in $\H_h$ is horizontal, and it is under this form that we prove the result.

Let us consider a harmonic function $f\in \H_{h}$. For every vertex $v$, define $g(v)=\|f(v)\|_{v}$. Since~$f$ is harmonic, we have
\[g(v)=\bigg\| \sum_{i=1}^{n} \frac{x_{e_i}}{\lambda_v} h_{e_{i}}^{-1} f(\ol{e_{i}}) \bigg\| \leq \sum_{i=1}^{n} \frac{x_{e_i}}{\lambda_v} \|h_{e_{i}}^{-1} f(\ol{e_{i}})\|=\sum_{i=1}^{n} \frac{x_{e_i}}{\lambda_v} g(\ol{e_{i}}),\]
so that the scalar function $g$ is sub-$1$-harmonic on $(\rG,\ul{x})$. Moreover, the fact that $f$ has finite Dirichlet energy and the triangle inequality applied on each edge under the form
\[\big|\|f(\ol{e})\|_{\ol{e}}-\|f(\ul{e})\|_{\ul{e}}\big| \leq \|d_{h}f(e)\|_{e}\]
imply that $g$ also has finite Dirichlet energy. 

We now use the fact, proved in \cite[Thm. 3.34]{Soardi}, that on a recurrent weighted graph, any sub-$1$-harmonic real-valued function  with finite Dirichlet energy is constant. Therefore, $f$ has constant norm. 

Thus, at each vertex $v$, with outgoing edges $e_{1},\ldots,e_{n}$, the equality
\[f(v)=\sum_{i=1}^{n} \frac{x_{e_i}}{\lambda_v} h_{e_{i}}^{-1} f(\ol{e_{i}})\]
expresses $f(v)$ as a convex combination of vectors of the same norm as $f(v)$. The strict convexity of the spheres of $\F_v$ implies that all these vectors must be equal. Therefore, $f$ is horizontal. 
\end{proof}

\subsection{Transient weighted graphs}

We now give a result in the case where the weighted graph~$(\rG,\ul{x})$ is transient. 
Given a subset $B$ of $\rE$, we denote by $\rV(B)$ the set of endpoints of an edge of~$B$, and by $\rG\setminus B$ the graph $(\rV\setminus \rV(B), \rE\setminus B)$.
We also denote by $X=(X_{n})_{n\geq 0}$ the random walk on $(\rG,\ul{x})$, and by $\P_{u}$ the distribution of this walk started at a vertex $u$.

\begin{theorem} \label{thm:dh-trans}
Let $(\F,h)$ be a local system over a transient weighted graph $(\rG,\ul{x})$. Assume that the following properties hold.
\begin{enumerate}[\indent 1.]
\item (Liouville property) Any bounded scalar $1$-harmonic function on $(\rG,\ul{x})$ is constant.
\item (Non flatness) $(\F,h)$ is not flat over $\rG$.
\item (Almost flatness) There exists a finite subset $B$ of $\rE$ such that $\rG\setminus B$ is connected and the restriction of $(\F,h)$ to $\rG\setminus B$ is flat.
\end{enumerate}
Then $\D_h\neq \{0\}$. Moreover, if $(\F,h)$ is fully twisted, then $\dim \D_h$ is at least equal to the rank of~$\F$.
\end{theorem}

During the proof, we will use two elementary properties of the transient weighted graph $(\rG,\ul{x})$, one about the random walk itself, and the other about the usual $1$-Laplacian on scalar functions over $\rG$.

\begin{lemma} \label{lem:hitprob0}
Let $W$ be a finite subset of $\rV$. Let $T=\inf\{n\geq 0 : X_n\in W\}$ be the first hitting time of~$W$. For all $u\in \rV$, set $\eta(u)=\P_u(T<\infty)$. Then for all $x\in \rV$,
\[\lim_{n\to \infty} \eta(X_{n}) = 0\ \  \ \P_{x} \text{-a.s.}\]
\end{lemma}

\begin{proof} Let us fix a vertex $x\in \rV$.  The function $\eta$ is non-negative, super-$1$-harmonic on $W$ and harmonic outside $W$, hence super-$1$-harmonic everywhere. Thus, the sequence of random variables $(\eta(X_{n}))_{n\geq 0}$ is under $\P_{x}$ a non-negative super-martingale, and converges $\P_{x}$-almost surely. 

For all $n\geq 0$, using traditional notation, the Markov property of the random walk at time $n$ yields
\[\E_{x}[\eta(X_{n})]=\E_{x}\big[\E_{X_{n}}[\1_{\{T<\infty\}}]\big]=\E_{x}\big[\E_{x}[\1_{\{T<\infty\}}\circ \theta_{n}|\mathscr F_{n}]\big]=\P_{x}(\exists k\geq n : X_{k}\in W).\]
As $n$ tends to infinity, the sequence of events on the right-hand side decreases to the event where $X$ visits~$W$ infinitely often, which, since the random walk is transient and $W$ is finite, has probability $0$. An application of Fatou's lemma now gives
\[\E_{x}\big[\text{a.s.}\!\lim_{n\to \infty} \eta(X_{n})\big]\leq \liminf_{n\to \infty} \E_{x}[\eta(X_{n})]=0,\]
from which the result follows.
\end{proof}

\begin{lemma} \label{lem:finDir}
Let $W$ be a finite subset of $\rV$. Let $a$ be a scalar function on $\rG$. Assume that $a$ is bounded and satisfies $\Delta_{1}a(u)=0$ for all $u\in \rV\setminus W$. Then $a$ has finite Dirichlet energy.
\end{lemma}

\begin{proof} Since $(\rG,\ul{x})$ is transient, there exists, for every vertex $u$, a bounded scalar function $g_{u}$ on $\rG$ with finite Dirichlet energy and such that $\Delta_{1}g_{u}=\1_{u}$ (see for example \cite[Thm. 2.11 and Prop. 2.12]{Lyons-Peres}). Therefore, since the $1$-Laplacian of $a$ vanishes outside $W$, there exists a linear combination of $a$ and the functions $\{g_{u}, u\in W\}$ of the form $a(\cdot)+\sum_{u\in W} c_{u}g_{u}(\cdot)$ that is $1$-harmonic over $\rG$. This linear combination is also bounded, and by our assumption that~$(\rG,\ul{x})$ has the Liouville property, it is therefore constant. It follows that $a$ is equal to an affine combination of the finite collection of functions $\{g_{u}, u\in W\}$, and has therefore finite Dirichlet energy.
\end{proof}

\begin{proof}[Proof of Theorem \ref{thm:dh-trans}] Let us fix an arbitrary vertex $v$ of $\rV\setminus B$. Let $H$ denote the set of horizontal sections of $(\F,h)$ over~$\rG\setminus B$.  Since $\rG\setminus B$ is connected, the evaluation map $\ev_{v}:H\to \F_{v}$ is an isomorphism, so that $\dim H=N$.

Since $(\F,h)$ is not flat, there exists a loop $\gamma$ based at $v$ such that $\hol_{h}(\gamma)\neq \id_{\F_{v}}$. Therefore, there exists a vector $\xi\in \F_{v}$ such that $\hol_{h}(\gamma)(\xi)\neq \xi$. Let us fix such a vector $\xi$, that we identify with a horizontal section of $\F$ over $\rG\setminus B$.

Let us note that if $(\F,h)$ is fully twisted, then for every $\xi\in \F_{v}\setminus \{0\}$, there exists a loop $\gamma$ based at $v$ such that $\hol_{h}(\gamma)(\xi)\neq \xi$. Therefore, in this case, $\xi$ can be chosen to be any non-zero vector of $\F_{v}$, with an appropriate choice of $\gamma$.
 
Now that $\xi$ and $\gamma$ are fixed, let us define a harmonic section $f\in \Omega^{0}(\F)$. For this, let us define the random time $T=\inf\{n\geq 0 : \forall m\geq n, X_{m} \in \rV\setminus \rV(B)\}$. The transience of $(\rG,\ul{x})$ ensures that for all vertex $u\in \rV$, this random time is finite $\P_{u}$-almost surely. 

For every vertex $u\in \rV$, the random sequence $\big(\hol_{h}(X_{|[0,n]}^{-1})(\xi(X_{n}))\big)_{n\geq T}$ of elements of $\F_{u}$ is well defined $\P_{u}$-almost surely, and constant. We denote by  $\hol_{h}(X^{-1})(\xi)$ its limit, and we set
\[f(u)=\E_{u}\big[\hol_{h}(X^{-1})(\xi)\big].\]
We claim that $d_{h}f$ is a non-zero element of $\D_{h}$. To prove this, we will prove that $f$ is harmonic, not horizontal, not in $\Omega^{0}_{\ell^{2}}(\F)$, and with finite Dirichlet energy. 

$\sbullet$ The fact that $f$ is harmonic follows immediately from the Markov property at time $1$ of the random walk $X$. 

$\sbullet$ Let us prove that $f$ is not horizontal. For this, let us observe that Lemma \ref{lem:hitprob0} applied to $W=\rV(B)$ guarantees the existence of a sequence $(u_n)_{n\geq 0}$ of pairwise distinct vertices of $\rG$ such that $\eta(u_n)$ tends to $0$ as $n$ tends to infinity. In particular, the probability for the random walk started from $u_n$ of ever crossing any edge of $B$ tends to $0$, and therefore $\P_{u_n}(\hol_{h}(X^{-1})(\xi)=\xi)$ tends to $1$, as $n$ tends to infinity. 

Then, firstly, by dominated convergence, $\|f(u_n)-\xi\|$ tends to $0$ as $n$ tends to infinity. Since the sequence $(u_n)_{n\geq 0}$ eventually exits the finite set $\rV(B)$, we conclude that if $f$ were horizontal, it would have to coincide with $\xi$ on $\rG\setminus B$.

Secondly, for $n$ large enough, we have $u_n\notin \rV(B)$ and $\eta(u_n)<1$, so that the probability for the random walk started from $u_n$ of never crossing any edge of $B$ is positive. Let us use now the assumption that $\rG\setminus B$ is connected to find a path $p$ from $u$ to $v$ in $\rG\setminus B$. Then on the one hand, $\hol_{h}(p\gamma p^{-1})(\xi)\neq \xi$. On the other hand, the probability that the random walk issued from $u$ starts by going exactly along~$p\gamma p^{-1}$, and then does never cross any edge of $B$, is positive. Therefore, the distribution of $\hol_{h}(X^{-1})(\xi)$ under~$\P_{u}$ is a probability measure on the sphere of~$\F_{u}$ of radius $\|\xi\|$ that is not the Dirac mass at $\xi$. By the uniform convexity of this sphere, the average of this probability distribution has norm strictly smaller than $\xi$, and is therefore not equal to $\xi$. It follows that $f(u)\neq \xi$, and we conclude that $f$ is not horizontal.

$\sbullet$ The fact that $f$ does not belong to $\Omega^{0}_{\ell^{2}}(\F)$ follows immediately from the fact that $\xi\neq 0$ and $\|f-\xi\|$ tends to $0$ at infinity.

$\sbullet$ Let us finally prove that $f$ has finite Dirichlet energy, that is, that $d_{h}f$ belongs to $\Omega^{1}_{\ell^{2}}(\F)$. For this, let us construct $N$ sections $s_{1},\ldots,s_{N}$ of $\F$ that are horizontal on $\rG\setminus B$, and such that $(s_{1}(u),\ldots,s_{N}(u))$ is an orthonormal basis of~$\F_{u}$ for each vertex $u\in \rV$, 

With the help of these sections, let us write $f=f_{1}s_{1}+\ldots+f_{N}s_{N}$, where $f_{1},\ldots,f_{N}$ are $N$ scalar functions on $\rG$. Since $s_{1},\ldots,s_{N}$ are horizontal on the complement of a finite subset of~$\rE$, the function $f$ has finite Dirichlet energy with respect to the differential $d_{h}$ if and only if the functions $f_{1},\ldots,f_{N}$ have finite Dirichlet energy with respect to the usual untwisted differential on the space of real-valued functions on $\rG$.

Let us now observe that since the sections $s_{1},\ldots,s_{N}$ are horizontal on $\rG\setminus B$, we have there the identity $\Delta_{h}f=(\Delta_{1} f_{1})s_{1}+\ldots+(\Delta_{1}f_{N})s_{N}$. Therefore, each of the scalar functions $f_{1},\ldots,f_{N}$ is $1$-harmonic on~$\rG$, except at a finite subset of vertices. It follows from Lemma \ref{lem:finDir} that each of them has finite Dirichlet energy, and therefore that $f$ itself has finite Dirichlet energy.
\end{proof}

\subsection{The example of $\Z^d$}

Let us consider the example where $\rG$ is the hypercubic lattice $\Z^d$ for some integer $d\geq 1$, with the usual vertices and edges, and with constant conductances. 

In this case, it is a classical fact that any bounded harmonic function is constant (see for instance \cite[Thm. 9.24]{Lyons-Peres}). Since the space of differentials of bounded Dirichlet harmonic functions is dense in the space $\D_1$ of differentials of Dirichlet harmonic functions (see \cite[Ex. 9.44]{Lyons-Peres}), this space $\D_1$ is the null space. In particular, the free and wired uniform spanning forests are equal. However, Theorems \ref{thm:dh-rec} and \ref{thm:dh-trans} have the following consequence.

\begin{proposition} \label{prop:Zd}
Let $(\F,h)$ be a local system over the graph $\Z^d$, with constant Euclidean weights.
\begin{enumerate}[\indent 1.]
    \item If $d\leq 2$, then $\D_h=\{0\}$.
    \item If $d\geq 3$ and $(\F,h)$ is flat outside a finite box of $\Z^{d}$ without being flat over $\Z^{d}$, then $\D_{h}\neq \{0\}$.
\end{enumerate}
\end{proposition}

A very simple case where the second assertion applies is that where the vector bundle is trivial, of rank $1$, and the connection is identically equal to $1$, except along one single edge, where it is equal to $-1$. Of course, in that case, the measure is not translation invariant, unlike the uniform spanning forest.

%%%%%%%%%%%%%%%%%%%%%%%%%%%%%%%%%%%%%%%%

\section{Wired and free DLSF}\label{sec:wired-free}

We are now ready to define two probability measures for determinantal linear spanning forests on infinite graphs. 
The justification for their names will come in the next paragraph about finite approximations.

\subsection{Abstract definition and stochastic comparison}

\begin{definition}The \new{wired determinantal linear spanning forest (WDLSF)}, resp. the \new{free determinantal linear spanning forest (FDLSF)}, is the determinantal probability measure on $\Omega^{1}(\F)$ associated to the splitting~$\sigma_{1}$ and to the subspace $\bigstar_{h}$, resp. to the subspace $\boxminus_{h}^{\perp}=\bigstar_{h}\oplus \D_h$.
\end{definition}

These two measures are related by the following stochastic inequality which is a specialization of a general result about determinantal probability measures~\cite{KL3}. 

\begin{proposition}\label{prop:wsf-fsf}
The wired determinantal linear spanning forest is stochastically dominated by the free determinantal linear spanning forest.
In particular, the trace of the wired DLSF is stochastically dominated by the trace of the free DLSF. Moreover, these measures are equal if and only if $\D_h=0$.
\end{proposition}
\begin{proof}
The result is a corollary of the stochastic domination property for determinantal probability measures~\cite[Proposition 4.15]{KL3}.
\end{proof}

In particular, combining Proposition~\ref{prop:wsf-fsf} and Proposition~\ref{prop:HD=0}, under the assumption of uniform holonomic spectral gap (Definitions~\ref{def:geom-reg} and \ref{def:geom-reg2}), the free and wired determinantal linear spanning forest are equal. 

\begin{corollary}
Assume that $(\F,h)$ satisfies the uniform holonomic spectral gap on $(\rG,\ul{x})$, or more generally a Poincaré inequality. Then WDLSF and FDLSF are equal.
\end{corollary}

By Theorem \ref{thm:dh-rec}, on any recurrent graph, we have $\WDLSF=\FDLSF$. By Proposition \ref{prop:Zd} however, for all $d\ge 3$, there exists infinitely many local systems on $\Z^d$ for which $\WDLSF\neq\FDLSF$.

When these measures coincide, we simply call DLSF the corresponding measure.

\subsection{Approximation by determinantal linear spanning forests on finite graphs}\label{subsubsec:approx}

Let us consider a sequence $\left(\rG_n=(\rV_n,\rE_n)\right)_{n\ge 1}$ of increasing subgraphs which exhaust the countably infinite graph $\rG$. Thus, for all $n\ge 1$, we have $\rE_n\subset \rE_{n+1}\subset \rE$ and $\rV_n\subset\rV_{n+1}\subset \rV$.

The following diagram sums up the situation.

\[\xymatrix{
\cdots \ar[r] & \Omega^{0}(\F_{|\rG_{n-1}}) \ar[r]^{\iota_{n}} \ar@<0.75ex>[d]^{d_{n-1}} \ar@{-->}[dr] & \Omega^{0}(\F_{|\rG_{n}}) \ar[r]^{\iota_{n+1}} \ar@<0.75ex>[d]^{d_{n}} \ar@{-->}[dr] & \Omega^{0}(\F_{|\rG_{n+1}}) \ar[r] \ar@<0.75ex>[d]^{d_{n+1}} & \cdots \\ 
\cdots  \ar[r] & \Omega^{1}(\F_{|\rG_{n-1}}) \ar[r]^{\epsilon_{n}} \ar@<0.75ex>[u]^{d^{*}_{n-1}} & \Omega^{1}(\F_{|\rG_{n}}) \ar[r]^{\epsilon_{n+1}} \ar@<0.75ex>[u]^{d^{*}_{n}}  & \Omega^{1}(\F_{|\rG_{n+1}}) \ar[r] \ar@<0.75ex>[u]^{d^{*}_{n+1}}  & \cdots
}\]

Define now for each integer $n$, the two spaces $\bigstar_n$ and $\boxminus_n$ as follows (on the diagram it corresponds to following the oblique dashed arrow, from above or below, respectively). 

\begin{definition}For all $n\ge 1$, let 
$\bigstar_n=d_{n}(\iota_n(\Omega^0(\F_{|\rG_{n-1}})))$ and $\boxminus_n=\varepsilon_n(\ker(d_{n-1}^*))$.
\end{definition}

These spaces are orthogonal and both subspaces of $\Omega^1(\F_{|\rG_{n}})$. However they do not generally span $\Omega^1(\F_{|\rG_{n}})$. This corresponds to the fact that from the point of view of the graph, one space is defined using a `wired' construction, and the other a `free' one, see Figure~\ref{fig:wired-free}. They differ from one another by a space of sections supported on the boundary, which may or may not vanish when we pass to infinity. 

\begin{figure}[!ht]
\centering
\includegraphics{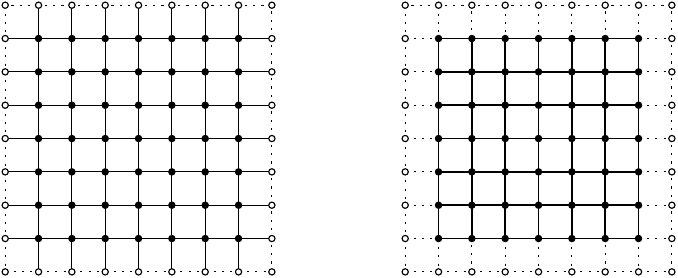}
\caption{\small Exhaustion of the square grid by concentric square boxes $(\rG_n)_{n\ge 1}$.
On both sides: the graph $\rG_n$ with the vertices of $\rV_{n-1}\subset\rV_{n}$ in black. 
On the left : wired boundary conditions. The space $\bigstar_n\subset \Omega^1(\F_{|\rG_{n}})$ is the space of differentials of sections vanishing on the white vertices.  On the right: free boundary conditions. The space $\boxminus_n\subset \Omega^1(\F_{|\rG_{n}})$ is the space of differentials of sections of $\Omega^{0}(\F_{|\rG_{n-1}})$, extended by $0$ on the edges of $\rG_{n}\setminus \rG_{n-1}$.}\label{fig:wired-free}
\end{figure}

The nice property is that these subspaces are increasing for inclusion and that the closure of their union are the spaces we defined earlier. 

\begin{proposition}
Consider an increasing sequence of subgraphs $(\rG_n)_{n\ge 1}$ which exhaust the infinite graph $\rG$. For all $n\ge 1$, we have
$\bigstar_n\subset\bigstar_{n+1}$ and $\boxminus_n\subset\boxminus_{n+1}$. Moreover, 
\[\bigstar_{h}=\overline{\bigcup_{n\geq 1}\bigstar_n} \  {\text and}\  \boxminus_{h}=\overline{\bigcup_{n\geq 1} \boxminus_n},\]
where the closures are with respect to the $\ell^{2}$-norms.
\end{proposition}
\begin{proof}
The first assertion follows easily from the definition by looking at the diagram. The second follows easily from the definition of these infinite dimensional spaces.
\end{proof}

It now follows from the results on determinantal probability measures on countable splittings~\cite[Section 4.9]{KL3} that we have the following convergence statement. 

\begin{theorem}\label{thm:infinite-volume}
Given any increasing exhaustion of our graph, the WDLSF is the weak limit of the determinantal linear spanning forest on the wired subgraphs, and the FDLSF is the weak limit of the determinantal linear spanning forest on the free subgraphs.
\end{theorem}

\subsection{Poincaré inequality and exponential decay of correlations}

Let us assume that $(\F,h)$ satisfies a Poincaré inequality on $(\rG,\ul{x})$, that is, that there exists a constant $C>0$ such that for all $f\in \Omega^0_{\ell^2}(\F)$, we have
\[(f, L_h f)=\lVert D_h f\rVert^2 \ge C \lVert f \rVert^2 \,.\]

Define $\Lambda=m^{\Omega^0(\F)}_{\ul{\lambda}}$ to be the operator acting by scaling by $\lambda_v$ on each fiber $\F_v$. 
Since $L_h=\id-P_h$, taking $C$ as above (which we may assume to be smaller than $1$), this implies that~$P_h$ has operator norm smaller than $1-C$ which is strictly smaller than $1$, and we can thus invert $L_h$ to obtain a Green section 
\[G_h \Lambda =L_h^{-1}=\sum_{n\ge 0} P_h^n,\]
where 
\[(P_h^n)_{u,v}=\sum_{\substack{p\colon u\to v\\ \vert p\vert=n}} \P(p) \hol_h(p)^{-1}.\]
See \cite{KL1} for a detail of this computation.

In particular, for any two vertices $u$ and $v$ at distance $m$, the coefficient $(u,v)$ of $P_h^n$ is zero if $n<m$. 
Hence 
\[\vert (G_h)_{u,v} \vert \lambda_v \le \sum_{n\ge m} \vert (P_h^n)_{u,v}\vert \le \sum_{n \ge m} \lVert P_h \rVert^n = \frac{\lVert P_h \rVert^m}{1-\lVert P_h \rVert}\,.\]

Hence the Green section at a pair of vertices $(u,v)$ decays exponentially with the graph-distance between $u$ and $v$.\footnote{Here and below, we use the graph distance, and not the weighted distance. Note that if the weights $\ul{x}$ are uniformly bounded above and away from zero, the graph-distance and the weighted distance are Lipschitz-equivalent.}
In particular, the correlations of the Gaussian Free Vector Field (see Section \ref{sec:GFVF}) decay at least exponentially with the graph-distance. 

Since $\proj{\bigstar_{\ell^2}}=D_h G_h D^*_h$, the correlation kernel of the DLSF also decays exponentially with the graph-distance. Let us record this result in the following proposition.

\begin{proposition}
If $(\F,h)$ satisfies a Poincaré inequality over $(\rG,\ul{x})$, then the WDLSF and FDLSF coincide, and the edge-edge correlation operators for the DLSF decay at least as fast as an exponential decay with the graph-distance. 
\end{proposition}

In view of Propositions \ref{prop:bound} and \ref{prop:bound2}, the conclusion of this proposition holds in particular when the local system satisfies a uniform holonomic spectral gap on $(\rG,\ul{x})$ in the sense of Definitions~\ref{def:geom-reg} and~\ref{def:geom-reg2}.

H\'elo\"{i}se Constantin already gave in~\cite{Constantin} another proof of this exponential decay, in the case of $\Z^d$ with a regular periodic connection, using an explicit Fourier transform of the correlation kernel.
Such an exponential decay of the Green function was also proven by \cite{Balaban} in the context of lattice gauge theory, under some assumptions. 

\begin{comment}
\section{Concluding remarks}

\begin{enumerate}[\indent 1.]
    \item Can one find a sampling algorithm in the spirit of Wilson's algorithm \cite{Wilson} or its bicircular-adapted modified version \cite{Kassel-Kenyon}? Can it provide a probabilistic interpretation for Theorem~\ref{thm:superposition-intro}?
    \item We introduced a polymatroidal stratification of the moduli space of local systems in Section \ref{sec:stratif}. Can we refine the limit theorem of Section~\ref{sec:infinitesimal-connections} to compactify the space of DLSF as a function of~$h$ at other limit points (other discontinuity points of the lower-semicontinuous rank function) than the trivial connection $h_0$ which we studied?
    \item Can one gather more geometric information on the loop-erased random walk, uniform spanning tree, or uniform CRSF, from the study of the partition function of \eqref{eq:law-crsf-n}?
    \item Peut-on étudier la fonction $w_h$ comme fonction sur les points entiers du polytope des bases ?
    \item Trouver un théorème de structure. 
    \item L'holonomie de fuite : je pense quand même que ce que j'avais écrit dans la version précédente de l'article fait un lien (avec le Mouton du Werner).
\end{enumerate}
\end{comment}

%%%%%%%%%%%%%%%%%%%%%%%%%%%%%%%%%%%%%%%%
%%%%%%%%%%%%%%%%%%%%%%%%%%%%%%%%%%%%%%%%

%%%%%%%%%%%%%%%%
%%%   BIBLIOGRAPHY   %%%
%%%%%%%%%%%%%%%%

%%%%%%%%
\newcommand{\arXiv}[1]{\href{http://arxiv.org/abs/#1}{arXiv:#1}}
%%%%%%%%
%\newcommand{\zbMath}{\href{http://www.zbmath.org/?q=#1}{ZB#1}}
%%%%%%%%%
\bibliographystyle{hmralphaabbrv}
\bibliography{dlsf}

%%%%%%%%%%%

\vfill 

\end{document}